\documentclass[final,1p,times]{elsarticle}

\usepackage{amssymb}
\usepackage{amsmath}
\usepackage{amsthm}

\numberwithin{equation}{section}
\newtheorem{theorem}{Theorem}[section]
\newtheorem{remark}[theorem]{Remark}
\newtheorem{lemma}[theorem]{Lemma}
\newtheorem{cor}[theorem]{Corollary}

\newtheorem{definition}[theorem]{Definition}

\newcommand{\nat}{\mathbb{N}}
\newcommand{\rzecz}{\mathbb{R}}

\newcommand{\eps}{\varepsilon }

\newcommand{\rd}{{\rzecz }^{d}}

\newcommand{\tr}{\text{\rm Tr}}

\newcommand{\curl}{\text{\rm curl}\, }
\newcommand{\diver}{\text{\rm div}}
\newcommand{\supp}{\text{\rm supp}}

\newcommand{\acal}{\mathcal{A}}
\newcommand{\bcal}{\mathcal{B}}
\newcommand{\ccal}{\mathcal{C}}
\newcommand{\dcal}{\mathcal{D}}

\newcommand{\fcal}{\mathcal{F}}

\newcommand{\kcal}{\mathcal{K}}
\newcommand{\lcal}{\mathcal{L}}
\newcommand{\mcal}{\mathcal{M}}

\newcommand{\ocal}{\mathcal{O}}
\newcommand{\pcal}{\mathcal{P}}
\newcommand{\rcal}{\mathcal{R}}
\newcommand{\scal}{\mathcal{S}}
\newcommand{\tcal}{\mathcal{T}}
\newcommand{\vcal}{\mathcal{V}}
\newcommand{\xcal}{\mathcal{X}}
\newcommand{\ycal}{\mathcal{Y}}
\newcommand{\zcal}{\mathcal{Z}}

\newcommand{\cmath}{\mathbb{C}}
\newcommand{\dmath}{\mathbb{D}}
\newcommand{\fmath}{\mathbb{F}}
\newcommand{\hmath}{\mathbb{H}}

\newcommand{\umath}{\mathbb{U}}
\newcommand{\vmath}{\mathbb{V}}
\newcommand{\xmath}{\mathbb{X}}

\newcommand{\nlim}{\lim_{n \to \infty }}
\newcommand{\kinf}{k \to \infty }

\newcommand{\ball}{\mathbb{B}}

\newcommand{\norm}[3]{{\|  #1 \| }_{#2}^{#3}}
\newcommand{\Norm}[3]{{\Bigl\|  #1 \Bigr\| }_{#2}^{#3}}
\newcommand{\ilsk}[3]{{( #1 , #2 ) }_{#3}}

\newcommand{\dual}[3]{{\langle #1 , #2 \rangle }_{#3}}
\newcommand{\Dual}[3]{{\Bigl< #1 , #2 \Bigr>}_{#3}}
\newcommand{\dirilsk}[3]{{( \! ( #1 , #2 ) \! ) }_{#3}}

\newcommand{\p}{\mathbb{P}}

\newcommand{\e}{\mathbb{E}}

\newcommand{\Xn}{{X}_{n}}

\newcommand{\Pn}{{P}_{n}}
\newcommand{\un}{{u}_{n}}
\newcommand{\wn}{{w}_{n}}
\newcommand{\Jn}[1]{{J}^{}_{#1}}

\newcommand{\unk}{{u}_{{n}_{k}}}
\newcommand{\taun}{{\tau}_{n}}

\newcommand{\lhs}{{\lcal }_{HS}}

\newcommand{\Bn}{{B}_{n}}

\newcommand{\bun}[1]{{\bar{u}}_{n}({#1})}

\newcommand{\td}{\tilde{d}}
\newcommand{\rtd}{{\rzecz }^{\td}}

\newcommand{\Phin}{{\Phi }_{n}}

\begin{document}

\begin{frontmatter}
\title{Stochastic hydrodynamic-type  evolution equations
driven by L\'{e}vy noise in 3D unbounded domains
- abstract framework and applications.}

\author[EM]{El\.zbieta Motyl}

\address[EM]{Department of Mathematics and Computer Science, University of \L\'{o}d\'{z}, ul. Banacha 22,
91-238 \L \'{o}d\'{z}, Poland}
\ead{emotyl@math.uni.lodz.pl}

\begin{abstract}
\noindent
The existence of  martingale solutions of the hydrodynamic-type equations in 3D possibly unbounded domains is proved.
The construction of the solution is based on the Faedo-Galerkin approximation. To overcome the difficulty related to the lack of the compactness of Sobolev embeddings in the case of unbounded domain we use certain Fr\'{e}chet space.
 We use also  compactness and tightness criteria  in some nonmetrizable spaces and a version of the Skorokhod Theorem in non-metric spaces. The general framework is applied to the stochastic Navier-Stokes,  magneto-hydrodynamic  (MHD) and the  Boussinesq equations.
\end{abstract}


\end{frontmatter}

\section{Introduction.}

\noindent
Let $\mathcal{O} \subset {\mathbb{R}}^{d} $, $d=2,3$, be an open  connected possibly unbounded subset with smooth boundary $\partial \mathcal{O} $.
Let $\mathbb{H} \subset {L}^{2}( \mathcal{O}; {\rzecz }^{\tilde{d}})) $ 
and $\mathbb{V} \subset {H}^{1}( \mathcal{O} ;{\rzecz }^{\tilde{d}} )$, where $\tilde{d} \in \mathbb{N} $, be two Hilbert spaces such that $\mathbb{V} \subset \mathbb{H}$, the embedding being continuous. 
Here ${H}^{1}( \mathcal{O} ;{\rzecz }^{\tilde{d}} )$ stands for the Sobolev space.
We consider the following stochastic equation
\begin{align} 
&  u(t)  + \int_{0}^{t} \bigl[  \acal u(s)  + \bcal (u(s)) +\rcal u(s) \bigr] \, ds 
   ={u}_{0} + \int_{0}^{t} {f}_{} (s) \, ds + \int_{0}^{t}\int_{{Y}_{0}} F(s,u({s}^{-});y) \tilde{\eta } (ds,dy) \nonumber \\
& \qquad  + \int_{0}^{t}\int_{Y\setminus {Y}_{0}} F(s,u({s}^{-});y) \eta  (ds,dy)
 + \int_{0}^{t}G (s,u(s) ) \, dW(s)  , \qquad t \in (0,T) .\label{E:Intro_equation}
\end{align}  
In this equation $\mathcal{A}, \mathcal{B},\mathcal{R}$ are maps defined in the spaces $\mathbb{H} $ or $\mathbb{V}$, satisfying appropriate conditions (A.1), (B.1)-(B.5) and (R.1), respectively, formulated in Section \ref{S:Statement}.  
Moreover, $W$ stands for a cylindrical Wiener process on a separable Hilbert space and $\eta $ is a time-homogeneous  Poisson random measure on a measurable space $(Y,\mathcal{Y} )$ with a $\sigma $-finite intensity measure $\mu $ and ${Y}_{0} \in \ycal $ is such that $\mu (Y \setminus {Y}_{0})<\infty $.
The processes $W$ and $\eta $ are assumed to be independent.
For example, if $L={(L(t))}_{t \ge 0}$ is a L\'{e}vy process in a Hilbert space $E$ and $\eta $ is the Poisson random measure 
 corresponding to the process of jumps 
${(\Delta L(t))}_{t \ge 0}$, where
\[
     \Delta L(t):= L(t)-L({t}^{-}), \qquad t \ge 0 ,
\]
then we can put ${Y}_{0}:= \{ x \in E:  \norm{x}{E}{} <1 \} $. In this case the noise terms considered in equation \eqref{E:Intro_equation} correspond to the L\'{e}vy-It\^{o} decomposition of the process $L$, see e.g. \cite{Applebaum_2009} and \cite{Peszat_Zabczyk_2007}.
We impose rather general assumptions (F.1)-(F.3) and (G.1)-(G.3) on the noise terms, see Section \ref{S:Statement}.
We prove the existence of a martingale solution  of equation \eqref{E:Intro_equation} understood as a system $(\Omega ,\fcal ,\fmath ,\p ,\eta , W,u)$, where $(\Omega ,\fcal ,\fmath ,\p)$ is a filtered probability and $u ={(u(t))}_{t \in [0,T]}$ is a stochastic process satisfying appropriate regularity properties and integral identity.
The trajectories of the process  $u$ are, in particular, $\hmath $-valued weakly \it c\`{a}dl\`{a}g \rm functions such that  
\begin{equation}
   {\e } \Bigl[ \sup_{t \in [0,T]} {|{u}(t)|}_{\hmath }^{2} 
 + \int_{0}^{T} \norm{{u}(t)}{\vmath }{2} \, ds\Bigr] <\infty . \label{E:Intro_estimates}
\end{equation} 
The construction of a solution is based on the Faedo-Galerkin method, i.e.
\begin{align*} 
   &   \un (t)  =  \Pn {u}_{0} - \int_{0}^{t}\bigl[ \Pn \acal \un (s)  + {\bcal }_{n}  \bigl(\un (s) \bigr)
  + \Pn \rcal \un (s) - \Pn f (s)  \bigr] \, ds  \nonumber   \\ 
  &+  \int_{0}^{t} \int_{{Y}_{0}} \Pn F(s,\un ({s}^{-}),y) \tilde{\eta } (ds,dy) 
  +  \int_{0}^{t} \int_{Y\setminus {Y}_{0}} \Pn F(s,\un ({s}^{-}),y) \eta  (ds,dy) \nonumber \\
 &+ \int_{0}^{t} \Pn G(s,\un (s)) \, dW(s) ,   
\, \, \,  t \in [0,T]  .   
\end{align*}
We prove that the processes ${(\un (t))}_{t \in [0,T]}$, satisfy the following  uniform  estimates
\begin{equation}
  \sup_{n \in \nat } \e \Bigl[ \sup_{t \in [0,t]} {|\un (t)|}_{\hmath }^{p} \Bigr] < \infty 
  \quad \mbox{ and }  \quad 
\sup_{n \in \nat } \e \Bigl[\int_{0}^{T} \norm{\un (t)}{\vmath }{2} \, dt \bigr]  < \infty ,
\label{E:Intro_apriori_estimates}
\end{equation}
where $p \in [1,2+\gamma ]$ and $\gamma > 0$ is a given parameter. For each $n \in \nat $, the process 
$\un $  generates a probability measure $\lcal (\un )$ on appropriate functional space. We prove that the set of laws $\{ \lcal (\un ), n \in \nat \} $ is tight in the space $\zcal $, where
\[
    \zcal := {L}^{2}_{w}(0,T; \vmath ) \cap {L}^{2}(0,T;{L}^{2}_{loc}(\ocal )) \cap
     \dmath (0,T;{\umath }^{\prime }) \cap \dmath (0,T; {\hmath }_{w}), 
\]
defined in Section \ref{S:Compactness_tightness}. 
To this end use the compactness and tightness criteria in  the space $\zcal $, see Lemma 
\ref{L:Dubinsky_cadlag_unbound} and Corollary \ref{C:tigthness_criterion_cadlag_unbound}
in Section \ref{S:Compactness_tightness}. 
They are counterparts for the present abstract settting of the corresponding criteria  proved in \cite{Motyl_NS_Levy_2012}. 
To prove the tightness of $\{ \lcal (\un ), \, n \in \nat   \} $ we use estimates 
\eqref{E:Intro_apriori_estimates} with $p=2$.
Next, we apply a version of the Skorokhod Embedding Theorem for non-metric spaces, see Appendix C, following easily from the Jakubowski's version of the Skorokhod Theorem \cite{Jakubowski_1998}
and from the version due to Brze\'{z}niak and Hausenblas \cite{Brzezniak_Hausenblas_2010}. 
At this stage we need estimates \eqref{E:Intro_apriori_estimates} with $p>2$.

\bigskip  \noindent
The abstract approach is applied to the stochastic
\begin{itemize}
\item Navier-Stokes equations, 
\item  magneto-hydrodynamic equations (MHD),
\item  Boussinesq equations  
\end{itemize}
in the domain $\ocal $. 
In applications, the present approach allows to consider the multiplicative Gaussian noise term, represented by 
$\int_{0}^{t}G(s,u(s)) \, dW(s)$, dependent  both on the state $u$ and their spatial derivatives $\frac{\partial u}{\partial {x}_{i}}$, 
$1\le i \le d $, $d=2,3$. Presence of the derivatives $\frac{\partial u}{\partial {x}_{i}}$ 
 in the noise term is important in modelling the turbulence, see \cite{Brzezniak_Capinski_Flandoli_1991} and \cite{Mikulevicius_Rozovskii_2004}.
Assumptions (G.1)-(G.3) formulated in Section \ref{S:Statement} cover the following example
\[
  G(t,u(t))\, dW(t) = \sum_{i=1}^{\infty } [({b}_{i}(x) \cdot \nabla )u(t,x) + {c}_{i}(x) u(x)] d {\beta }_{i}(t),
\] 
where ${{(\beta )}_{i}}_{i \in \nat } $ are independent real-valued standard Wiener processes, see Section 8 in \cite{Brzezniak_Motyl_NS}.
 
\bigskip \noindent
The present paper is a straightforward generalization of the results of \cite{Motyl_NS_Levy_2012}, where the stochastic Navier-Stokes equations are considered. Here, we construct an abstract framework which covers also other hydro\-dy\-na\-mic-type equations, e.g. stochastic magneto-hydrodynamic and Boussinesq equations. In comparison to \cite{Motyl_NS_Levy_2012} we consider more general L\'{e}vy noise term 
and  additionally  we prove estimates \eqref{E:Intro_estimates} on the solution of equation 
\eqref{E:Intro_equation}. Moreover, to construct a process $u$ it is sufficient to use estimates \eqref{E:Intro_apriori_estimates} with $p>2$ (instead of $p>4$).   

\bigskip  
\noindent
The theory of the stochastic Navier-Stokes equations driven by Gaussian noise  were developed in many papers, see e.g. \cite{Bensoussan_Temam_73}, \cite{Brzezniak_Capinski_Flandoli_1991},
\cite{Capinski_Gatarek_1994}, \cite{Capinski_Peszat_1997}, \cite{Flandoli_Gatarek},
\cite{Capinski_Peszat_2001}, \cite{Mikulevicius_Rozovskii_2004}, \cite{Mikulevicius_Rozovskii},
\cite{Rockner+Zhang_2009}, \cite{Rockner+Zhang_2010} and 
\cite{Brzezniak_Motyl_NS}. The noise term of Poissonian type is considered in the papers
\cite{Dong_Xie_2009}, \cite{Dong_Xie_2011},
\cite{Dong_Zhai_2011} and  \cite{Brzezniak+Hausenblas+Zhu_2013},
and more general L\'{e}vy noise in \cite{Motyl_NS_Levy_2012}  and \cite{Sakthivel+Sritharan_2012}.  
We consider these equations  because of their importance in other hydrodynamic models, e.g. magneto-hydrodynamic equations and Boussinesq equations. 

\bigskip  \noindent
The stochastic magneto-hydrodynamic equations driven by Gaussian noise  
 in 2D domains were considered by
Barbu and Da Prato \cite{Barbu_DaPrato_2007} for additive noise term and Chueshov and Millet \cite{Chueshov_Millet_2009} for multiplicative noise term. 
In the papers by Sritharan and Sundar \cite {Sritharan+Sundar_1999}, and
 by Sango \cite{Sango_2010}  the analysis of the existence of solutions in 2D and 3D bounded domains 
is provided.
In \cite{Sango_2010} the noise term depends both on the velocity $u$ and the magnetic field $\mathfrak{b}$ but does not depend on their spatial derivatives. This follows from assumptions (24) and (25) in  \cite{Sango_2010}.
Here we will generalize these results to the case of unbounded domain when the Gaussian noise term depends on $u, \mathfrak{b}$ and their derivatives $\frac{\partial u}{\partial {x}_{i}}$, 
$\frac{\partial \mathfrak{b}}{\partial {x}_{i}}$, $1\le i \le d $, of the velocity and the magnetic field.
Moreover, we add also  Poissonian type noise term of the form given in equation \eqref{E:Intro_equation}. 

\bigskip \noindent 
The Boussinesq equations has been studied by  Foia\c{s},  Manley and  Temam \cite{Foias_Manley_Temam_87} 
and  Ghidaglia \cite{Ghidaglia_86} in the deterministic case. 
The stochastic Boussinesq equations driven by Gaussian noise is considered by  Duan and  Millet \cite{Duan_Millet_2009}, Ferrario \cite{Ferrario_97} in 2D domains of the form $\rzecz \times [0,1]$.
Martingale solutions in 2D and 3D domains of the form ${\rzecz }^{d-1} \times [0,1]$ with periodic boundary conditions in the directions ${x}_{i}$, $1\le i \le d-1$ were considered in \cite{Brzezniak_Motyl_2010}. 
In the present paper, we generalize the results to the cases of unbounded  domain  $\ocal $. Moreover, we consider  a general L\'{e}vy  noise.

\bigskip  \noindent
The present paper is split into two main parts. 
The first one, consisting of Sections 2-5, concerns the abstract framework.
In Section \ref{S:Statement} we formulate the problem and the general assumptions.
The compactness and tightness criterion  are contained in Section \ref{S:Compactness_tightness}. 
Section \ref{S:Existence} contains the proof of the main theorem on the existence of a martingale solutions. 
The second part ( Section \ref{S:Applications}) is devoted to applications.
Some auxilliary results are given in Appendices .


\pagebreak

\section{Statement of the problem}  \label{S:Statement}

\noindent
Let $\ocal \subset \rd $ be an open  connected possibly unbounded subset with smooth boundary $\partial \ocal $, where $d=2,3$.
Let $(\hmath ,\ilsk{\cdot }{\cdot }{\hmath })$ and $(\vmath ,\ilsk{\cdot }{\cdot }{\vmath })$ be two Hilbert spaces such that 
\[
  \hmath \subset {L}^{2}(\ocal ; \rtd)   \quad \mbox{and}  \quad
  \vmath \subset {H}^{1}(\ocal ;\rtd ) ,
\]
where $\td $ is a positive integer, and 
 the norms in $\hmath $ and $\vmath $ induced by the inner products, denoted by 
$|\cdot {|}_{\hmath }$ and $\norm{\cdot }{\vmath }{}$, are equivalent to the  norms inherited from the spaces ${L}^{2}(\ocal ; \rtd)$ and ${H}^{1}(\ocal ; \rtd)$, respectively.
We assume that $\vmath \hookrightarrow \hmath $ the embedding being dense and continuous.
Moreover, we assume that the inner product in the space $\vmath $ is of the following form
\begin{equation}  \label{E:V_il_sk}
      \ilsk{u}{v}{\vmath }  = \ilsk{u}{v}{\hmath } + \dirilsk{u}{v}{} , \qquad u,v \in \vmath
\end{equation}  
Then the norm  in $\vmath $ is of the form
\begin{equation}  \label{E:norm_V}
   \norm{u}{\vmath }{2} = {|u|}_{\hmath }^{2} + \norm{u}{}{2} , \qquad u \in \vmath ,
\end{equation} 
where $\norm{u}{}{2} = \dirilsk{u}{u}{}$. 
Identifying $\hmath$ with its dual ${\hmath }^{\prime }$, we have the following continuous embeddings
\begin{equation*}
   \vmath \hookrightarrow \hmath \cong {\hmath }^{\prime } \hookrightarrow {\vmath }^{\prime } .
\end{equation*}

\bigskip  \noindent
The dual pairing between a Hilbert space $X$ and its dual space ${X}^{\prime }$ will be denoted by $\dual{\cdot }{\cdot }{{X}^{\prime },X}$. If no confusion seems likely we omit the subscripts ${X}^{\prime }$, $X$ and write $\dual{\cdot }{\cdot }{}$.

\bigskip  \noindent
Let $\bigl( {\ocal }_{R} {\bigr) }_{R \in \nat } $ be a sequence of open and bounded subsets of $\ocal $ with
regular boundaries $\partial {\ocal }_{R}$  such that
\[
{\ocal }_{R} \subset {\ocal }_{R+1} 
\quad \text{ and  } \quad \bigcup_{R=1}^{\infty } {\ocal }_{R} = \ocal .
\]
We will use the  space
${L}^{2}(0,T;{L}^{2}_{loc}(\ocal ))$  of measurable functions $ u:[0,T] \times \ocal  \to \rtd  $
 such that for all $ R \in \nat $
\begin{equation*}
  {p}_{T,R}^{}(u)
 := \Bigl(  \int_{0}^{T} \int_{{\ocal }_{R}}  |u(t,x){|}^{2} dxdt \! {\Bigr) }^{\frac{1}{2}} <\infty  ,  \label{E:seminorms_q=2}
\end{equation*}
with the Fr\'{e}chet topology  generated by the sequence of seminorms
    $ ({p}_{T,R}{)}_{R\in \nat } $.

\bigskip  \noindent
\bf Assumptions. \rm 
We assume that  $\acal ,\bcal $ and $\rcal $ are maps satisfying the following conditions.
\begin{description}
\item[(A.1)] $\acal : \vmath \to \vmath ' $ is a linear map such that 
\begin{equation} \label{E:A_acal_rel}
  \dual{\acal u}{v}{} = \dirilsk{u}{v}{} , \qquad u,v \in \vmath .
\end{equation}
\item[(B.1)] $\bcal : \vmath \times \vmath \to \vmath '$ is a bilinear  map and
there exists a constant ${c}_{1}>0 $ such that
\begin{align}
  | \bcal (u,v) {|}_{{\vmath }^{\prime }} \le {c}_{1} \norm{u}{\vmath }{} \norm{v}{\vmath }{} , 
   \qquad u,v \in \vmath .  \label{E:estimate_B}
\end{align}
\item[(B.2)] $\bcal $ satisfies the following condition
\begin{equation}  \label{E:antisymmetry_B}
      \dual{\bcal (u,v)}{w}{} = - \dual{\bcal (u,w)}{v}{} , \qquad u,v,w \in \vmath .
\end{equation}
We will also use the following notation $\bcal (u) := \bcal (u,u)$.
\item[(B.3)] $\bcal : \vmath \to {\vmath }^{\prime }$ is locally Lipschitz continuous, 
i.e. for every $r>0$ there exists a constant ${L}_{r}$
such that 
\begin{equation*}
   {| \bcal (u) - \bcal (\tilde{u}) | }_{{\vmath }^{\prime }} \le {L}_{r} \norm{u- \tilde{u}}{\vmath}{} ,
   \qquad u , \tilde{u} \in \vmath , \quad \norm{u}{\vmath}{}, \norm{\tilde{u}}{\vmath}{} \le r  .
\end{equation*}
\item[(B.4)] There exist a separable Hilbert space ${\vmath }_{\ast } \subset \vmath $, the embedding being dense and continuous,   such that
$\bcal $ can be extended to a bilinear   map from $\hmath \times \hmath $ 
into ${\vmath }_{\ast }^{\prime }$.
Moreover, there exists a constant ${c}_{2}>0 $ such that
\begin{align}
  | \bcal (u,w) {|}_{{\vmath }_{\ast }^{\prime }} \le {c}_{2} {|u|}_{\hmath}^{} {|w|}_{\hmath}^{}, \qquad u,w \in \hmath .  \label{E:estimate_B_ext}
\end{align}
\item[(B.5)] 
For all $\varphi  \in {\vmath }_{\ast } $ the map ${\tilde{B}}_{\varphi }$ defined by 
\begin{equation} \label{E:B**}
     \bigl( {\tilde{\bcal }}_{\varphi }(u)\bigr) (t):= \dual{\bcal (u(t)}{\varphi }{ }, 
     \quad u \in {L}^{2}(0,T;\hmath ), \quad  t \in [0,T] 
\end{equation}
restricted to bounded subsets of ${L}^{2}(0,T;\hmath ) $ is a continuous map  into 
$ {L}^{1}([0,T]; \rzecz ) $
 if in the space ${L}^{2}(0,T;\hmath )$ we consider the  topology inherited from the space 
${L}^{2}(0,T;{L}^{2}_{loc}(\ocal )) $.
\item[(R.1)] $\rcal  : \hmath \to \vmath '$ is linear and continuous and there exists a constant 
${c}_{3}>0$ such that 
\begin{equation*}
    - \dual{\rcal u}{u}{} \le {c}_{3} |u{|}_{\hmath }^{2}, \qquad u \in \vmath .
\end{equation*}
\end{description}

\bigskip
\begin{remark}  \label{R:B.5_comment}
Condition (B.5) is equivalent to the following one 
\begin{itemize}
\item  if  $({u}_{n}) $ is  a   sequence bounded in ${L}^{2}(0,T;\hmath )$  and ${u}_{n} \to u $ in ${L}^{2}(0,T;{L}^{2}_{loc}(\ocal ))$, then for all $\varphi  \in {\vmath }_{\ast }$:
\begin{equation}
   \nlim \int_{0}^{T}  \dual{ \bcal (\un (s))-\bcal (u(s))}{\varphi }{}   \, ds =0.  
\end{equation}
\end{itemize}
\end{remark}

\bigskip \noindent
Moreover, we impose the following conditions on the random forces, the deterministic force $f$ and the initial state ${u}_{0}$. 
We assume that
\begin{description}
\item[(C.1)] ${u}_{0} \in \hmath $, $f \in {L}^{2}([0,T];{\vmath }^{\prime })$ and 
$\mathfrak{A}:=(\Omega , \fcal , \fmath , \p )$ is a filtered probability space with a filtration
 $\fmath ={({\fcal }_{t})}_{t \ge 0}$ satisfying usual hypotheses.
\item[(F.1)] Let $(Y, \ycal )$ be a measurable space and let $\mu $ be a $\sigma $-finite measure on $(Y, \ycal )$. Let ${Y}_{0} \in \ycal $ be such that $\mu (Y\setminus {Y}_{0}) < \infty $.
Assume that $\eta $ is a  time homogeneous Poisson random measure on  $(Y, \ycal )$  over
$\mathfrak{A}$ with the (jump) intensity measure $\mu $,
\item[(F.2)] $F:[0,T]\times \hmath \times Y \to \hmath $ is  a measurable function 
and there exists a constant $L$ such that
\begin{equation}
    \int_{Y} |F(t,{u}_{1};y)- F(t,{u}_{2};y) {|}_{\hmath }^{2}  \mu (dy) \le  L |{u}_{1}-{u}_{2}{|}_{\hmath }^{2} 
   , \quad {u}_{1},  {u}_{2} \in \hmath  , \, \,  t \in [0,T] \label{E:F_Lipschitz_cond} ,
\end{equation}
and for each $p \in \{ 1,2,2+\gamma , 4, 4+2\gamma \} $ there exists a constant ${C}_{p}$ such that
\begin{equation}   
    \int_{Y} |F(t,u;y) {|}_{\hmath }^{p} \, \mu (dy) \le  {C}_{p} (1 + |u{|}_{\hmath }^{p}), \qquad  u \in \hmath  , \quad t \in [0,T],  \label{E:F_linear_growth}
\end{equation}
where $\gamma >0$ is some positive constant.
\item[(F.3)] Moreover, for all $\varphi  \in \hmath $ the mapping ${\tilde{F}}_{\varphi }$ defined by 
\begin{equation} \label{E:F**}
     \bigl( {\tilde{F}}_{\varphi }(u)\bigr) (t,y):= \ilsk{F(t,u({t}^{-});y)}{\varphi }{\hmath }, 
     \quad u \in {L}^{2}(0,T;\hmath ), \quad (t,y) \in [0,T] \times Y 
\end{equation}
is a continuous from ${L}^{2}(0,T;\hmath ) $ into $ {L}^{2}([0,T]\times Y, dl\otimes \mu ; \rzecz ) $
 if in the space ${L}^{2}(0,T;\hmath )$ we consider the  topology inherited from the space 
$ {L}^{2}(0,T;{L}^{2}_{loc}(\ocal )) $.
\footnote{Here $l$ denotes the Lebesgue measure on the interval $[0,T]$.} 
\item[(G.1)] $ W(t)$  is a cylindrical  Wiener process in a separable Hilbert space ${Y}_{W}$ defined on the stochastic basis $\mathfrak{A}$. The process $W$ is independent of $\eta $. 
\item[(G.2)]  $G: [0,T] \times \vmath  \to \lhs ({Y}_{W},\hmath ) $ and there exists a constant ${L}_{G}>0$
 such that 
\begin{equation}
   \norm{G(t,{u}_{1}) - G(t,{u}_{2})}{\lhs ({Y}_{W},\hmath )}{2} \le {L}_{G} \norm{{u}_{1}-{u}_{2}}{\vmath }{2} ,
   \quad {u}_{1}, {u}_{2} \in \vmath  , \, \,  t \in [0,T] .
\end{equation} 
Moreover there exist ${\lambda }_{}$, $\kappa \in \rzecz $ and $a\in \bigl( 2-\frac{2}{3+\gamma },2\bigr]$  such that
\begin{equation} \label{E:G}
     2 \dual{\acal u }{u }{} -  \norm{G(t,u )}{\lhs ({Y}_{W},\hmath )}{2}
     \ge  a \norm{u}{}{2} -{\lambda }_{} {|u |}_{\hmath }^{2} - \kappa  , \quad u \in \vmath  , \, \, t \in [0,T].
\end{equation}
\item[(G.3)]
Moreover, $G $ extends to a continuous mapping $G :[0,T] \times \hmath \to \lhs  ( {Y}_{W},  {{\vmath }^{\prime }}) $  such that
\begin{equation} \label{E:G*}
   \norm{G(t,u)}{\lhs ({Y}_{W}, {\vmath }^{\prime })}{2} \le C (1 + {|u|}_{\hmath }^{2}) , \qquad u \in \hmath  . 
\end{equation}
for some $C>0$. Moreover, for every $\varphi  \in \vmath $ the map ${\tilde{G}}_{\varphi }$ defined by
\begin{equation} \label{E:G**}
  \bigl( {\tilde{G}}_{\varphi }(u)\bigr)  (t) := \dual{G(t,u(t))}{\varphi }{} ,
 \qquad u \in {L}^{2}(0,T;\hmath ) , \quad t \in [0,T]
\end{equation}
is a continuous mapping from ${L}^{2}(0,T;\hmath ) $ into $ {L}^{2}([0,T];\lhs ({Y}_{W},\rzecz ) ) $
 if in the space  ${L}^{2}(0,T;\hmath )$ we consider the  topology inherited from the space ${L}^{2}(0,T;{L}^{2}_{loc}(\ocal ))$.
\end{description}
For any Hilbert space $E$ the symbol $\lhs ({Y}_{W};E)$ denotes the space of Hilbert-Schmidt operators from ${Y}_{W}$ into  $E$. 


\bigskip  \noindent
Let us consider the following stochastic equation
\begin{align} 
&  u(t)  + \int_{0}^{t} \bigl[  \acal u(s)  + \bcal (u(s)) +\rcal u(s) \bigr] \, ds 
   ={u}_{0} + \int_{0}^{t} {f}_{} (s) \, ds + \int_{0}^{t}\int_{{Y}_{0}} F(s,u({s}^{-});y) \tilde{\eta } (ds,dy) \nonumber \\
& \qquad  + \int_{0}^{t}\int_{Y\setminus {Y}_{0}} F(s,u({s}^{-});y) \eta  (ds,dy)
 + \int_{0}^{t}G (s,u(s) ) \, dW(s)  , \qquad t \in (0,T) .\label{E:equation}
\end{align}  

\bigskip
\begin{definition}  \rm  \label{D:solution}
 \bf A martingale solution \rm of  equation \eqref{E:equation}
is a system 
\noindent
$\bigl( \bar{\mathfrak{A}}, \bar{\eta }, \bar{W} ,\bar{u}\bigr) $,
where
\begin{itemize}
\item[$\bullet $]  $\bar{\mathfrak{A}}:= \bigl( \bar{\Omega }, \bar{\fcal },  \bar{\fmath } ,\bar{\p }  \bigr) $ is a filtered probability space with a filtration $\bar{\fmath } = \{ {\bar{\fcal }_{t}}{\} }_{t \ge 0} $, 
\item[$\bullet $] $\bar{\eta }$  is a time homogeneous Poisson random measure on $(Y, \ycal )$ over
$\bar{\mathfrak{A}} $ with the intensity measure $\mu $,
\item[$\bullet $] $\bar{W}$ is a cylindrical Wiener process on the space ${Y}_{W}$ over
$\bar{\mathfrak{A}} $,
\item[$\bullet $] $\bar{u}: [0,T] \times \Omega \to \hmath $ is a predictable process with 
$\bar{\p } $-a.e. paths
\[
  \bar{u}(\cdot , \omega ) \in \dmath \bigl( [0,T], {\hmath }_{w} \bigr)
   \cap {L}^{2}(0,T;\vmath )
\]
such that for all $ t \in [0,T] $ and all $\varphi \in \vmath $ the following identity holds $\bar{\p }$ - a.s.
\begin{align*}
 &\ilsk{\bar{u}(t)}{\varphi}{\hmath }  + \int_{0}^{t} \dual{\acal \bar{u}(s)}{\varphi}{}  ds
+ \int_{0}^{t} \dual{\bcal (\bar{u}(s))}{\varphi}{}  ds 
+ \int_{0}^{t} \dual{\rcal \bar{u}(s)}{\varphi}{}  ds
   \nonumber \\
 & =\ilsk{{u}_{0}}{\varphi}{\hmath } 
 + \int_{0}^{t} \dual{f(s)}{\varphi}{} ds
  + \int_{0}^{t} \int_{{Y}_{0}} \ilsk{F(s,\bar{u}(s);y)}{\varphi}{\hmath } \,  \tilde{\bar{\eta }} (ds,dy) \\
 & + \int_{0}^{t} \int_{Y\setminus {Y}_{0}} \ilsk{F(s,\bar{u}(s);y)}{\varphi}{\hmath } 
 \, \bar{\eta } (ds,dy) 
  + \Dual{\int_{0}^{t}G(s,\bar{u}(s))\,  d\bar{W}(s)}{\varphi}{} . 
\end{align*}
\end{itemize}
\end{definition}

\bigskip \noindent
The symbol ${\hmath }_{w}$ denotes the Hilbert space $\hmath $ endowed with the weak topology
and $\dmath ([0,T];{\hmath }_{w}) $ is the space of weakly \it c\`{a}dl\`{a}g \rm functions
$ u : [0,T] \to \hmath  $.
Recall that $u:[0,T]\to \hmath $ is  weakly \it c\`{a}dl\`{a}g \rm  iff for every $h \in \hmath $ the real-valued function
$
     [0,T] \ni t \mapsto \ilsk{u(t)}{h}{\hmath }
$
is \it c\`{a}dl\`{a}g. \rm 
 
\bigskip  \noindent
The main result of the present paper is expressed in the following theorem.

\bigskip 
\begin{theorem} \label{T:existence}  
Let assumptions (A.1), (B.1)-(B.5), (R.1), (C.1), (F.1)-(F.3) and (G.1)-(G.3) be satisfied.
Then there exists a martingale solution $\bigl( \bar{\mathfrak{A} }, \bar{\eta }, \bar{W} ,\bar{u}\bigr) $ 
of  problem \eqref{E:equation} such that
\begin{equation}  \label{E:u_estimates}
   \bar{\e } \Bigl[ \sup_{t \in [0,T]} {|\bar{u}(t)|}_{\hmath }^{2} 
 + \int_{0}^{T} \norm{\bar{u}(t)}{\vmath }{2} \, ds\Bigr] <\infty .
\end{equation} 
\end{theorem}

\bigskip  \noindent
Assumption (F.3) and second part of assumption (G.3) are important in the case of unbounded domain $\ocal $. If $\ocal $ is bounded, they can be omitted.
Assumptions (G.2)-(G.3) allow to consider the Gaussian noise term  dependent both on $u$ and ${\partial}_{{x}_{i}}u$, $i=1,...,d$.
This corresponds to inequality \eqref{E:G} with  $a<2$. In the case when $a=2$  the noise term $G$ depends on $u$ but not on its spatial derivatives.

\bigskip  \noindent
The proof of Theorem \ref{T:existence} is based on the Faedo-Galerkin method. 
To this end we need appropriate orthonormal basis in the space $\hmath $. In the next section we recall a general approach used also in \cite{Brzezniak_Motyl_NS} and \cite{Motyl_NS_Levy_2012} in the case of Navier-Stokes equations.


\bigskip
\section{Auxiliary results from functional analysis - space $\umath $ and an orthonormal basis in $\hmath $} \label{S:Funct_anal}

\noindent
Let us recall that we have the following three separable Hilbert spaces such that
\begin{equation}  \label{E:V_*,V,H}
    {\vmath }_{\ast } \subset \vmath \subset \hmath , 
\end{equation}
the embedding being dense and continuous. 
Since ${\vmath }_{\ast }$ is a separable Hilbert space, there exists a Hilbert space $\umath $ such that $\umath \subset {\vmath }_{\ast}$, $\umath $ is dense in ${\vmath }_{\ast }$ and 
the embedding
\begin{equation} \label{E:U_comp_V_*}
   \umath \hookrightarrow  {\vmath }_{\ast }
\end{equation}
is compact. In particular, $\umath $ is compactly embedded into the space $\hmath $.
Let us  denote 
$
   \iota  : \umath  \hookrightarrow \hmath 
$
and let 
$
   {\iota }^{*}  :  \hmath  \to \umath  
$
be its adjoint operator.
Note that $\iota $ is compact and since the range of $\iota $ is dense in $\hmath $, ${\iota }^{*} : \hmath  \to \umath  $ is one-to-one. Let us put $D(L) := {\iota }^{*}(\hmath ) \subset \umath $ and
\begin{align}
  Lu :=& \bigl( {\iota }^{*} {\bigr) }^{-1} u , \qquad u \in D(L) .  \label{E:op_L}
\end{align}
It is clear that $L:D(L) \to \hmath  $ is onto. Let us also notice that
\begin{equation} \label{E:op_L_ilsk}
    \ilsk{Lu}{w}{\hmath } = \ilsk{u}{w}{\umath }, \qquad u \in D(L), \quad w \in \umath .
\end{equation}
By equality (\ref{E:op_L_ilsk}) and the densiness of $\umath $ in $\hmath $, we infer that $D(L)$ is dense in $\hmath $.

\bigskip  \noindent
Since $L$ is self-adjoint and ${L}^{-1}$ is compact, there exists an orthonormal basis $\{ {e}_{i} {\} }_{i \in \nat }$ of $\hmath $ composed of the eigenvectors of operator $L$. Let us fix $n \in \nat $ and let $\Pn $ be the operator from ${\umath }^{\prime }$ to $span \{ {e}_{1},..., {e}_{n}\} $ defined by
\begin{equation} \label{E:P_n}
  \Pn {u}^{*} := \sum_{i=1}^{n} \bigl< {u}^{*}| {e}_{i}\bigr> {e}_{i}, \qquad {u}^{*} \in {\umath }^{\prime },
\end{equation}
where $\dual{\cdot }{\cdot }{}$ denotes the duality pairing between the space $\umath $ and its dual ${\umath }^{\prime }$.
Note that the restriction of $\Pn $ to $\hmath $, denoted still by $\Pn $, is given by 
\begin{equation}
   \Pn u = \sum_{i=1}^{n} \ilsk{ u}{ {e}_{i}}{\hmath }  {e}_{i}, \qquad  u \in \hmath  ,
\end{equation}
and thus it is the $\ilsk{\cdot }{\cdot }{\hmath }$-orthogonal projection
onto  $span \{ {e}_{1},..., {e}_{n}\} $. Restrictions of $\Pn $ to other spaces considered in 
 \eqref{E:V_*,V,H} will also be denoted by $\Pn $. Moreover, it is easy to see that
\begin{equation}  \label{E:tP_n-P_n}
   \ilsk{\Pn {u}^{*}}{v}{\hmath } = \dual{ {u}^{*}}{\Pn v}{} , \qquad {u}^{*} \in {\umath }^{\prime }, \quad v \in \umath .
\end{equation}
Let us denote
$
   {\tilde{e}}_{i} := \frac{{e}_{i}}{\norm{{e}_{i}}{U}{}} , \ i \in \nat .
$
The following lemma is a straightforward counterpart of Lemma 2.4 in \cite{Brzezniak_Motyl_NS} corresponding to our abstract setting. 
\begin{lemma} \label{L:P_n|U} \
\begin{description}
\item[(a)] The system $\bigl\{ {\tilde{e}}_{i} {\bigr\} }_{n \in \nat }$ is the orthonormal basis in the space $\bigl( \umath ,\ilsk{\cdot }{\cdot }{\umath }\bigr) $.
\item[(b)] For every $n\in \nat $ and $u \in \umath $
\begin{equation} \label{E:P_n|U}
  \Pn u = \sum_{i=1}^{n} \ilsk{u}{{\tilde{e}}_{i}}{\umath } {\tilde{e}}_{i}, 
\end{equation}
i.e., the restriction of ${P}_{n}$ to $\umath $ is the $\ilsk{\cdot }{\cdot }{\umath }$-projection onto the subspace $span \{ {e}_{1},...,{e}_{n}  \} $.
\item[(c)] For every $u\in \umath $
\begin{description}
\item[(i)] $\lim_{n \to \infty } \norm{{P}_{n}u-u}{\umath }{} =0$,
\item[(ii)] $\lim_{n \to \infty } \norm{{P}_{n}u-u}{{\vmath }_{\ast }}{} =0$, 
\item[(iii)] $\lim_{n \to \infty } \norm{{P}_{n}u-u}{\vmath }{} =0$.
\end{description}
\end{description}
\end{lemma}

\begin{proof}
The proof is essentially the same as the proof of Lemma 2.4 in \cite{Brzezniak_Motyl_NS} and thus omitted.
\end{proof}


\section{Compactness and tightness results}  \label{S:Compactness_tightness}

\subsection{Deterministic compactness criterion}
 \noindent
Let us recall that we have the following separable Hilbert spaces
\[ 
   \umath \hookrightarrow \vmath \hookrightarrow \hmath \cong {\hmath }^{\prime } \hookrightarrow {\umath }^{\prime },
\]
where the embedding $\umath \hookrightarrow \vmath $ is dense and compact and  the embedding $\vmath \hookrightarrow \hmath $ is continuous.
Let us consider the following functional spaces being the counterparts in our framework of the spaces used in \cite{Motyl_NS_Levy_2012}, see also \cite{Metivier_88} and \cite{Metivier_Viot_88}:
\begin{itemize}
\item $\dmath ([0,T],{\umath }^{\prime }) $ := the space of  c\`{a}dl\`{a}g functions 
 $ u:[0,T] \to {\umath }^{\prime } $  with the topology  $ {\tcal }_{1}$
               induced by the Skorokhod metric,
\item ${L}_{w}^{2}(0,T;\vmath ) $ := the space ${L}^{2} (0,T;\vmath )$ with the weak topology 
                     $ {\tcal}_{2} $,      
\item ${L}^{2}(0,T;{L}^{2}_{loc}(\ocal ))$ := the space of measurable functions 
 $ u:[0,T]\times \ocal  \to \rtd  $ such that for all $ R \in \nat $
\begin{equation}                 
      {p}_{T,R}^{}(u):= 
     \Bigl(  \int_{0}^{T} \int_{{\ocal }_{R}}  |u(t,x){|}^{2} dxdt  {\Bigr) }^{\frac{1}{2}} <\infty  ,
      \label{E:seminorms} 
\end{equation}      
with the topology  $ {\tcal }_{3}$  generated by the seminorms 
$({p}_{T,R}{)}_{R\in \nat } .$                     
\end{itemize}
Let ${\hmath }_{w}$ denote the Hilbert space $\hmath $ endowed with the weak topology. 
Let us consider the fourth space, see \cite{Motyl_NS_Levy_2012},
\begin{itemize} 
\item $\dmath ([0,T];{\hmath }_{w}) $ : = the space of weakly c\`{a}dl\`{a}g functions  
$ u : [0,T] \to \hmath $  with the weakest topology ${\tcal }_{4}$ such that for all 
                          $h \in \hmath  $   the  mappings 
\begin{equation}
 \dmath ([0,T];{\hmath }_{w}) \ni u  \mapsto \ilsk{u(\cdot )}{h}{\hmath } \in \dmath  ([0,T];\rzecz ) 
\label{E:D([0,T];H_w)_cadlag}  
\end{equation} 
are continuous.                     
In particular,  
$\un \to u $ in $\dmath ([0,T];{\hmath }_{w}) $ iff  for all $ h \in \hmath  $:
$
  \ilsk{\un (\cdot )}{h}{\hmath }  \to \ilsk{u(\cdot )}{h}{\hmath }   \mbox{ in the space }  
  \dmath ([0,T];\rzecz ).
$ 
\end{itemize}

  \noindent
We will use the following modification of  \cite[Theorem 2]{Motyl_NS_Levy_2012}.

\noindent
\begin{lemma} \rm 
\label{L:Dubinsky_cadlag_unbound} \it
Let 
\begin{equation} \label{E:Z_cadlag}
 {\zcal }_{}: =   {L}_{w}^{2}(0,T;\vmath )  \cap {L}^{2}(0,T;{L}^{2}_{loc}(\ocal )) 
\cap \dmath ([0,T]; {\umath }^{\prime }) 
  \cap \dmath ([0,T],{\hmath }_{w})
\end{equation}
and let $\tcal $ be  the supremum of the corresponding topologies. 
Let
\[
\kcal \subset {L}^{\infty }(0,T;\hmath ) \cap {L}^{2}(0,T;\vmath ) \cap \dmath ([0,T];{\umath }^{\prime })
\]
satisfy the following three conditions 
\begin{itemize}
\item[(a)  ] for all $u \in \kcal $ and  all $t \in [0,T]$, $u(t) \in \hmath   $ and 
$\, \, \sup_{u\in \kcal } \sup_{s \in[0,T]} {|u(s)|}_{\hmath } < \infty  $, 
\item[(b)] $ \sup_{u\in \kcal } \int_{0}^{T} \norm{u(s)}{\vmath }{2} \, ds < \infty  $,
  i.e. $\kcal $ is bounded in ${L}^{2}(0,T;\vmath )$,
\item[(c)] $\lim{}_{\delta \to 0 } \sup_{u\in \kcal } {w}_{[0,T],{\umath }^{\prime }}(u;\delta ) =0 $.
\end{itemize}
Then $\kcal \subset {\zcal }_{}$  and $\kcal $ is $\tcal $-relatively compact in ${\zcal }_{}$.
\end{lemma}

\bigskip  \noindent
The proof of Lemma \ref{L:Dubinsky_cadlag_unbound} is given in Appendix A.

\subsection{Tightness criterion}
 
 \noindent
Let $(\Omega , \fcal ,\p )$ be a probability space with filtration $\mathbb{F}:=({\fcal }_{t}{)}_{t \in [0,T]}$ satisfying the usual hypotheses.
Using Lemma \ref{L:Dubinsky_cadlag_unbound}, we get the corresponding tightness criterion in the
measurable  space  $(\zcal ,\sigma (\zcal ))$, where $\zcal $ is defined by \eqref{E:Z_cadlag}
and $\sigma (\zcal )$ denotes the topological $\sigma $-field, see \cite[Corollary 1]{{Motyl_NS_Levy_2012}}.

\begin{cor}  \label{C:tigthness_criterion_cadlag_unbound}
\it Let $(\Xn {)}_{n \in \nat }$ be a sequence of c\`{a}dl\`{a}g $\mathbb{F}$-adapted 
${\umath }^{\prime }$-valued processes such that
\begin{description}
\item[(a)] there exists a positive constant ${C}_{1}$ such that
\[
         \sup_{n\in \nat}\e \bigl[ \sup_{s \in [0,T]} {|\Xn (s) |}_{\hmath }  \bigr]  \le {C}_{1} ,
\]
\item[(b)] there exists a positive constant ${C}_{2}$ such that
\[
    \sup_{n\in \nat}\e \Bigl[  \int_{0}^{T} \norm{\Xn (s)}{\vmath }{2} \, ds    \Bigr]  \le {C}_{2} ,
\]
\item[(c)]  $(\Xn {)}_{n \in \nat }$ satisfies the Aldous condition  in ${\umath }^{\prime }$.
\end{description}
Let ${\tilde{\p }}_{n}$ be the law of $\Xn $ on ${\zcal }_{}$.
Then for every $\eps >0 $ there exists a compact subset ${K}_{\eps }$ of ${\zcal }_{}$ such that
\[
   {\tilde{\p }}_{n} ({K}_{\eps })  \ge 1 - \eps .
\]
\end{cor}

\bigskip  \noindent 
Let us recall the Aldous condition in the form given by  M\'{e}tivier.

\begin{definition} (M. M\'{e}tivier) \label{D:Aldous}
\rm A sequence $({X}_{n}{)}_{n\in \nat }$  satisfies the \bf  Aldous condition \rm
in the space ${\umath }^{\prime }$
iff

\bigskip  \noindent
 $ \forall \, \eps >0 \ \forall \, \eta >0 \ \exists \, \delta >0
$ such that for every sequence $({{\tau}_{n} } {)}_{n \in \nat }$ of $\mathbb{F}$-stopping times with
${\tau }_{n}\le T$ one has
\[
    \sup_{n \in \nat} \, \sup_{0 \le \theta \le \delta }  \p \bigl\{
    {| {X}_{n} ({\tau }_{n} +\theta )-{X}_{n} ( {\tau }_{n}  ) |}_{{\umath }^{\prime }} \ge \eta \bigr\}  \le \eps .
\]
\end{definition}

\section{Existence of solutions} \label{S:Existence}

\subsection{Faedo-Galerkin approximation}

\noindent
Let $\umath $ be the space defined by \eqref{E:U_comp_V_*}.
Let $\{ {e}_{i} {\} }_{i =1}^{\infty  }$ be the orthonormal basis in $\hmath $ composed of eigenvectors of the operator $L$ defined by \eqref{E:op_L}. In particular, $\{ {e}_{i} {\} }_{i =1}^{\infty  } \subset \umath $.
Let ${\hmath }_{n}:= span \{ {e}_{1}, ..., {e}_{n} \} $ be the subspace with the norm inherited from $\hmath $ and
let $\Pn $ be defined by \eqref{E:P_n}.
Consider the following map
\[
  {\bcal }_{n} (u):= \Pn \bcal ({\chi }_{n}(u),u) , \qquad u \in {\hmath }_{n},
\]
where ${\chi }_{n}:\hmath  \to \hmath  $ is defined by ${\chi }_{n}(u) = {\theta }_{n}(|u {|}_{{\umath }^{\prime }})u$, where  ${\theta }_{n } : \rzecz \to [0,1]$  of class ${\ccal }^{\infty }$ such that
\begin{eqnarray*}
 {\theta }_{n}(r)  = 1 \quad \mbox{if} \quad  r \le n  \quad \mbox{ and } \quad 
  {\theta }_{n}(r)  = 0 \quad \mbox{if} \quad  r \ge n+1 .
\end{eqnarray*}
Since ${\hmath }_{n} \subset H$, ${\bcal }_{n}$ is well defined. Moreover, ${\bcal }_{n}:{\hmath }_{n} \to {\hmath }_{n}$ is globally Lipschitz continuous.
Let us consider the classical Faedo-Galerkin approximation in the space $ {\hmath }_{n}$
\begin{align} 
   &   \un (t)  =  \Pn {u}_{0} - \int_{0}^{t}\bigl[ \Pn \acal \un (s)  + {\bcal }_{n}  \bigl(\un (s) \bigr)
  + \Pn \rcal \un (s) - \Pn f (s)  \bigr] \, ds  \nonumber   \\ 
  &+  \int_{0}^{t} \int_{{Y}_{0}} \Pn F(s,\un ({s}^{-}),y) \tilde{\eta } (ds,dy) 
  +  \int_{0}^{t} \int_{Y\setminus {Y}_{0}} \Pn F(s,\un ({s}^{-}),y) \eta  (ds,dy) \nonumber \\
 &+ \int_{0}^{t} \Pn G(s,\un (s)) \, dW(s) ,   
\, \, \,  t \in [0,T]  .   \label{E:Galerkin}
\end{align}

\bigskip
\begin{lemma} \label{L:Galerkin_existence}
For each $n \in \nat $,  there exists a unique  $\fmath $-adapted,  c\`{a}dl\`{a}g ${\hmath }_{n}$ valued process  ${u}_{n}$ satisfying the Galerkin equation (\ref{E:Galerkin}).
\end{lemma}

\begin{proof}
The assertion follows from Theorem 9.1 in \cite{Ikeda_Watanabe_81}, see also \cite{Albeverio_Brzezniak_Wu_2010}. 
\end{proof}

\bigskip  
\noindent
In the following lemma we will prove uniform estimates of the solutions ${u}_{n} $ of \eqref{E:Galerkin}.
Actually, these estimates hold provided the noise terms satisfy only condition \eqref{E:F_linear_growth} in assumption (F.2) and inequality \eqref{E:G} in assumption (G.2). 
The proof of the lemma is based on the It\^{o} formula, see \cite{Ikeda_Watanabe_81} or \cite{Metivier_82},
and the Burkholder-Davis-Gundy inequality, see \cite{Peszat_Zabczyk_2007}.

\bigskip
\begin{lemma} \label{L:Galerkin_estimates }
The processes $({u}_{n} {)}_{n \in \nat }$ satisfy the following estimates.
\begin{itemize}
\item[(i) ]
For every $p\in [1,2+\gamma] $  there exists a  positive constant ${C}_{1}(p)$  such that
\begin{equation} \label{E:H_estimate}
 \sup_{n \ge 1 } \e \bigl( \sup_{0 \le s \le T } |\un (s){|}_{\hmath }^{p} \bigr) \le {C}_{1}(p) .
\end{equation}
\item[(ii)] There exists a positive constant ${C}_{2}$ such that
\begin{equation} \label{E:V_estimate}
  \sup_{n \ge 1 }  \e \bigl[ \int_{0}^{T} \norm{ \un (s)}{\vmath }{2} \, ds \bigr] \le {C}_{2}.
\end{equation}
\end{itemize}
(Here $\gamma >0$ is the constant  defined in assumption (F.2).)
\end{lemma}
\noindent

\bigskip  \noindent
The proof of Lemma \ref{L:Galerkin_estimates } is postponed to Appendix D.

\bigskip
\noindent
The solutions ${u}_{n} $, $n \in \nat $, of the Galerkin equations define  probability measures
$\lcal ({u}_{n})$, $n \in \nat $, on the measurable space $(\zcal , \sigma (\tcal ))$, defined by 
\eqref{E:Z_cadlag} with the topological $\sigma $-field $\sigma (\tcal )$. 
Using Corollary \ref{C:tigthness_criterion_cadlag_unbound} and Lemma \ref{L:Galerkin_estimates }
we will prove that the set of measures 
$\bigl\{ \lcal ({u}_{n} ) , n \in \nat  \bigr\} $ is tight on $(\zcal , \sigma (\tcal ))$.
We use inequalities \eqref{E:V_estimate} and \eqref{E:H_estimate} with $p=2$.

\bigskip
\begin{lemma} \label{L:comp_Galerkin}
The set of measures $\bigl\{ \lcal ({u}_{n} ) , n \in \nat  \bigr\} $ is tight on 
$(\zcal , \sigma (\tcal ))$.
\end{lemma}

\bigskip  \noindent
The proof of Lemma \ref{L:Galerkin_estimates } is given in Appendix D.

\bigskip  \noindent
Further construction of a martingale solution of equation \eqref{E:equation} is based on the Skorokhod Theorem for nonmetric spaces, see \cite{Brzezniak_Hausenblas_2010} and Appendix C of the present paper.
This theorem guaranties, in particular, the existence of a sequence $({\bar{u}}_{n})$ of $\zcal $-valued stochastic processes such that $\lcal (\un )= \lcal ({\bar{u}}_{n})$, $n \in \nat $, convergent almost surely to a limit process  on a different probability space. The main difficulty accur in passing to the limit in the nonlinear term, in the cases of unbounded domain. Here we need inequality 
\eqref{E:H_estimate} with $p>2$, as well as Assumption (B.5). 
Actually, this is the only place, where we use \eqref{E:H_estimate} with $p>2$.
Let us mention that in the next section, devoted to applications, we will prove that the nonlinear terms appearing in the hydrodynamic-type equations satisfy Assumption (B.5).
Similar problems occur in the noise terms, where Assumptions (F.3) and (G.3) are important.  
The method used in the following proof of Theorem \ref{T:existence} is closely related to the approach due to Brze\'{z}niak and Hausenblas  
\cite{Brzezniak_Hausenblas_2010}.

\bigskip
\subsection{Proof of Theorem \ref{T:existence} }

\bigskip 
\noindent
We will apply the Skorokhod Theorem for the sequence of laws of $({u}_{n}, {\eta }_{n},{W}_{n})$, where
${\eta }_{n}:= \eta $ and ${W}_{n}:= W $, $n \in \nat $.
Since  by Lemma \ref{L:comp_Galerkin} the set of measures $\bigl\{ \lcal ({u}_{n} ) , n \in \nat  \bigr\} $ is tight on the space $\zcal $, the set $ \bigl\{ \lcal ({u}_{n}, {\eta }_{n},{W}_{n}) , n \in \nat  \bigr\}$ is tight on  $\zcal \times {M}_{\bar{\nat }}([0,T]\times Y) \times \ccal ([0,T];{Y}_{W} )$.
Here $\ccal ([0,T];{Y}_{W})$ denotes the space of ${Y}_{W}$-valued continuous functions with the standard supremum-norm and ${M}_{\bar{\nat }}([0,T]\times Y)$ is defined in Appendix B.
By Corollary \ref{C:Skorokhod_J,B,H} and Remark \ref{R:separating_maps}, see Appendix C, there exists a subsequence $({n}_{k}{)}_{k\in \nat }$, a probability space 
$\bigl( \bar{\Omega }, \bar{\fcal },\bar{\p }  \bigr) $ and, on this space, 
$\zcal \times {M}_{\bar{\nat }}([0,T]\times Y) \times \ccal ([0,T];{Y}_{W} )$-valued random variables $({u}_{\ast },{\eta }_{\ast },{W}_{\ast })$, 
$({\bar{u}}_{k}, {\bar{\eta }}_{k},{\bar{W}}_{k})$, $k \in \nat $ such that
\begin{itemize}
\item[(i) ] $\lcal \bigl( ({\bar{u}}_{k}, {\bar{\eta }}_{k},{\bar{W}}_{k}) \bigr) = \lcal \bigl( ({u}_{{n}_{k}}, {\eta }_{{n}_{k}},{W}_{{n}_{k}}) \bigr) $ for all $ k \in \nat $;
\item[(ii) ] ${(\bar{u}}_{k}, {\bar{\eta }}_{k},{\bar{W}}_{k}) \to ({u}_{\ast },{\eta }_{\ast },{W}_{\ast })$ in $\zcal \times {M}_{\bar{\nat }}([0,T]\times Y)  \times \ccal ([0,T];{Y}_{W} )$ with probability $1$ on $\bigl( \bar{\Omega }, \bar{\fcal },\bar{\p }  \bigr) $ as $k \to \infty $;
\item[(iii) ] $ ({\bar{\eta }}_{k} (\bar{\omega }), {\bar{W}}_{k}(\bar{\omega }) )
 =  ({\eta }_{\ast }(\bar{\omega }),{W}_{\ast }(\bar{\omega }))$ for all $\bar{\omega } \in \bar{\Omega }$.
\end{itemize}
We will denote these sequences again by $\bigl(  ({u}_{n}, {\eta }_{n},{W}_{n})  {\bigr) }_{ n\in \nat }$ and 
$\bigl(  ({\bar{u}}_{n}, {\bar{\eta }}_{n},{\bar{W}}_{n})  {\bigr) }_{ n\in \nat } $.
Moreover, ${\bar{\eta }}_{n}$, $n \in \nat $, and ${\eta }_{*}$ are time homogeneous Poisson random measures on $(Y, \ycal )$ with the intensity measure $\mu $ and ${\bar{W}}_{n}$, $n \in \nat $, and ${W}_{*}$ are cylindrical Wiener processes, see \cite[Section 9]{Brzezniak_Hausenblas_2010}.
By the definition of the space $\zcal $, see \eqref{E:Z_cadlag}, we have
\begin{equation} \label{E:Z_cadlag_NS_bar_un_conv} 
  {\bar{u}}_{n} \to {u}_{*} \quad \mbox{in} \quad  {L}_{w}^{2}(0,T;\vmath ) \cap 
{L}^{2}(0,T;{L}^{2}_{loc}(\ocal ))  \cap \dmath ([0,T];{\umath }^{\prime})
  \cap \dmath ([0,T];{\hmath }_{w}) \quad \bar{\p } \mbox{-a.s.}
\end{equation}
Since the random variables ${\bar{u}}_{n}$ and ${u}_{n}$ are identically distributed, we have the following inequalities.
For every $p\in [1,2+\gamma ] $  
\begin{equation} \label{E:H_estimate_bar_u_n}
 \sup_{n \ge 1 } \bar{\e } \bigl( \sup_{0 \le s \le T } |\bun{s}{|}_{\hmath }^{p} \bigr) \le {C}_{1}(p) .
\end{equation}
and
\begin{equation} \label{E:V_estimate_bar_u_n}
  \sup_{n \ge 1 }  \bar{\e } \bigl[ \int_{0}^{T} \norm{ \bun{s}}{\vmath }{2} \, ds \bigr] \le {C}_{2}.
\end{equation}
By inequality \eqref{E:V_estimate_bar_u_n}, there exists a subsequence of $({\bar{u}}_{n})$, still denoted by $({\bar{u}}_{n})$, convergent weakly in the space ${L}^{2}([0,T]\times \bar{\Omega };\vmath )$. Since by 
\eqref{E:Z_cadlag_NS_bar_un_conv} ${\bar{u}}_{n} \to {u}_{\ast }$ in $\zcal $, we infer that 
${u}_{\ast } \in {L}^{2}([0,T]\times \bar{\Omega };\vmath )$, i.e. 
\begin{equation} \label{E:V_estimate_u_ast}
   \bar{\e } \Bigl[ \int_{0}^{T} \norm{{u}_{\ast }}{\vmath }{2} \, ds  \Bigr] < \infty .
\end{equation}  
Similarly, by inequality \eqref{E:H_estimate_bar_u_n} with $p:=2$ we can choose a further subsequence of
$({\bar{u}}_{n})$ convergent weak star in the space ${L}^{2}(\bar{\Omega };{L}^{\infty }(0,T;\hmath ))$, and using \eqref{E:Z_cadlag_NS_bar_un_conv}, deduce that 
\begin{equation} \label{E:H_estimate_u_ast}
   \e \Bigl[ \sup_{t \in [0,T]} {|{u}_{\ast }(t)|}_{\hmath }^{2} \Bigr] < \infty .
\end{equation}
\bf Step 1. \rm 
Let us fix $\varphi \in \umath $. Analogously to \cite{Brzezniak_Hausenblas_2010}, let us denote
\begin{align} 
 &  {\Lambda }_{n}({\bar{u}}_{n}, {\bar{\eta }}_{n},{\bar{W}}_{n}, \varphi) (t)
:=  \ilsk{{\bar{u}}_{n}(0)}{\varphi}{\hmath } \nonumber \\
 & - \int_{0}^{t} \dual{\Pn \acal {\bar{u}}_{n}(s)}{\varphi}{}  ds   
- \int_{0}^{t} \dual{{\bcal }_{n} ({\bar{u}}_{n}(s))}{\varphi}{}  ds 
  - \int_{0}^{t} \dual{\Pn \rcal {\bar{u}}_{n}(s)}{\varphi}{}  ds \nonumber \\
 & + \int_{0}^{t} \dual{\Pn f(s)}{\varphi}{}\, ds 
 + \int_{0}^{t} \int_{Y} \ilsk{\Pn F(s,{\bar{u}}_{n}({s}^{-});y)}{\varphi}{\hmath } \,  {\tilde{\bar{\eta }}}_{n} (ds,dy)
  \nonumber \\
 & 
 + \int_{0}^{t} \int_{Y\setminus {Y}_{0}} \ilsk{\Pn F(s,{\bar{u}}_{n}({s}^{});y)}{\varphi}{\hmath } 
 \, d\mu (y)ds
+ \Dual{\int_{0}^{t}\Pn G(s,{\bar{u}}_{n}(s))\,  d{\bar{W}}_{n}(s)}{\varphi}{} 
\label{E:K_n_bar_u_n}
\end{align}
and 
\begin{align} 
 &  \Lambda  ({u}_{\ast }, {\eta }_{\ast },{W}_{\ast }, \varphi) (t)
:=  \ilsk{{u}_{\ast }(0)}{\varphi}{\hmath } \nonumber \\
 &   - \int_{0}^{t} \dual{ \acal {u}_{\ast }(s)}{\varphi}{}  ds 
 - \int_{0}^{t} \dual{ B({u}_{\ast }(s))}{\varphi}{}  ds 
  - \int_{0}^{t} \dual{ \rcal {u}_{\ast }(s)}{\varphi}{}  ds \nonumber \\
 & + \int_{0}^{t} \dual{ f(s)}{\varphi}{}\, ds 
 + \int_{0}^{t} \int_{Y} \ilsk{ F(s,{u}_{\ast }({s}^{-});y)}{\varphi}{\hmath } 
 \,  \tilde{{\eta }}_{\ast  } (ds,dy) \nonumber \\
 & + \int_{0}^{t} \int_{Y\setminus {Y}_{0}} \ilsk{ F(s,{u}_{\ast }({s}^{});y)}{\varphi}{\hmath } 
 \, d \mu (y)ds
+  \Dual{\int_{0}^{t}G(s,{u}_{\ast }(s))\,  d{W}_{\ast }(s)}{\varphi}{}  ,
\quad t \in [0,T] .
\label{E:K_u*}
\end{align}
We will prove the following lemma.

\pagebreak

\begin{lemma}  \label{L:convergence_existence}
For all $\varphi \in \umath $
\begin{itemize}
\item[(a)] $\lim_{n\to \infty } \bar{\e } \bigl[ \int_{0}^{T} {|\ilsk{{\bar{u}}_{n}(t)-{u}_{\ast }(t)}{\varphi}{\hmath }|}^{2} \, dt \bigr] =0 $,
\item[(b)] $\lim_{n\to \infty } \bar{\e } \bigl[ {|\ilsk{{\bar{u}}_{n}(0)-{u}_{\ast }(0)}{\varphi}{\hmath }|}^{2}  \bigr] =0 $,
\item[(c)] $\lim_{n\to \infty } \bar{\e } \bigl[ \int_{0}^{T} 
 \bigl| \int_{0}^{t}\dual{{P}_{n} \acal {\bar{u}}_{n}(s)-\acal {u}_{\ast }(s)}{\varphi}{} \,ds \bigr| \, dt \bigr] =0 $,
\item[(d)] $\lim_{n\to \infty } \bar{\e } \bigl[ \int_{0}^{T} 
 \bigl| \int_{0}^{t}\dual{{P}_{n} {B}_{n} ({\bar{u}}_{n}(s))- B({u}_{\ast }(s))}{\varphi}{} \,ds \bigr| \, dt \bigr] =0 $, 
\item[(e)] $\lim_{n\to \infty } \bar{\e } \bigl[ \int_{0}^{T} 
 {\bigl| \int_{0}^{t}\dual{{P}_{n} \rcal {\bar{u}}_{n}(s)-\rcal {u}_{\ast }(s)}{\varphi}{} \,ds \bigr| }^{2} 
   \, dt \bigr] =0 $, 
\item[(f)]  $\lim_{n\to \infty } \bar{\e } \bigl[ \int_{0}^{T} 
 \bigl| \int_{0}^{t} \int_{Y} \dual{\Pn F(s,{\bar{u}}_{n}({s}^{-});y) -F(s,{u}_{\ast }({s}^{-});y)}{\varphi}{}  \, d\mu (y) ds {\bigr| }^{2} \, dt \bigr] =0 $,    
\item[(g)] $\lim_{n\to \infty } \bar{\e } \bigl[ \int_{0}^{T} 
 \bigl| \int_{0}^{t} \int_{Y} \dual{\Pn F(s,{\bar{u}}_{n}({s}^{-});y) -F(s,{u}_{\ast }({s}^{-});y)}{\varphi}{}  \, {\tilde{\eta }}_{\ast } (ds,dy) {\bigr| }^{2} \, dt \bigr] =0 $,
\item[(h)] $\lim_{n\to \infty } \bar{\e } \bigl[ \int_{0}^{T} 
 \bigl| \dual{ \int_{0}^{t} [\Pn G(s,{\bar{u}}_{n}(s)) - G(s,{u}_{\ast }(s))] \, d {W}_{\ast }(s) }{\varphi}{}
 {\bigr| }^{2} \, dt \bigr] =0 $.
\end{itemize}
\end{lemma}

\bigskip  \noindent
\begin{proof} Let us fix $\varphi \in \umath $.
\bf Ad. (a). \rm   
Since by (\ref{E:Z_cadlag_NS_bar_un_conv})
$ {\bar{u}}_{n} \to {u}_{\ast } $ in $\dmath ([0,T];{\hmath }_{w})$ $\bar{\p }$-a.s., 
$\ilsk{{\bar{u}}_{n}(\cdot )}{\varphi}{\hmath } \to \ilsk{{u}_{\ast }(\cdot )}{\varphi}{\hmath } $
in $\dmath ([0,T];\rzecz )$ $\bar{\p }$-a.s. Hence, in particular for almost all $t \in [0,T]$
\[
  \lim_{n \to \infty } \ilsk{{\bar{u}}_{n}(t)}{\varphi}{\hmath } = \ilsk{{u}_{\ast }(t)}{\varphi}{\hmath },
  \qquad  \mbox{ $\bar{\p }$-a.s. }
\]
Since by \eqref{E:H_estimate_bar_u_n}, 
 $\sup_{t\in [0,T]} {|{\bar{u}}_{n}(t)|}_{\hmath }^{2} < \infty $, $\bar{\p }$-a.s.,
using the Dominated Convergence Theorem we infer that 
\begin{equation} \label{E:Vitali_{u}_{n}_conv}
  \lim_{n \to \infty } \int_{0}^{T} {| \ilsk{{\bar{u}}_{n}(t) -{u}_{\ast }(t)}{\varphi}{\hmath } |}^{2} \, dt =0  \qquad  \mbox{ $\bar{\p }$-a.s. }.
\end{equation} 
Moreover,  by the H\"{o}lder inequality and \eqref{E:H_estimate_bar_u_n}
for every $n \in \nat$ and every $r\in \bigl( 1, 1+ \frac{\gamma }{2}\bigr] $
\begin{align}
&  \bar{\e }\Bigl[ \Bigl|  \int_{0}^{T} {|{\bar{u}}_{n}(t) -{u}_{\ast }(t) | }_{\hmath }^{2} \, dt 
 {\Bigr| }^{r}\Bigr] 
  \le c \bar{\e }\Bigl[  \int_{0}^{T} \bigl( {| {\bar{u}}_{n}(t)  | }_{\hmath }^{2r}
+ {| {u}_{\ast }(t)| }_{\hmath }^{2r} \bigr) \, dt \Bigr]   \le \tilde{c} {C}_{1}(2r) , 
 \label{E:Vitali_{u}_{n}_est}
\end{align} 
where $c, \tilde{c}>0$ are some constants .
The assertion (a) follows now from \eqref{E:Vitali_{u}_{n}_conv}, \eqref{E:Vitali_{u}_{n}_est} and the Vitali Theorem. 

 \bigskip  \noindent
\bf Ad. (b). \rm 
Since by \eqref{E:Z_cadlag_NS_bar_un_conv} ${\bar{u}}_{n} \to {u}_{\ast }$ in $\dmath (0,T;{\hmath }_{w})$ $\bar{\p }$-a.s. and ${u}_{\ast }$ is right-continuous at $t=0$, we infer that
$ \ilsk{{\bar{u}}_{n}(0)}{\varphi}{\hmath } \to \ilsk{{u}_{\ast }(0)}{\varphi}{\hmath } $, $\bar{\p }$-a.s.
By \eqref{E:H_estimate_bar_u_n} assertion (b) follows from the Vitali Theorem.

\bigskip  \noindent
\bf Ad. (c). \rm 
By \eqref{E:Z_cadlag_NS_bar_un_conv} ${\bar{u}}_{n} \to {u}_{\ast }$ in ${L}_{w}^{2}(0,T;\vmath )$, 
$\bar{\p }$-a.s.
Moreover, since $\varphi \in \umath $,  $\Pn \varphi \to \varphi $ in $\vmath $, see Lemma \ref{L:P_n|U} (c) in Section \ref{S:Funct_anal}.
Thus by   \eqref{E:A_acal_rel} we infer that $\bar{\p }$-a.s.
\begin{align} 
 &\lim_{n\to \infty } \int_{0}^{t} \dual{\Pn \acal {\bar{u}}_{n}(s)}{\varphi}{} \, ds 
 = \lim_{n \to \infty } \int_{0}^{t} \dirilsk{{\bar{u}}_{n}(s)}{\Pn \varphi}{} \, ds 
 = \int_{0}^{t} \dirilsk{ {u}_{\ast }(s)}{\varphi}{} \, ds 
 = \int_{0}^{t} \dual{ {u}_{\ast }(s)}{\varphi}{} \, ds \qquad \mbox{} 
 \label{E:Vitali_{A}_{n}_conv}
\end{align}
By   \eqref{E:A_acal_rel}, the H\"{o}lder inequality and \eqref{E:V_estimate_bar_u_n} we have the following inequality
 for all $t \in [0,T]$ and  $n \in \nat $
\begin{eqnarray} 
& &  \bar{\e }\Bigl[ \Bigl | \int_{0}^{t}  \dual{ \Pn \acal {\bar{u}}_{n}(s)}{\varphi}{} \, ds {\Bigr| }^{2} \Bigr]
  = \bar{\e } \Bigl[ \Bigl | \int_{0}^{t} \dirilsk{{\bar{u}}_{n}(s)}{\Pn \varphi}{} \, ds {\Bigr| }^{2} \Bigr] \nonumber \\
& & \le c \, \| { \varphi}{\| }_{\umath }^{2} \,  \bar{\e } 
 \Bigl[  \int_{0}^{T} \norm{ {\bar{u}}_{n}(s)}{\vmath }{2} \, ds  \Bigr] 
  \le \tilde{c}{C}_{2} 
  \label{E:Vitali_{A}_{n}_est}
\end{eqnarray}
for some constants $c, \tilde{c}>0$. 
By \eqref{E:Vitali_{A}_{n}_conv}, \eqref{E:Vitali_{A}_{n}_est} and the Vitali Theorem we infer that for all $t \in [0,T]$
\[
  \lim_{n\to \infty } \bar{\e } \Bigl[ \bigl| \int_{0}^{t} 
\dual{\Pn \acal {\bar{u}}_{n}(s)-\acal {u}_{\ast }(s)}{\varphi}{} \, ds 
 {\bigr| }^{} \Bigr] =0.
\]
Hence assertion (c) follows from \eqref{E:V_estimate_bar_u_n} and the Dominated Convergence Theorem.

\bigskip  \noindent
\bf Ad. (d). \rm 
Let us move to the nonlinear term. Here assumption (B.5) will be very important.
Since by \eqref{E:V_estimate_bar_u_n} and \eqref{E:norm_V} the sequence $({\bar{u}}_{n})$ is bounded in ${L}^{2}(0,T;\hmath )$ and 
by \eqref{E:Z_cadlag_NS_bar_un_conv} ${\bar{u}}_{n} \to {u}_{\ast }$ in 
${L}^{2}(0,T;{L}^{2}_{loc}(\ocal )) $, $\bar{\p }$-a.s.,
by assumption (B.5) (see Remark \ref{R:B.5_comment}) we infer that $\bar{\p }$-a.s.
for all $t \in [0,T]$ and all $\varphi \in {\vmath }_{\ast}$
\[
 \lim_{n\to \infty } \int_{0}^{t} \dual{ \bcal ({\bar{u}}_{n}(s)) - \bcal ({u}_{*}(s))}{\varphi}{} \, ds =0 .
\]
It is easy to see that for sufficiently large $n \in \nat $, 
$
     {\bcal }_{n} ({\bar{u}}_{n}(s)) = \Pn \bcal ({\bar{u}}_{n}(s)) , \ s \in [0,T].
$
Moreover, if $\varphi \in \umath $ then by Lemma \ref{L:P_n|U} (c), $\Pn \varphi \to \varphi $ in ${\vmath }_{\ast }$. Since $\umath  \subset {\vmath }_{\ast }$, we infer that for all $\varphi \in \umath $ and all $t \in [0,T]$
\begin{equation} \label{E:Vitali_{B}_{n}_conv}
 \lim_{n\to \infty } \int_{0}^{t} \dual{ {\bcal }_{n} ({\bar{u}}_{n}(s)) - \bcal ({u}_{\ast }(s))}{\varphi}{} \, ds =0 
 \qquad \bar{\p }\mbox{-a.s}.
\end{equation}
By the H\"{o}lder inequality, \eqref{E:estimate_B_ext} and  \eqref{E:H_estimate_bar_u_n} we obtain for all 
 $t \in [0,T]$,  $r\in \bigl(0,\frac{\gamma }{2}\bigr] $ and  $n \in \nat $
\begin{align} 
&  \bar{\e } \Bigl[ \Bigl| \int_{0}^{t} \dual{{\bcal }_{n} ({\bar{u}}_{n}(s)) }{\varphi}{} \, ds {\Bigr| }^{1+r}  \Bigr] 
\le \bar{\e }  \Bigl[ \Bigl( \int_{0}^{t} {|{\bcal }_{n} ({\bar{u}}_{n}(s))| }_{{\vmath }_{\ast }^{\prime }}  \norm{\varphi}{{\vmath }_{\ast }}{} \, ds {\Bigr) }^{1+r}  \Bigr] \nonumber \\
&  \le ({c}_{2}\norm{\varphi}{{\vmath }_{\ast }}{} {)}^{1+r} \, {t}^{r} \, 
 \e \Bigl[ \int_{0}^{t} {| {\bar{u}}_{n}(s) | }_{\hmath }^{2+2r} \, ds \Bigr]
 \le \tilde{C} \bar{\e } \bigl[ \sup_{s\in [0,T]} {|{\bar{u}}_{n}(s) | }_{\hmath }^{2+2r}   \bigr] 
  \le \tilde{C} {C}_{1}(2+2r) .  \label{E:Vitali_{B}_{n}_est}
\end{align}
In view of \eqref{E:Vitali_{B}_{n}_conv} and \eqref{E:Vitali_{B}_{n}_est}, by the Vitali Theorem we 
obtain for all $t\in [0,T]$
\begin{equation} \label{E:Lebesque_{B}_{n}_conv}
  \lim_{n \to \infty } \bar{\e } \Bigl[ \Bigl| \int_{0}^{t} 
 \dual{\Bn ({\bar{u}}_{n}(s)) - B({u}_{\ast }(s))}{\varphi}{} \, ds {\Bigr| }^{}  \Bigr] =0 .
\end{equation}
Since by \eqref{E:H_estimate_bar_u_n} for all $t \in [0,T]$ and all $n \in \nat $
\[
   \bar{\e } \Bigl[ \bigl| \int_{0}^{t} \dual{\Bn ({\bar{u}}_{n}(s)) }{\varphi}{} \, ds {\bigr| }^{} \Bigr] \le c \bar{\e } \bigl[ \sup_{s\in [0,T]} {|{\bar{u}}_{n}(s))| }_{\hmath }^{2}   \bigr]
 \le c {C}_{1}(2) ,
\]
where $c>0$ is a certain constant,
by \eqref{E:Lebesque_{B}_{n}_conv} and the Dominated Convergence Theorem, we infer that 
assertion (d) holds.

\bigskip \noindent
\bf Ad. (e).  \rm 
Since by  \eqref{E:Z_cadlag_NS_bar_un_conv} ${\bar{u}}_{n} \to {u}_{\ast }$ in ${L}^{2}_{w}(0,T;\vmath )$ and the embedding $\vmath \subset \hmath $ is continuous, 
${\bar{u}}_{n} \to {u}_{\ast }$ in ${L}^{2}_{w}(0,T;\hmath )$, $\bar{\p }$-a.s.
Furthemore, since $\rcal : \hmath \to {\vmath }^{\prime }$ is linear and continuous, 
$\rcal {\bar{u}}_{n} \to \rcal {u}_{\ast }$ in ${L}^{2}_{w}(0,T;{\vmath }^{\prime })$, $\bar{\p }$-a.s.
Since moreover by Lemma \ref{L:P_n|U} (c) $\Pn \varphi \to \varphi $ in $\vmath $, we infer that
\begin{equation} 
\lim_{n\to \infty } \int_{0}^{t} \dual{\Pn \rcal {\bar{u}}_{n}(s)}{\varphi}{} \, ds 
 = \lim_{n \to \infty } \int_{0}^{t} \dual{\rcal {\bar{u}}_{n}(s)}{\Pn \varphi}{} \, ds \nonumber \\
= \int_{0}^{t} \dual{\rcal {u}_{\ast }(s)}{\varphi}{} \, ds 
  \quad \mbox{$\bar{\p }$-a.s.} 
 \label{E:Vitali_{R}_{n}_conv}
\end{equation}
By assumption (R.1), Lemma \ref{L:P_n|U} (c) and \eqref{E:H_estimate_bar_u_n} we have the following inequalities  for all $r \in (0,\gamma )$, $t \in [0,T]$ and  $n \in \nat $
\begin{align} 
&   \bar{\e }\Bigl[ \bigl | \int_{0}^{t}  \dual{ \Pn \rcal {\bar{u}}_{n}(s)}{\varphi}{} \, ds {\bigr| }^{2+r} \Bigr]
  = \bar{\e } \Bigl[ \bigl | \int_{0}^{t} \dual{{\bar{u}}_{n}(s)}{\Pn \varphi}{} \, ds {\bigr| }^{2+r} \Bigr] \nonumber \\
&  \le c \, \norm{ \varphi}{\vmath }{2+r} \, 
  \bar{\e } \Bigl[  \sup_{s\in [0,T]} {| {\bar{u}}_{n}(s)|}_{\hmath }^{2+r}   \Bigr] 
  \le \tilde{c}{C}_{1}(2+r) 
  \label{E:Vitali_{R}_{n}_est}
\end{align}
for some constants $c, \tilde{c}>0$. 
Therefore by \eqref{E:Vitali_{R}_{n}_conv}, \eqref{E:Vitali_{R}_{n}_est} and the Vitali Theorem we infer that for all $t \in [0,T]$
\[
  \lim_{n\to \infty } \bar{\e } \Bigl[ \bigl| \int_{0}^{t} 
\dual{\Pn \rcal {\bar{u}}_{n}(s)-\rcal {u}_{\ast }(s)}{\varphi}{} \, ds 
 {\bigr| }^{2} \Bigr] =0.
\]
Hence assertion (d) follows from  \eqref{E:H_estimate_bar_u_n} and the Dominated Convergence Theorem.

\bigskip  \noindent
\bf Ad. (f). \rm 
Assume that $\varphi \in \hmath $.
For all $t \in [0,T]$ we have
\begin{align*}
&  \int_{0}^{t} \int_{Y} {| \dual{F(s, {\bar{u}}_{n}({s}^{-});y)- F(s,{u}_{\ast}({s}^{-});y)}{\varphi}{} |}^{2} \, d\mu (y)ds  \nonumber \\  
& = \int_{0}^{t} \int_{Y} {| {\tilde{F}}_{\varphi}({\bar{u}}_{n})(s,y) - {\tilde{F}}_{\varphi}({u}_{\ast })(s,y)| }^{2} \, d\mu (y)ds  
 \le \norm{{\tilde{F}}_{\varphi}({\bar{u}}_{n}) - {\tilde{F}}_{\varphi}({u}_{\ast })}{{L}^{2}([0,T]\times Y;\rzecz )}{2}, 
\end{align*}
where ${\tilde{F}}_{\varphi}$ is the mapping defined by \eqref{E:F**}.  
Since by \eqref{E:Z_cadlag_NS_bar_un_conv}
${\bar{u}}_{n} \to {u}_{*}$ in ${L}^{2}(0,T;{L}^{2}_{loc}(\ocal )) $, $\bar{\p }$-a.s.,
by  assumption (F.3) we infer that for all $t\in [0,T]$ 
\begin{eqnarray}
& & \lim_{n \to \infty } \int_{0}^{t} \int_{Y} 
 {| \ilsk{F(s, {\bar{u}}_{n}({s}^{-});y)- F(s,{u}_{\ast }({s}^{-});y)}{\varphi}{\hmath } |}^{2} 
 \, d\mu (y)ds  =0. 
\label{E:Vitali_{noise}_{n}_conv}
\end{eqnarray}
Moreover, by inequality \eqref{E:F_linear_growth} in assumption (F.2) and by \eqref{E:H_estimate_bar_u_n}
for every $t \in [0,T]$ every $r \in \bigl[ 1,1+ \frac{\gamma }{2} \bigr] $ and every $n \in \nat $ the following inequalities hold
\begin{align} 
&  \bar{\e } \Bigl[ \bigl| \int_{0}^{t} \int_{Y} {| 
\ilsk{F(s, {\bar{u}}_{n}({s}^{-});y)- F(s,{u}_{\ast }({s}^{-});y)}{\varphi}{\hmath }|}^{2} 
 \, d\mu (y)ds {\bigr| }^{r} \Bigr] \nonumber \\
&  \le {2}^{r} {|\varphi|}_{H}^{2r}\bar{\e } \Bigl[ \bigl| \int_{0}^{t} \int_{Y} \bigl\{ 
 {| F(s, {\bar{u}}_{n}({s}^{-});y) | }_{\hmath }^{2}
 +{| F(s,{u}_{\ast }({s}^{-});y) | }_{\hmath }^{2} \bigr\} \, d\mu (y)ds {\bigr| }^{r} \Bigr] \nonumber \\
& \le {2}^{r} {C}_{2}^{r} {|\varphi|}_{H}^{2r}\, \bar{\e } \Bigl[ \bigl|  \int_{0}^{t} \bigl\{ 2+
  {| {\bar{u}}_{n}(s) |}_{\hmath }^{2} + {|{u}_{\ast }(s)|}_{\hmath }^{2}
\bigr\}  \, ds {\bigr| }^{r} \Bigr] 
 \le c \bigl( 1+ \bar{\e } \bigl[ \sup_{s \in [0,T]} 
 {|{\bar{u}}_{n}(s)| }_{\hmath }^{2r}\bigr] \bigr)
 \nonumber \\
& \le c (1+ {C}_{1}(2r))  
\label{E:Vitali_{noise}_{n}_est}
\end{align} 
for some constant $c>0$.
Thus by \eqref{E:Vitali_{noise}_{n}_conv}, \eqref{E:Vitali_{noise}_{n}_est} and the Vitali Theorem for all $t\in [0,T]$
\begin{eqnarray}
\lim_{n\to \infty } \bar{\e } \Bigl[  \int_{0}^{t} \int_{Y} 
{| \dual{F(s, {\bar{u}}_{n}({s}^{-});y)- F(s,{u}_{\ast }({s}^{-});y)}{\varphi}{}|}^{2} \, d\mu (y)ds  \Bigr] =0,
\quad  \varphi \in \hmath  . \label{E:Lebesgue_{noise}_{n}_conv_vcal}
\end{eqnarray} 
Moreover, since the restriction of $\Pn $ to the space $\hmath $ is the $\ilsk{\cdot }{\cdot }{\hmath }$-projection onto ${\hmath }_{n}$, see Section \ref{S:Funct_anal}, we infer that also
\begin{eqnarray}
\lim_{n\to \infty } \bar{\e } \Bigl[  \int_{0}^{t} \int_{Y} {| \dual{\Pn F(s, {\bar{u}}_{n}({s}^{-});y)- F(s,{u}_{\ast }({s}^{-});y)}{\varphi}{} |}^{2} \, d\mu (y)ds  \Bigr] =0,
\quad  \varphi \in \hmath  . \label{E:Lebesgue_{noise}_{n}_conv_H}
\end{eqnarray}
Since $\umath \subset \hmath $, \eqref{E:Lebesgue_{noise}_{n}_conv_H} holds for $\varphi \in \umath $.

\bigskip \noindent
\bf Ad. (g). \rm 
By \eqref{E:Lebesgue_{noise}_{n}_conv_H} and  by the properties of the integral with respect to the compensated Poisson random measure
and the fact that ${\bar{\eta }}_{n} = {\eta }_{\ast }$, we have
\begin{equation}  \label{E:Lebesque_{noise}_{n}_conv}
\lim_{n\to \infty } \bar{\e } \Bigl[ \Bigl|  \int_{0}^{t} \int_{Y}  \dual{ \Pn F(s, {\bar{u}}_{n}({s}^{-});y)- F(s,{u}_{\ast }({s}^{-});y)}{\varphi}{}  \, {\tilde{{\eta }}}_{\ast } (ds,dy)
{\Bigr| }^{2}  \Bigr]  =0.
\end{equation}
Moreover, by inequality \eqref{E:Vitali_{noise}_{n}_est} with $r:=1$
we obtain the following inequality
\begin{align} 
&  \bar{\e } \Bigl[ \bigl|  \int_{0}^{t} \int_{Y}  \dual{\Pn F(s, {\bar{u}}_{n}({s}^{-});y)- F(s,{u}_{\ast }({s}^{-});y)}{\varphi}{}  \, {\tilde{\eta }}_{\ast } (ds,dy)
{\bigr| }^{2}  \Bigr] \nonumber \\
&  =  \bar{\e } \Bigl[   \int_{0}^{t} \int_{Y} 
 {|\ilsk{\Pn F(s, {\bar{u}}_{n}({s}^{-});y)- F(s,{u}_{\ast }({s}^{-});y)}{\varphi}{\hmath } |}^{2} \, \mu (dy)ds  \Bigr] 
 \le c ( 1+ {C}_{1}(2)) .
 \label{E:Lebesque_{noise}_{n}_est}
\end{align}
Now assertion (g) follows from
 \eqref{E:Lebesque_{noise}_{n}_conv}, \eqref{E:Lebesque_{noise}_{n}_est} and the Dominated Convergence Theorem. 

\bigskip  \noindent
\bf Ad. (h). \rm 
Let us assume  that $\varphi \in \vmath $.
We have
\begin{align*} 
&  \int_{0}^{t}  {\| \dual{G(s, {\bar{u}}_{n}(s))- G(s,{u}_{\ast }(s))}{\varphi}{} 
  \| }^{2}_{\lhs ({Y}_{W};\rzecz )} \, ds  \\
 & = \int_{0}^{t}  {\|  {\tilde{G}}_{\varphi}({\bar{u}}_{n})(s) - {\tilde{G}}_{\varphi}({u}_{\ast })(s)
    \| }^{2}_{\lhs ({Y}_{W};\rzecz )} \, ds 
 \le \norm{{\tilde{G}}_{\varphi}({\bar{u}}_{n}) - {\tilde{G}}_{\varphi}({u}_{\ast })}{{L}^{2}([0,T];\lhs ({Y}_{W};\rzecz )) }{2} ,
 \end{align*}
where   ${\tilde{G}}_{\varphi}$ is the mapping defined by \eqref{E:G**}.
Since by \eqref{E:Z_cadlag_NS_bar_un_conv}
${\bar{u}}_{n} \to {u}_{\ast }$ in ${L}^{2}(0,T;{L}^{2}_{loc}(\ocal )) $, $\bar{\p }$-a.s.,
by the second part of assumption (G.3) we infer that for all $t\in [0,T]$ and all $\varphi \in \vmath $
\begin{align} 
&  \lim_{n \to \infty } \int_{0}^{t} {\| \dual{G(s,{\bar{u}}_{n}(s))- G(s,{u}_{\ast }(s))}{\varphi}{} 
  \| }^{2}_{\lhs ({Y}_{W};\rzecz )} \, ds =0  .
   \label{E:Vitali_{Gaussian}_{n}_conv}
\end{align}
Moreover, by \eqref{E:G*} and \eqref{E:H_estimate_bar_u_n} we see that
for every $t \in [0,T]$ every $r \in \bigl( 1,1+ \frac{\gamma }{2} \bigr] $ and every $n \in \nat $
\begin{align} 
& \bar{\e } \Bigl[  \bigl| \int_{0}^{t} {\| \dual{G(s, {\bar{u}}_{n}(s))- G(s,{u}_{\ast }(s))}{\varphi}{} 
  \| }^{2}_{\lhs ({Y}_{W};\rzecz )} \, ds {\bigr| }^{r} \Bigr]  \nonumber \\
&  \le  c \, \bar{\e } \Bigl[ \norm{\varphi}{\vmath }{2r} \cdot \int_{0}^{t} \bigl\{ 
  \norm{G(s, {\bar{u}}_{n}(s))}{\lhs ({Y}_{W};{\vmath }^{\prime })}{2r} 
  + \norm{G(s, {u}_{\ast }(s))}{\lhs ({Y}_{W};{\vmath }^{\prime })}{2r} \bigr\} \, ds   \Bigr] \nonumber \\
&  \le  {c}_{1} \, \bar{\e } \Bigl[ \int_{0}^{T} (1+ {|{\bar{u}}_{n}(s)|}_{\hmath }^{2r}
    + {|{u}_{\ast }(s)|}_{\hmath }^{2r}) \, ds  \Bigr] \nonumber  \\
&   \le  \tilde{c} \Bigl\{ 1+ \bar{\e } \Bigl[ \sup_{s\in [0,T]} {|{\bar{u}}_{n}(s)|}_{\hmath }^{2r}
    + \sup_{s\in [0,T]} {|{u}_{\ast }(s)|}_{\hmath }^{2r})  \Bigr] \Bigr\} \le \tilde{c} (1+2 {C}_{1}(2r))
 \label{E:Vitali_{Gaussian}_{n}_est}   
\end{align}
for some positive constants $c,{c}_{1},\tilde{c}$.
Thus by \eqref{E:Vitali_{Gaussian}_{n}_conv}, \eqref{E:Vitali_{Gaussian}_{n}_est} and the Vitali 
Theorem, we infer that
\begin{equation}
\lim_{n\to \infty } \bar{\e } \Bigl[  \int_{0}^{t}  
 {\| \dual{G(s, {\bar{u}}_{n}(s))- G(s,{u}_{\ast }(s))}{\varphi}{} \| }^{2}_{\lhs ({Y}_{W};\rzecz )} \, ds
  \Bigr] =0 \quad \mbox{for all } \varphi \in \vmath .  \label{E:Lebesgue_{Gaussian}_{n}_conv_V}
\end{equation}  
For every $\varphi \in \vmath  $ and every $s \in [0,T]$  we have
\begin{align*}
&   \dual{\Pn G(s, {\bar{u}}_{n}(s))- G(s,{u}_{\ast }(s))}{\varphi}{}
 = \dual{ G(s, {\bar{u}}_{n}(s))}{\Pn \varphi}{}
 - \dual{ G(s,{u}_{\ast }(s))}{\varphi}{} \\
&  = \dual{ G(s, {\bar{u}}_{n}(s))}{\Pn \varphi - \varphi}{}
 + \dual{ G(s, {\bar{u}}_{n}(s)) -G(s,{u}_{\ast }(s))}{\varphi}{} \\
&  \le \norm{G(s, {\bar{u}}_{n}(s))}{\lhs ({Y}_{W},{\vmath }^{\prime })}{}  
 \norm{\Pn \varphi -\varphi}{\vmath }{}
 + \dual{ G(s, {\bar{u}}_{n}(s)) -G(s,{u}_{\ast }(s))}{\varphi}{} . 
\end{align*}
Thus by inequality  \eqref{E:G*} in assumption (G.3) and by \eqref{E:H_estimate_bar_u_n} we obtain
\begin{align*}
&   \bar{\e } \Bigl[  \int_{0}^{t}  
 {\| \dual{\Pn G(s, {\bar{u}}_{n}(s))- G(s,{u}_{\ast }(s))}{\varphi}{} \| }^{2}_{\lhs ({Y}_{W};\rzecz )} \, ds  \Bigr] \\
 &  \le 2C \norm{\Pn \varphi -\varphi}{\vmath }{2} \, \bar{\e } \Bigl[  \int_{0}^{T} 
 \bigl( 1+ {|{\bar{u}}_{n}(s)|}_{\hmath }^{2} \bigr) \, ds \Bigr] 
 + 2 \bar{\e } \Bigl[  \int_{0}^{t}  
 {\| \dual{ G(s, {\bar{u}}_{n}(s))- G(s,{u}_{\ast }(s))}{\varphi}{} \| }^{2}_{\lhs ({Y}_{W};\rzecz )} \, ds  \Bigr]  \\
 &  \le 2CT(1+ {C}_{1}(2)) \norm{\Pn \varphi -\varphi}{\vmath }{2} 
 + 2 \bar{\e } \Bigl[  \int_{0}^{t}  
 {\| \dual{ G(s, {\bar{u}}_{n}(s))- G(s,{u}_{\ast }(s))}{\varphi}{} \| }^{2}_{\lhs ({Y}_{W};\rzecz )} \, ds  \Bigr] .
\end{align*}
Since $\umath \subset \vmath  $ and by Lemma \ref{L:P_n|U} (c), $\norm{\Pn \varphi -\varphi}{\vmath }{} \to 0 $ for all $\varphi\in \umath $, by 
\eqref{E:Lebesgue_{Gaussian}_{n}_conv_V} we infer that
\begin{equation*}
 \lim_{n\to \infty } \bar{\e } \Bigl[ \int_{0}^{t}  
 {\| \dual{\Pn G(s, {\bar{u}}_{n}(s))- G(s,{u}_{\ast }(s))}{\varphi}{} \| }^{2}_{\lhs ({Y}_{W};\rzecz )} \, ds  \Bigr]  =0 \quad \mbox{for all } \varphi \in \umath  .
\end{equation*} 
Hence by the properties of the It\^{o} integral we infer that for all $t \in [0,T]$ and all $\varphi \in \umath $
\begin{equation}  \label{E:Lebesque_{Gaussian}_{n}_conv}
\lim_{n\to \infty } \bar{\e } \Bigl[ \bigl| 
 \Dual{\int_{0}^{t}\bigl[ \Pn G(s, {\bar{u}}_{n}(s))- G(s,{u}_{\ast }(s)) \bigr] \, d{W}_{\ast }(s) }{\varphi}{}
 {\bigr| }^{2} \Bigr] =0 .
\end{equation}
Moreover, by the It\^{o} isometry, inequality  \eqref{E:G*} in assumption (G.3), and 
\eqref{E:H_estimate_bar_u_n} we have for all $t\in [0,T]$ and all $n \in \nat $ 
\begin{align} 
&  \bar{\e } \Bigl[ \bigl| 
 \Dual{\int_{0}^{t}\bigl[ \Pn G(s, {\bar{u}}_{n}(s))- G(s,{u}_{\ast }(s)) \bigr] \, d{W}_{\ast }(s) }{\varphi}{}
 {\bigr| }^{2} \Bigr] \nonumber \\
&   = \bar{\e } \Bigl[\int_{0}^{t} {\| \dual{\Pn G(s, {\bar{u}}_{n}(s))- G(s,{u}_{\ast }(s))}{\varphi}{} 
  \| }^{2}_{\lhs ({Y}_{W};\rzecz )} \, ds \Bigr] \nonumber \\
&  \le  c \Bigl\{ 1+ \bar{\e } \Bigl[ \sup_{s\in [0,T]}  {|{\bar{u}}_{n}(s)|}_{\hmath }^{2}
    + \sup_{s\in [0,T]} {|{u}_{\ast }(s)|}_{\hmath }^{2})  \Bigr] \Bigr\} \le c (1+2 {C}_{1}(2)) 
 \label{E:Lebesque_{Gaussian}_{n}_est}  
\end{align}
for some $c>0$.
By \eqref{E:Lebesque_{Gaussian}_{n}_conv}, \eqref{E:Lebesque_{Gaussian}_{n}_est} and the Dominated Convergence Theorem we infer that
\begin{eqnarray}
 \lim_{n\to \infty } \int_{0}^{T}\bar{\e } \Bigl[ \bigl| 
 \Dual{\int_{0}^{t}\bigl[ \Pn G(s, {\bar{u}}_{n}(s))- G(s,{u}_{\ast }(s)) \bigr] 
 \, d{W}_{\ast }(s) }{\varphi}{}
 {\bigr| }^{2} \Bigr] =0.  \label{E:{Gaussian}_{n}_conv}
\end{eqnarray}
This completes the proof of Lemma \ref{L:convergence_existence}.
\end{proof} 

\bigskip  \noindent
As a direct consequence of  Lemma \ref{L:convergence_existence} we get the following corollary

\bigskip 
\begin{cor}
We have 
\begin{equation} \label{E:bar_u_n_convergence}
 \lim_{n \to \infty } \norm{\ilsk{{\bar{u}}_{n}(\cdot )}{\varphi}{\hmath } 
 - \ilsk{{u}_{\ast }(\cdot )}{\varphi}{\hmath } }{{L}^{2}([0,T]\times \bar{\Omega })}{} =0 
\end{equation}
and
\begin{equation} \label{E:K_n_bar_u_n_convergence}
     \lim_{n \to \infty } \norm{  {\Lambda  }_{n}({\bar{u}}_{n}, {\bar{\eta }}_{n},{\bar{W}}_{n},\varphi)
 - \Lambda  ({u}_{\ast }, {\eta }_{\ast },{W}_{\ast }, \varphi) }{{L}^{1}([0,T]\times \bar{\Omega })}{} =0 .
\end{equation}
\end{cor}

\bigskip  
\begin{proof}
Assertion  (\ref{E:bar_u_n_convergence}) follows form the  equality 
\[
 \norm{\ilsk{{\bar{u}}_{n}(\cdot )}{\varphi}{\hmath } 
 - \ilsk{{u}_{\ast }(\cdot )}{\varphi}{\hmath } }{{L}^{2}([0,T]\times \bar{\Omega })}{2} \\
=  \bar{\e } \Bigl[ \int_{0}^{T} 
{| \ilsk{{\bar{u}}_{n}(t) -{u}_{\ast }(t)}{\varphi}{\hmath } |}^{2} \, dt \Bigr] 
\]
and Lemma \ref{L:convergence_existence}  (a).
Let us move to  \eqref{E:K_n_bar_u_n_convergence}.
Note that by the Fubini Theorem, we have
\begin{align*}
&  \norm{  {\Lambda  }_{n}({\bar{u}}_{n}, {\bar{\eta }}_{n},{\bar{W}}_{n},\varphi)
 - \Lambda  ({u}_{\ast }, {\eta }_{\ast },{W}_{\ast }, \varphi) }{{L}^{1}([0,T]\times \bar{\Omega })}{} \\
& = \int_{0}^{T} \bar{\e } \bigl[ {| {\Lambda  }_{n}({\bar{u}}_{n}, {\bar{\eta }}_{n},{\bar{W}}_{n},\varphi)(t)
 - \Lambda  ({u}_{\ast }, {\eta }_{\ast },{W}_{\ast } ,\varphi)(t) |}^{}\, \bigr] dt .
 \end{align*}
To prove  \eqref{E:K_n_bar_u_n_convergence} it is sufficient to note that 
by Lemma \ref{L:convergence_existence} (b)-(g),  each  term on the right hand side of
\eqref{E:K_n_bar_u_n} tends at least in ${L}^{1}([0,T]$ $\times \bar{\Omega })$ to the corresponding term in \eqref{E:K_u*}. 
\end{proof}
\noindent

\bigskip  \noindent
\bf Step 2. \rm Since ${u}_{n}$ is a solution of the Galerkin equation, for all $t\in [0,T]$
\[
   \ilsk{{u}_{n}(t)}{\varphi}{\hmath } = {\Lambda  }_{n} ({u}_{n},{\eta }_{n},{W}_{n},\varphi)(t) , \qquad \p \mbox{-a.s.}
\]
In particular,
\[
  \int_{0}^{T} \e \bigl[ {|  \ilsk{\un (t)}{\varphi}{\hmath } 
 - {\Lambda  }_{n} ({u}_{n},{\eta }_{n},{W}_{n},\varphi)(t) |}^{} \, \bigr] \, dt  =0.
\]
Since $\lcal ({u}_{n},{\eta }_{n},{W}_{n}) = \lcal ({\bar{u}}_{n},{\bar{\eta }}_{n},{\bar{W}}_{n})$,
\[
  \int_{0}^{T} \bar{\e } \bigl[ {| \ilsk{\bun{t}}{\varphi}{\hmath } 
- {\Lambda  }_{n} ({\bar{u}}_{n},{\bar{\eta } }_{n},{\bar{W}}_{n},\varphi)(t) |}^{} \, \bigr] \, dt  =0.
\]
Moreover, by \eqref{E:bar_u_n_convergence} and \eqref{E:K_n_bar_u_n_convergence}
\[
\int_{0}^{T} \bar{\e } \bigl[ {| 
\ilsk{{u}_{\ast }(t)}{\varphi}{\hmath } - {\Lambda  }_{} ({u}_{\ast },{\eta }_{\ast },{W}_{\ast },\varphi)(t) |}^{} \, \bigr] \, dt  =0.
\]
Hence for $l$-almost all $t \in [0,T]$ and $\bar{\p }$-almost all $\omega \in \bar{\Omega }$
\[
\ilsk{{u}_{\ast }(t)}{\varphi}{\hmath } - {\Lambda  }_{} ({u}_{\ast},{\eta }_{\ast },{W}_{\ast },\varphi)(t)=0, 
\]
i.e. for $l$-almost all $t \in [0,T]$ and $\bar{\p }$-almost all $\omega \in \bar{\Omega }$
\begin{align}
& \ilsk{{u}_{\ast }(t)}{\varphi}{\hmath } 
+\int_{0}^{t} \dual{\acal {u}_{\ast }(s)}{\varphi}{} \, ds
+ \int_{0}^{t} \dual{B({u}_{\ast }(s))}{\varphi}{} \, ds 
+\int_{0}^{t} \dual{\rcal {u}_{\ast }(s)}{\varphi}{} \, ds 
  \nonumber \\
 & =  \ilsk{{u}_{\ast }(0)}{\varphi}{\hmath }
+\int_{0}^{t} \dual{f(s)}{\varphi}{}\, ds  
 + \int_{0}^{t} \int_{{Y}_{0}} \ilsk{F(s,{u}_{\ast }(s);y)}{\varphi}{\hmath } \,  {\tilde{\eta }}_{\ast } (ds,dy) \nonumber \\
& + \int_{0}^{t} \int_{Y\setminus {Y}_{0}} \ilsk{F(s,{u}_{\ast }(s);y)}{\varphi}{\hmath } \, {\eta }_{\ast } (ds,dy) 
 + \Dual{ \int_{0}^{t} G(s,{u}_{\ast }(s)) \, d {W}_{\ast }(s) }{\varphi}{} . \label{E:solution}
\end{align}
Since ${u}_{\ast }$ is $\zcal $-valued random variable, in particular ${u}_{\ast }\in \dmath ([0,T];{\hmath }_{w})$, i.e. ${u}_{*}$ is weakly c\`{a}dl\`{a}g, 
we infer that 
equality \eqref{E:solution} holds  for all $t\in [0,T]$ and all $\varphi \in \umath $.
Since $\umath $ is dense in $\vmath $,   equality \eqref{E:solution} holds for all $\varphi \in \vmath $, as well.
Putting $\bar{\mathfrak{A}}:=(\bar{\Omega }, \bar{\fcal },\bar{\p }, \bar{\fmath })$,
 $\bar{\eta }:= {\eta }_{\ast }$, $\bar{W}:= {W}_{\ast }$ and $\bar{u}:= {u}_{\ast }$, we infer that the system 
$(\bar{\mathfrak{A}}, \bar{\eta },\bar{W}, \bar{u})$ is a martingale solution of the equation 
\eqref{E:equation}. 
By \eqref{E:V_estimate_u_ast} and \eqref{E:H_estimate_u_ast} the process $\bar{u}$ satisfies inequality
\eqref{E:u_estimates}.
The proof of Theorem \ref{T:existence} is thus complete.
\qed

\pagebreak

\section{Applications} \label{S:Applications}

\noindent
In this section $\ocal $ is an open connected possibly unbounded subset of $\rd $, $d=2,3$, with smooth boundary $\partial \ocal $.

\subsection{\bf Stochastic Navier-Stokes equations\rm }

\noindent
Let us consider the stochastic Navier-Stokes equations
\begin{equation} \label{E:NS}
\begin{split}
   & du(t) + (u \cdot \nabla ) u 
  - \frac{1}{Re} \Delta u  + \nabla p =  {f}_{}(t)  
    + \int_{{Y}_{0}} {F}_{}(t,u({t}^{-});y) {\tilde{\eta }}_{} (dt,dy),  \\
    & \qquad  + \int_{Y \setminus {Y}_{0}} {F}_{}(t,u({t}^{-});y) {\eta }_{} (dt,dy)
  +  {G}_{}(t,u(t)) \, d {W}_{}(t) \\
  & u(0) = {u}_{0}
\end{split}  
\end{equation}
in  $[0,T] \times \ocal $, with the incompressibility condition 
\begin{equation} \label{E:NS_incompressibility}
  \diver u =0 , 
\end{equation}
the homogeneous boundary condition 
\begin{equation} \label{E:NS_boundary}
{u}_{| \partial \ocal }=0.
\end{equation}
and the initial condition $u(0)={u}_{0}$.
In this problem $u=u(t,x)=({u}_{1}(t,x), ...{u}_{d}(t,x))$ and $p=p(t,x)$ 
represent the velocity and the pressure of the fluid. In equation \eqref{E:NS} $Re >0$
is the Reynolds number related to the kinematic viscosity (we may put $Re:=1$).
Furthermore, ${f}_{}$ stands for the deterministic
external forces, ${G}_{}(t,u)\, d {W}_{}(t)$, where ${W}_{}$ is a cylindrical  Wiener process on a Hilbert space ${Y}_{W}^{}$,
$ \int_{{Y}_{0}} {F}_{}(t,u({t}^{-});y) {\tilde{\eta }}_{} (dt,dy) $
and $ \int_{Y \setminus {Y}_{0}} {F}_{}(t,u({t}^{-});y) {\eta }_{} (dt,dy) $, where ${\eta }_{}$ is a time-homogeneous Poisson random measure on a measurable space $({Y}_{},{\ycal }_{})$ and 
${Y}_{0} \in \ycal $ is such that $\mu (Y \setminus {Y}_{0})<\infty $, 
stands for the random forces. Processes ${W}_{}$ and ${\eta }_{}$ are assumed to be independent.
Let us recall the functional setting of the problem 
\eqref{E:NS}-\eqref{E:NS_boundary}, see e.g. Temam \cite{Temam79}.

\bigskip  \noindent
\bf Function spaces. \rm 
Let us recall basic spaces used in the theory of Navier-Stokes equations.
We will denote them with subscript $0$.
\begin{align}
 &  {\vcal }_{0}:= {\ccal }^{\infty }_{c} (\ocal , \rd ) \cap \{ \diver = 0 \}  ,\label{vcal_0}\\
 &  {H}_{0} := \text{the closure of ${\vcal }_{0}$ in ${L}^{2}(\ocal , \rd )$} ,   \label{H_0} \\
 &  {V}_{0} := \text{the closure of ${\vcal }_{0}$ in ${H}^{1}(\ocal , \rd )$} .    \label{V_0}
\end{align} 
In the space ${H}_{0}$ we consider the inner product and the norm inherited from ${L}^{2}(\ocal , \rd )$ and 
denote them by $\ilsk{\cdot }{\cdot }{0}$ and $|\cdot {|}_{0}$, respectively, i.e.
\begin{equation} \label{E:il_sk_H_0}
\ilsk{u}{v}{0}:= \ilsk{u}{v}{{L}^{2}} , \qquad
 {|u|}_{0} := \norm{u}{{L}^{2}}{} , \qquad u, v \in {H}_{0}.
\end{equation}
In the space ${V}_{0}$ we consider the inner product 
\[
    \ilsk{u}{v}{{V}_{0}}:= \ilsk{u}{v}{0} +  \dirilsk{u}{v}{0},
\]
where
\begin{equation} \label{E:il_sk_Dirichlet}
    \dirilsk{u}{v}{0} := \frac{1}{Re}  \ilsk{\nabla u}{\nabla v}{{L}^{2}}
   =\frac{1}{Re}  \sum_{i=1}^{d}\int_{\ocal }\frac{\partial u}{\partial {x}_{i}} \cdot \frac{\partial v}{\partial {x}_{i}} \, dx , 
   \quad u,v \in  {V}_{0}   
\end{equation}
and the norm 
$
  \norm{u}{{V}_{0}}{2} := \dirilsk{u}{u}{{V}_{0}} = \norm{u}{{L}^{2}}{2} 
  + \frac{1}{Re}  \norm{\nabla u}{{L}^{2}}{2}.
$

\bigskip  \noindent
\bf The operator ${\acal }_{0}$. \rm 
We define the operator ${\acal }_{0}:{V}_{0} \to {V}_{0}'$ by setting
\begin{equation}  \label{E:Acal_0}
    \dual{{\acal }_{0}u}{v}{} =   \dirilsk{u}{v}{0} , \qquad u,v \in {V}_{0}.
\end{equation}
\bf The form $b$\rm .
Let us consider the following tri-linear form, see \cite{Temam79}.
\begin{equation}  \label{E:form_b}
     b(u,w,v ) = \int_{\ocal  }\bigl( u \cdot \nabla w \bigr) v \, dx 
   = \sum_{i=1}^{d} \int_{\ocal } {u}_{i} \frac{\partial w}{\partial {x}_{i}} \, v \, dx ,
\end{equation}
where $u:\ocal \to \rd $, $w,v: \ocal \to {\rzecz }^{{d}_{1}}$ and ${d}_{1}\in \nat $.
(We will consider the cases when ${d}_{1}=d$ or ${d}_{1}=1$.)
We  recall basic properties of the form $b$ in some Sobolev spaces. We will use them also in the magneto-hydrodynamic equations and in the Boussinesq equations.
Using the H\"{o}lder inequality and the Sobolev embedding Theorem, see \cite{Adams},
we obtain the following inequalities
\begin{align} 
    |b(u,w,v )| 
 &\le {\| u\| }_{{L}^{4}} \norm{\nabla w }{{L}^{2}}{} {\| v\| }_{{L}^{4}}  \notag 
 \notag  \\
 &    \le c \norm{u }{{H}^{1}}{} \norm{w }{{H}^{1}}{} \norm{v }{{H}^{1}}{} , 
\qquad u \in {H}^{1}(\ocal , \rd ), \quad v,w \in {H}^{1}(\ocal ;{\rzecz }^{{d}_{1}}) ,  
\label{E:b_estimate_H^1}
\end{align}
where $c$ is a positive constant.
Thus, the form $b$ is continuous on ${H}^{1}(\ocal ;\rd )\times {H}^{1}(\ocal ;{\rzecz }^{{d}_{1}})\times {H}^{1}(\ocal ;{\rzecz }^{{d}_{1}})$, see  Temam \cite{Temam95} and \cite{Temam79}.
If we define a bilinear map $B$ by $B(u,w):=b(u,w,\cdot )$, then by \eqref{E:b_estimate_H^1}, we infer that
$B(u,w) \in {H}^{-1}(\ocal ;{\rzecz }^{{d}_{1}})$ and
\begin{equation}
 {|B(u,w)|}_{{H}^{-1}} \le c \norm{u }{{H}^{1}}{} \norm{w }{{H}^{1}}{},
 \qquad u \in {H}^{1}(\ocal , \rd ) , \quad w \in {H}^{1}(\ocal ;{\rzecz }^{{d}_{1}}).
 \label{E:B_estimate_H^1}
\end{equation} 
(The symbol ${H}^{-1}(\ocal ;{\rzecz }^{{d}_{1}})$ stands for the dual space of 
${H}^{1}(\ocal ;{\rzecz }^{{d}_{1}})$.)
Furthermore, if $\diver u =0 $, then
\begin{align*}
 \sum_{i=1}^{n} \frac{\partial }{\partial {x}_{i}} ({u}_{i}w )
 & = \biggl( \sum_{i=1}^{n} \frac{\partial {u}_{i}}{\partial {x}_{i}} \biggr) w
      + \sum_{i=1}^{n} {u}_{i} \frac{\partial w}{\partial {x}_{i}}
   = (\diver u ) w + u \cdot \nabla w =  u \cdot \nabla w  .
\end{align*}
If moreover $(u \cdot n {)}_{|\partial \ocal } =0 $, then by the integration by parts formula, see \cite{Temam79},
\begin{align*}
b(u,w, v ) &= \int_{\ocal } \bigl(  u \cdot \nabla w \bigr) v \, dx
=  \sum_{i=1}^{n} \int_{\ocal } \frac{\partial }{\partial {x}_{i}} ({u}_{i}w )  v \, dx
= - \sum_{i=1}^{n} \int_{\ocal } {u}_{i} w \frac{\partial v}{\partial {x}_{i}} \, dx  \\
&= - \int_{\ocal } \biggl( \sum_{i=1}^{n} {u}_{i} \frac{\partial v}{\partial {x}_{i}}  \biggr) w \, dx
= - \int_{\ocal } \bigl( u \cdot \nabla v   \bigr) w \, dx
= - b(u,v ,w).
\end{align*}
Thus for $u \in {H}^{1}(\ocal , \rd )$ such that $\diver u =0 $ and $(u \cdot n {)}_{|\partial \ocal } =0 $ we have 
\begin{align}  \label{E:antisymmetry_b_H^1}
b(u,w, v ) &=  - b(u,v ,w) , \qquad v,w \in {H}^{1}(\ocal ;{\rzecz }^{{d}_{1}}) . 
\end{align}
In particular,
\begin{align}  \label{E:wirowosc_b_H^1}
b(u,v , v ) =0 , \qquad v \in {H}^{1}(\ocal ;{\rzecz }^{{d}_{1}}).
\end{align}
Hence
\begin{align}  \label{E:antisymmetry_B_H^1}
\dual{B(u,w)}{v}{} &=  - \dual{B(u,v)}{w}{} , \qquad v,w \in {H}^{1}(\ocal ;{\rzecz }^{{d}_{1}})  
\end{align}
and, in particular,
\begin{align}  \label{E:wirowosc_B_H^1}
\dual{B(u,v)}{v}{}  =0 , \qquad v \in {H}^{1}(\ocal ;{\rzecz }^{{d}_{1}}).
\end{align}

\bigskip  \noindent 
Let $m > \frac{d}{2} +1$.  By the Sobolev embedding Theorem, see \cite{Adams}, we have 
$
    {H}^{m-1}(\ocal ,{\rzecz }^{{d}_{1}} ) 
   \hookrightarrow {L}^{\infty } (\ocal , {\rzecz }^{{d}_{1}} ).
$
Thus if $u \in {H}^{1}(\ocal , \rd )$, $\diver u =0 $ and $(u \cdot n {)}_{|\partial \ocal } =0 $,
$w \in {H}^{1}(\ocal ;{\rzecz }^{{d}_{1}})$
and $v \in {H}^{m} (\ocal  ,{\rzecz }^{{d}_{1}}) $, then
\begin{align}
& |b(u,w,v)|  = |b(u,v,w)| 
 = \biggl| \sum_{i=1}^{n} \int_{\ocal } {u}_{i} w \frac{\partial v}{\partial {x}_{i}} \, dx \biggr|  
  \notag \\
& \le  \norm{u}{{L}^{2}}{} \norm{w}{{L}^{2}}{} \norm{\nabla v}{{L}^{\infty }}{}
 \le  {c}_{} \norm{u}{{L}^{2}}{} \norm{w}{{L}^{2}}{} \norm{v}{{H}^{m}}{},  \label{E:b_estimate_H^m}
\end{align}
where  ${c}_{} >0 $ is a certain constant. 
Hence the operator $B$ can be uniquely  extended to the tri-linear form, denoted still by $B$,
\[
   B: {H}_{0} \times {L}^{2}(\ocal ;{\rzecz }^{{d}_{1}}) \to {H}^{-m} (\ocal ;{\rzecz }^{{d}_{1}}) ,
\]
where ${H}_{0}$ is the space of ${L}^{2}(\ocal ,\rd )$ defined by \eqref{H_0},
and the following inequality holds
\begin{equation}
   {B(u,w)}_{{H}^{-m}} \le {c}_{} \norm{u}{{L}^{2}}{} \norm{w}{{L}^{2}}{} ,
   \qquad u \in {H}_{0} , \quad w \in {L}^{2}(\ocal ;{\rzecz }^{{d}_{1}}),
   \label{E:B_estimate_H^m}
\end{equation}
see e.g. Vishik and Fursikov \cite{Vishik_Fursikov_88}.

\bigskip  \noindent
In the following lemma we prove the property of the map $B$ related to assumption (B.5) in the abstract framework.

\bigskip
\begin{lemma} \label{L:B_conv_aux}
Let $u\in {L}^{2}(0,T;{H}_{0})$ and let ${({u}_{n})}_{n}$ be a bounded sequence in ${L}^{2}(0,T;{H}_{0})$ such that ${u}_{n} \to u $ in ${L}^{2}(0,T;{L}^{2}_{loc}(\ocal ;\rd ))$. 
Let $w \in {L}^{2}(0,T;{L}^{2}(\ocal ;{\rzecz }^{{d}_{1}}))$ and let 
${({w}_{n})}_{n}$ be a bounded sequence in ${L}^{2}(0,T;{L}^{2}(\ocal ;{\rzecz }^{{d}_{1}}))$ such that ${u}_{n} \to u $ in ${L}^{2}(0,T;{L}^{2}_{loc}(\ocal ;{\rzecz }^{{d}_{1}} ))$.
If $m > \frac{d}{2}+1$, then for all $t \in [0,T]$ and all $\varphi  \in {H}^{m}_{0}(\ocal ;{\rzecz }^{{d}_{1}})$:
\[
   \lim_{n \to \infty } \int_{0}^{t} \dual{B({u}_{n}(s),{w}_{n}(s))}{\varphi }{} \, ds =
   \int_{0}^{t} \dual{B(u(s),w(s))}{\varphi }{} \, ds . 
\] 
\end{lemma}

\bigskip
\begin{proof}
Assume first that $\varphi  \in \dcal (\ocal ,{\rzecz }^{{d}_{1}}) $. Then there exists $R>0$ such that $\supp \varphi $ is a compact subset of ${\ocal }_{R}$.
Then, using the integration by parts formula, we infer that for every $u \in {H}_{0} $ and $w \in {L}^{2}(\ocal ,{\rzecz }^{{d}_{1}})$
\begin{align} 
&  | \dual{B(u,w)}{\varphi  }{} | = \Bigl| \int_{{\ocal }_{R}} ( u  \cdot \nabla \varphi  ) w \, dx \Bigr| 
  \nonumber \\
&  \le \norm{u}{{L}^{2}({\ocal }_{R})}{}  \norm{w}{{L}^{2}({\ocal }_{R})}{}
  \norm{\nabla \varphi  }{{L}^{\infty }({\ocal }_{R})}{}  
  \le C \norm{u}{{L}^{2}({\ocal }_{R})}{}  \norm{w}{{L}^{2}({\ocal }_{R})}{}
\norm{\varphi  }{{H}^{m}}{} .
 \label{E:estimate_B(O_R)_ext}   
\end{align}
We have
$ B(\un , \wn ) - B(u,w) =  B(\un -u , \wn ) + B(u,\wn -w) $.
Using inequality \eqref{E:estimate_B(O_R)_ext} and the H\"{o}lder inequality, we obtain
\begin{align*}
 & \Bigl| \int_{0}^{t} \dual{ B \bigl( \un (s) ,\wn (s) \bigr) }{\varphi  }{} \, ds 
- \int_{0}^{t}  \dual{ B \bigl( u(s), w(s)\bigr)  }{\varphi }{} \, ds \Bigr|  \\
 &\le \Bigl|  \int_{0}^{t} \dual{ B \bigl( \un (s) - u(s) ,\wn (s) \bigr)  }{\varphi }{} \, ds  \Bigr|
  + \Bigl| \int_{0}^{t}  \dual{ B \bigl( u (s) ,\wn (s) - w(s) \bigr)  }{\varphi }{} \, ds \Bigr| \\
 &\le C \cdot \Bigl(
\norm{\un - u }{{L}^{2}(0,T;{L}^{2}({\ocal }_{R}))}{}  \norm{\wn  }{{L}^{2}(0,T;{L}^{2}({\ocal }_{R}))}{}
+\norm{u}{{L}^{2}(0,T;{L}^{2}({\ocal }_{R}))}{} \norm{\wn - w }{{L}^{2}(0,T;{L}^{2}({\ocal }_{R})}{} 
 \Bigr) \, \norm{\varphi }{{H}^{m }}{} \\
& \le C \cdot  \Bigl( {p}_{T,R} (\un -u) \, \norm{\wn }{{L}^{2}(0,T;{L}^{2}({\ocal }_{}))}{}
  + \norm{u}{{L}^{2}(0,T;{H}_{0})}{}  {p}_{T,R}(\wn -w )\Bigr) 
\, \norm{\varphi }{{H}^{m }}{}  ,
\end{align*}
where ${p}_{T,R}$ is the seminorm defined by \eqref{E:seminorms} and $C$ stands for a positive constant.
Since $\un \to u $ and $\wn \to w$ in ${L}^{2}(0,T;{L}^{2}_{loc}(\ocal ))$, we infer that for all 
$ \varphi   \in \dcal (\ocal ,{\rzecz }^{{d}_{1}}) $
\begin{equation}
    \lim_{n \to \infty } \int_{0}^{t} \dual{ B(\un (s),\wn (s))  }{\varphi }{} \, ds
 = \int_{0}^{t} \dual{ B \bigl( u(s),w(s)\bigr)  }{\varphi }{} \, ds  . \label{E:Appendix_D_B_conv}
\end{equation}
If $\varphi  \in {H}^{m}_{0}(\ocal ,{\rzecz }^{{d}_{1}})$ then for every  $\eps > 0 $ there exists 
${\varphi }_{\eps } \in \dcal (\ocal ,{\rzecz }^{{d}_{1}}) $
such that $\norm{\varphi  - {\varphi }_{\eps }}{{H}^{m }}{} \le \eps $.
Then for all $s \in [0,t]$
\begin{align*}
 &\bigl| \dual{ B(\un (s),\wn (s)) - B(u(s),w(s)) }{\varphi }{} \bigr| 
  \le  \bigl| \dual{ B(\un (s),\wn (s))-B(u(s),w(s)) }{ \varphi -{\varphi }_{\eps }}{} \bigr| \\
 &\quad  + \bigl| \dual{ B(\un (s),\wn (s))- B(u(s),w(s))}{{\varphi }_{\eps } }{} \bigr|  \\
  & \le  \bigl( {\bigl| B(\un (s),\wn (s))\bigr| }_{{H}^{-m }}
  + {\bigl| B(u(s),w(s)) \bigr| }_{{H}^{-m }} \bigr)
  \cdot \norm{\varphi  -{\varphi }_{\eps }}{{H}^{m }}{}  \\
  &\quad  + \bigl| \dual{ B(\un (s),\wn (s)) - B(u(s),w(s))}{{\varphi }_{\eps }}{} \bigr| \\
   &\le \eps  \bigl( {|\un (s )|}_{0}  \, \norm{\wn (s)}{{L}^{2}(\ocal )}{}
   +  {|u(s)|}_{0} \, \norm{w(s)}{{L}^{2}(\ocal )}{} \bigr)
 + \bigl| \dual{ B(\un (s),\wn (s))- B(u(s),w(s))}{{\varphi }_{\eps } }{} \bigr| .
\end{align*}
Hence
\begin{align*}
 &\Bigl| \int_{0}^{t} \dual{B(\un (s),\wn (s)) - B(u(s),w(s))}{\varphi }{}  \, ds \Bigr| \\
  &\le  \eps \int_{0}^{t}  \bigl(  {|\un (s )|}_{0}  \, \norm{\wn (s)}{{L}^{2}(\ocal )}{}
   +  {|u(s)|}_{0} \, \norm{w(s)}{{L}^{2}(\ocal )}{}  \bigr)  d s 
  + \Bigl| \int_{0}^{t}\dual{B(\un (s),\wn (s))-B(u(s),w(s))}{{\varphi }_{\eps } }{}  ds  \Bigr| \\
 &\le \frac{\eps }{2}\cdot \bigl( \sup_{n\ge 1}(\norm{\un }{{L}^{2}(0,T;{H}_{0})}{2}
  + \norm{\wn }{{L}^{2}(0,T;{L}^{2}(\ocal ))}{2} ) +\norm{u}{{L}^{2}(0,T;{H}_{0})}{2}
  + \norm{w}{{L}^{2}(0,T;{L}^{2}(\ocal ))}{2} \bigr)  \\
& + \Bigl| \int_{0}^{t}\dual{B(\un (s),\wn (s))-B(u(s),w(s))}{{\varphi }_{\eps } }{}  ds \Bigr|
.
\end{align*}
Passing to the upper limit as $n \to \infty $, we obtain
\[
 \limsup_{n \to \infty }
\Bigl| \int_{0}^{t} \dual{B(\un (s),\wn (s))-B(u(s),w(s))}{\varphi }{}  \, ds \Bigr|
 \le   M \eps ,
\]
where $M:= \frac{1}{2} \bigl(\sup_{n\ge 1}(\norm{\un }{{L}^{2}(0,T;{H}_{0})}{2}
  + \norm{\wn }{{L}^{2}(0,T;{L}^{2}(\ocal ))}{2} ) +\norm{u}{{L}^{2}(0,T;{H}_{0})}{2}
  + \norm{w}{{L}^{2}(0,T;{L}^{2}(\ocal ))}{2} \bigr) <\infty $.
Since $\eps >0$ is arbitrary, we infer that (\ref{E:Appendix_D_B_conv}) holds for all $\varphi  \in {H}^{m}_{0}(\ocal ;{\rzecz }^{{d}_{1}})$.
The proof of the lemma  is thus complete.
\end{proof}

\bigskip  \noindent
\bf The  operator ${B}_{0}$\rm .
We will now concentrate on  the bilinear map $B$ in the spaces ${H}_{0}$ and ${V}_{0}$ defined by  \eqref{H_0} and \eqref{V_0}, respectively. We will denote it by ${B}_{0}$.
By \eqref{E:B_estimate_H^1} 
we infer that for $u,w \in {V}_{0}$, ${B}_{0}(u,w) \in {V}_{0}'$ and  
 the following inequality holds
\begin{align}  \label{E:estimate_B_0}
 |{B}_{0}(u,w) {|}_{{V}_{0}^{\prime }}   
   \le c \norm{u }{0}{}\norm{w }{0}{},\qquad u,w \in {V}_{0} .
\end{align}
In particular, the map ${B}_{0}:{V}_{0} \times {V}_{0} \to {V}_{0}' $ is bilinear and continuous.
Furthermore, by \eqref{E:antisymmetry_B_H^1}
\begin{align}  \label{E:antisymmetry_B_0}
    \dual{{B}_{0}(u,w)}{v}{} = - \dual{{B}_{0}(u,v)}{w}{} , \qquad u,w,v \in {V}_{0} 
\end{align}
and hence
\begin{align}  \label{E:wirowosc_b}
  \dual{{B}_{0}(u,v)}{v}{}=0,
 \qquad u,v \in {V}_{0}.
\end{align}
Let for any $m>0$, 
\begin{align}  \label{E:U_m}
  {U}_{m} := \text{the closure of ${\vcal }_{0}$ in ${H}^{m}(\ocal , \rd )$} .
\end{align}
In the space ${U}_{m}$ we consider the inner product inherited from ${H}^{m}(\ocal , \rd )$.
Let $m > \frac{d}{2} +1$.  By \eqref{E:B_estimate_H^m},  ${B}_{0}$ 
is a bounded bilinear operator
$
    {B}_{0} : {H}_{0} \times {H}_{0} \to {U}_{m}^{\prime } 
$
and the following inequality holds
\begin{align}  \label{E:estimate_B_0_ext}
 |{B}_{0}(u,w) {|}_{{U}_{m}'} \le c {|u|}_{0}  {|w|}_{0} ,\qquad u,w \in {H}_{0}.
\end{align}
We will also use the following notation, ${B}_{0}(u):={B}_{0}(u,u)$.

\bigskip  \noindent
Let us also recall that the mapping
${B}_{0}:{V}_{0} \to {V}_{0}'$  is locally Lipschitz continuous, i.e. for every $r>0$ there exists a constant ${L}_{r}$
such that 
\begin{equation}  \label{E:Lipschitz_B_0}
   \bigl| {B}_{0}(u) - {B}_{0}(\tilde{u}) {\bigr| }_{{V}_{0}'} \le {L}_{r} \norm{u - \tilde{u}}{0}{} ,
   \qquad u , \tilde{u } \in {V}_{0} , \quad \norm{u}{0}{}, \norm{\tilde{u}}{0}{} \le r  .
\end{equation}

\bigskip  \noindent
\bf Solution of the Navier-Stokes equations.  \rm 
Let ${u}_{0}\in {H}_{0}$, ${f}_{}\in {L}^{2}(0,T; {V}_{0}')$,
 ${G}_{}: [0,T] \times {V}_{0} \to \lhs ({Y}_{W},{H}_{0})$, where ${Y}_{W}$ is a separable Hilbert space, and
$F: [0,T] \times {H}_{0} \times Y \to {H}_{0} $, where $(Y, \ycal )$ is a measurable space,
be given.  

\begin{definition}  \rm  \label{D:NS}
 \bf A martingale solution \rm of  problem \eqref{E:NS}-\eqref{E:NS_boundary}
is a system 
\noindent
$\bigl( \bar{\mathfrak{A}}, \bar{\eta }, \bar{W},\bar{u} \bigr) $,
where $\bigl( \bar{\mathfrak{A}}, \bar{\eta }, \bar{W}\bigr) $ is as in Definition \ref{D:solution}
and $\bar{u}: [0,T] \times \bar{\Omega }\to {H}_{0} $ is a predictable process with $\bar{\p } $ - a.e. paths
\[
  \bar{u}(\cdot , \omega ) \in \dmath \bigl( [0,T], {{H}_{0,} }_{w} \bigr)
   \cap {L}^{2}(0,T;{V}_{0} )
\]
such that  for all $ t \in [0,T] $ and all $ v \in {\vcal }_{0} $ the following identity holds $\bar{\p }$-a.s.
\begin{align*}
 \ilsk{\bar{u}(t)}{v}{0} + & \int_{0}^{t} \dual{ {\acal }_{0}\bar{u}(s)}{v}{} \, ds
+ \int_{0}^{t} \dual{{B}_{0}(\bar{u}(s))}{v}{} \, ds \\
 & = \ilsk{{u}_{0}}{v}{0} +  \int_{0}^{t} \dual{f(s)}{v}{} \, ds
 + \int_{0}^{t} \int_{{Y}_{0}} \ilsk{{F}_{}(s,\bar{u}({s}^{-});y)}{v}{\hmath } \,  \tilde{\bar{\eta }} (ds,dy)  \\
 & + \int_{0}^{t} \int_{Y \setminus {Y}_{0}} \ilsk{{F}_{}(s,\bar{u}({s}^{-});y)}{v}{\hmath } 
 \, \bar{\eta } (ds,dy) 
+ \Dual{\int_{0}^{t} {G}_{}(s,\bar{u}(s))\, d\bar{W}(s)}{v}{} .
\end{align*}
\end{definition}
 \noindent
We apply the abstract framework with
$\hmath := {H}_{0}$, $\vmath := {V}_{0} $, ${\vmath }_{\ast }:= {U}_{m} $ 
with $m> \frac{d}{2}+1$, defined by \eqref{H_0}, \eqref{V_0} and \eqref{E:U_m}, respectively. Furthemore,
\begin{align*}
   & \acal : = {\acal }_{0} , \qquad \rcal :=0  
   \qquad \mbox{and} \qquad  \bcal  := {B}_{0} . 
\end{align*}
By \eqref{E:Acal_0}, it is evident that the operator ${\acal }_{0}$ satisfies condition (A.1).
By  \eqref{E:estimate_B_0}, \eqref{E:antisymmetry_B_0}, \eqref{E:estimate_B_0_ext} 
and  \eqref{E:Lipschitz_B_0} the map ${B}_{0}$ satisfies  conditions (B.1)-(B.4).
By Lemma \ref{L:B_conv_aux}, the map ${B}_{0}$ satisfies assumption (B.5).
Applying Theorem \ref{T:existence}, we obtain the following result about 
the existence of the solution of the Navier-Stokes problem.

\bigskip
\begin{cor} \label{C:existence_NS}
For every
${u}_{0}\in {H}_{0}$, ${f}_{}\in {L}^{2}(0,T; {V}_{0}^{\prime })$,
 ${G}_{}: [0,T] \times {V}_{0} \to \lhs (Y,{H}_{0})$ 
 satisfying conditions (G.1)-(G.3) and $F: [0,T] \times {H}_{0} \times Y \to {H}_{0} $
 satisfying conditions (F.1)-(F.3)
there exists a martingale solution $\bigl( \bar{\mathfrak{A}}, \bar{\eta }, \bar{W},\bar{u} \bigr) $ of  problem \eqref{E:NS}-\eqref{E:NS_boundary} such that
\[
   \bar{\e } \Bigl[ \sup_{t \in [0,T]} {|\bar{u}(t)|}_{{H}_{0}}^{2}  
   + \int_{0}^{T} \norm{\bar{u}(t)}{{V}_{0}}{2} \, dt  \Bigr] < \infty .
\]
\end{cor}

\bigskip
\subsection{\bf Magneto-hydrodynamic equations (MHD) \rm } \label{Sub:MHD}

\noindent
The mathematical model of the motion of a resistive fluid is obtained by coupling the Navier-Stokes equations and the Maxwell equations (see Sermange and Temam \cite{Sermange_Temam_83}, 1983).
We will consider the following stochastic magneto-hydrodynamic (MHD) equations 
\begin{align}  \label{E:MHD}
  &  du(t) + (u \cdot \nabla ) \, u - \frac{1}{Re}  \Delta u
      - s (\mathfrak{b} \cdot \nabla ) \, \mathfrak{b}  + \nabla p
+ s \nabla \bigl( \frac{1}{2} {\mathfrak{b}}^{2} \bigr) \nonumber \\
 &\qquad ={f}_{0}(t)
  + {G}_{0}(t,u(t), \mathfrak{b}(t)) \, d{W}_{0}(t)
 + \int_{{Y}^{0}_{0}} {F}_{0}(t,u({t}^{-}),\mathfrak{b}({t}^{-});y) {\tilde{\eta }}_{0} (dt,dy) , \nonumber  \\
& \qquad + \int_{{Y}^{0} \setminus {Y}^{0}_{0}} {F}_{0}(t,u({t}^{-}),\mathfrak{b}({t}^{-});y) {\eta }_{0} (dt,dy) , \nonumber  \\
  &  d\mathfrak{b} (t) + (u \cdot \nabla ) \, \mathfrak{b}
- (\mathfrak{b} \cdot \nabla ) \, u - \frac{1}{Rm}  \, \curl (\curl \mathfrak{b}) \nonumber  \\
 &\qquad = {f}_{1}(t)
   + {G}_{1}(t,u(t), \mathfrak{b}(t)) \, d{W}_{1}(t)
+ \int_{{Y}^{1}_{0}} {F}_{1}(t,u({t}^{-}),\mathfrak{b}({t}^{-});y) {\tilde{\eta }}_{1} (dt,dy)
   \nonumber \\
&\qquad  + \int_{{Y}^{1}\setminus {Y}^{1}_{0}} {F}_{1}(t,u({t}^{-}),\mathfrak{b}({t}^{-});y) {\eta }_{1} (dt,dy)  ,
\end{align}
in $(0,T)\times \ocal $,
with the conditions 
\begin{equation}  \label{E:MHD_incomp}
\begin{split} 
  & \diver u =0, \quad \diver \mathfrak{b} =0 \quad \text{ in $(0,T)\times \ocal $}
 \end{split}
\end{equation} 
and the following boundary conditions
\begin{equation}  \label{E:MHD_boundary}
\begin{split}  
  & u= 0 \quad \text{ and }  \quad \mathfrak{b} \cdot n = 0 \quad   
    \quad \text{on } \partial \ocal  ,  
 \end{split}
\end{equation}
where $n=({n}_{1},...,{n}_{d})$ stands for the unit outward normal on $\partial \ocal $.
Moreover, we impose the initial conditions
\begin{equation}  \label{E:MHD_initial}
\begin{split} 
   u(0)= {u}_{0} , \qquad \mathfrak{b} (0)={\mathfrak{b}}_{0}.
 \end{split}
\end{equation}
Here $u$, $p$, $\mathfrak{b}$ are interpreted as the velocity, the pressure
and the magnetic field. 
The three positive constants
$\frac{1}{Re} $, $\frac{1}{Rm} $ and $s$ correspond to the kinematic viscosity, the magnetic diffusivity
and the Hartman number, respectively,
see  Duvaut and Lions \cite{Duvaut_Lions_76} and Sermange and Temam \cite{Sermange_Temam_83}.
These equations are used to describe the turbulent flows in magnetohydrodynamics.
Moreover, ${f}_{0}$ and ${f}_{1}$ stand for deterministic external forces,   ${W}_{0}$ and ${W}_{1}$ are cylindrical Wiener processes in Hilbert spaces ${Y}_{W}^{0}$ and ${Y}_{W}^{1}$, respectively, ${\tilde{\eta }}_{0}$ and ${\tilde{\eta }}_{1}$ are compensated time homogeneous Poisson random measures with intensities ${\mu }_{0}$ and ${\mu }_{1}$ on measurable spaces $({Y}^{0},{\ycal }^{0})$ and 
$({Y}^{1},{\ycal }^{1})$, respectively. The sets ${Y}^{0}_{0} \in {\ycal }^{0}$ 
and ${Y}^{1}_{0} \in {\ycal }^{1}$ are such that ${\mu }_{0}({Y}^{0}\setminus {Y}^{0}_{0})<\infty $ and
${\mu }_{1}({Y}^{1}\setminus {Y}^{1}_{0})<\infty $.
The processes ${W}_{0},{W}_{1},{\eta }_{0}, {\eta }_{1}$ are assumed to be independent.

\bigskip  \noindent
\bf Function spaces. \rm
Let us recall that 
the spaces used in the theory of the magneto-hydro\-dy\-na\-mic equations are products of the
spaces used for the Navier-Stokes equations, i.e. ${\vcal }_{0}$, ${H}_{0}$ and ${V}_{0}$ defined
by \eqref{vcal_0}, \eqref{H_0}, \eqref{V_0}  
and spaces used in the theory of the Maxwell equations (spaces denoted with the subscript 1).
Namely, see \cite{Sermange_Temam_83},
\begin{align*}
{\vcal }_{1} &= \{ \mathfrak{c} \in {\ccal }^{\infty } (\overline{\ocal },\rd), \, \, \, 
             \diver \mathfrak{c} = 0,
               \text{ and } (\mathfrak{c} \cdot n {)}_{|\partial \ocal } =0   \} , \\
   {V}_{1} &= \text{the closure of ${\vcal }_{1}$ in ${H}^{1}(\ocal ,\rd )$}  
            = \{ \mathfrak{c} \in {H}^{1}(\ocal ,\rd ), \, \, \, \diver \mathfrak{c} =0  
                \text{ and } (\mathfrak{c} \cdot n {)}_{|\partial \ocal  } =0 \} ,  \\
   {H}_{1} &= \text{the closure of ${\vcal }_{1}$ in ${L}^{2}(\ocal ,\rd )$}  .
\end{align*}
In the space ${H}_{1}$ we consider the inner product and the norm defined by
\begin{align}
     \ilsk{\mathfrak{b}}{\mathfrak{c}}{1} := s \ilsk{\mathfrak{b}}{\mathfrak{c}}{{L}^{2}} ,
     \qquad \norm{\mathfrak{b}}{}{2}:= \ilsk{\mathfrak{b}}{\mathfrak{b}}{1} ,
  \qquad \mathfrak{b},\mathfrak{c} \in {H}_{1}.   
\end{align}
In the space ${V}_{1}$, we consider the inner product 
$
  \ilsk{\mathfrak{b}}{\mathfrak{c}}{{V}_{1}} :=
  \ilsk{\mathfrak{b}}{\mathfrak{c}}{1}
  +\dirilsk{\mathfrak{b}}{\mathfrak{c}}{1} , 
$
where
\begin{align}
   \dirilsk{\mathfrak{b}}{\mathfrak{c}}{1} 
:=  \frac{s}{Rm}  \ilsk{\curl \mathfrak{b}}{\curl \mathfrak{c} }{{L}^{2}} ,
\qquad \mathfrak{b},\mathfrak{c} \in {V}_{1}.
\end{align} 
and the norm  $\norm{\mathfrak{b}}{1}{2} := \ilsk{\mathfrak{b}}{\mathfrak{b}}{{V}_{1}}$. 
Finally, we define the spaces
\begin{equation}  \label{E:MHD_H_V}
   \vmath := {V}_{0} \times {V}_{1} , \qquad \hmath := {H}_{0} \times {H}_{1} , \qquad 
   {\vmath }^{\prime }:= \text{the dual space of $\vmath $}
\end{equation}
with the following inner products
\begin{align*}
  \ilsk{\Phi }{\Psi }{\hmath } & := 
   \ilsk{u}{v}{0} + \ilsk{\mathfrak{b}}{\mathfrak{c}}{1}
\qquad \text{for all} \quad \Phi =(u,\mathfrak{b}), \, \, \,  \Psi =(v,\mathfrak{c}) \in \hmath   \\
 \ilsk{\Phi }{\Psi }{\vmath } & = \ilsk{\Phi }{\Psi }{\hmath } 
 + \dirilsk{\Phi }{\Psi }{} 
 \qquad \text{for all} \qquad \Phi , \Psi \in \vmath ,
\end{align*} 
where
$
  \dirilsk{\Phi }{\Psi }{} :=\dirilsk{u}{v}{0} +\dirilsk{\mathfrak{b}}{\mathfrak{c}}{1} 
$
and $\ilsk{\cdot }{\cdot }{0} $ and $\dirilsk{\cdot }{\cdot }{0}$ are defined by 
\eqref{E:il_sk_H_0} and \eqref{E:il_sk_Dirichlet}, respectively.
We have
$
      \vmath  \subset \hmath  \subset {\vmath }^{\prime } ,
$
where the embeddings are dense and continuous.

\bigskip  \noindent
\bf The operator $\acal $.  \rm 
We define the operators
${\acal }_{1}$ and  
$\acal $ by the following formulae
\begin{align}
 \dual{{\acal }_{1}\mathfrak{b}}{\mathfrak{c}}{} &:= \dirilsk{\mathfrak{b}}{\mathfrak{c}}{1}, 
\qquad \mathfrak{b},\mathfrak{c} \in {V}_{1},  \notag \\
 \dual{\acal \Phi }{\Psi }{} &:= \dual{{\acal }_{0}u}{v}{} 
   + \dual{{\acal }_{1}\mathfrak{b}}{\mathfrak{c}}{},
 \qquad \Phi ,\Psi \in \vmath ,    \label{E:MHD_Acal}
\end{align}
where ${\acal }_{0}$ is defined by \eqref{E:Acal_0}. It is clear that
${\acal }_{1} \in \lcal ({V}_{1},{V}_{1}')$ and  
$\acal \in \lcal (\vmath ,\vmath ')$.
Let us also notice that 
\begin{equation}
  \dual{\acal \Phi }{\Psi }{} = \dirilsk{\Phi }{\Psi }{}, \qquad \Phi ,\Psi \in \vmath .  
\end{equation}

\bigskip  \noindent
\bf The form $\hat{b}$ and the operator $\hat{B}$\rm .
Using the form $b$  defined by \eqref{E:form_b} we will consider the tri-linear  form  $\hat{b}$ on $\vmath \times \vmath \times \vmath $, where $\vmath $ is defined by \eqref{E:MHD_H_V}, see  Sango \cite{Sango_2010} and Sermange and Temam \cite{Sermange_Temam_83}. Namely,
\begin{align*}
 \hat{b} ({\Phi }^{(1)},{\Phi }^{(2)},{\Phi }^{(3)})
  := & b({u}^{(1)}, {u}^{(2)}, {u}^{(3)})
     -  s b({\mathfrak{b}}^{(1)}, {\mathfrak{b}}^{(2)}, {u}^{(3)})  \\
     & + sb({u}^{(1)}, {\mathfrak{b}}^{(2)}, {\mathfrak{b}}^{(3)})
     - sb({\mathfrak{b}}^{(1)}, {u}^{(2)}, {\mathfrak{b}}^{(3)}) ,
\end{align*}
where ${\Phi }^{(i)} = ({u}^{(i)},{\mathfrak{b}}^{(i)}) \in \vmath $, $i=1,2,3$. 
By \eqref{E:b_estimate_H^1} we see that the form $\hat{b}$ is continuous.
Moreover, by \eqref{E:antisymmetry_b_H^1} and \eqref{E:wirowosc_b_H^1}
the form $\hat{b}$ has the following properties, see also \cite{Sango_2010},
\begin{align} \label{E:16_Sango}
   &  \hat{b} ({\Phi }^{(1)},{\Phi }^{(2)},{\Phi }^{(3)})
       = -  \hat{b} ({\Phi }^{(1)},{\Phi }^{(3)},{\Phi }^{(2)}) , \qquad 
{\Phi }^{(i)} \in \vmath , \quad i=1,2,3  
\end{align}
and in particular
\begin{align}  
   & \hat{b} ({\Phi }^{(1)},{\Phi }^{(2)},{\Phi }^{(2)}) =0 , \qquad {\Phi }^{(1)}, {\Phi }^{(2)} \in \vmath .
   \label{E:15_Sango}
\end{align}        
Now, let us define a bilinear map $\hat{B}$ by 
\begin{equation}  \label{E:MHD_B}
   \hat{B} (\Phi ,\Psi ) := 
   \hat{b} (\Phi ,\Psi ,\cdot )  , 
       \qquad \qquad \Phi ,\Psi  \in \vmath . 
\end{equation}
We will also use the notation $ \hat{B} (\Phi ) :=  \hat{B} (\Phi ,\Phi ) $.
For $m>0$ us define the following space
\begin{equation}
   {\vmath }_{m} := \text{ the closure of  ${\vcal }_{0} \times {\vcal }_{1}  $ in the space
   ${H}^{m}(\ocal , \rd ) \times {H}^{m}(\ocal , \rd )$.} \label{E:V_m}
\end{equation}

\noindent
We will collect properties of the map $\hat{B}$ in the following lemma.

\begin{lemma}   \label{L:MHD_properties_B}
\begin{description}
\item[(1) ] There exists a constant ${c}_{1} >0 $ such that 
\[
       | \hat{B}(\Phi , \Psi ){|}_{\vmath '} \le {c}_{1}\norm{\Phi }{\vmath }{} \norm{\Psi }{\vmath }{}  , \qquad \Phi ,\Psi  \in \vmath . 
\]
In particular, the form $\hat{B}: \vmath \times \vmath \to \vmath '$ is bilinear and continuous.
Moreover, 
\[
      \dual{\hat{B}(\Phi ,\Psi )}{\Theta  }{} = - \dual{\hat{B}(\Phi ,\Theta  )}{\Psi  }, 
      \qquad \Phi ,\Psi , \Theta  \in \vmath . 
\]
\item[(2) ] The mapping $\hat{B}$ is locally Lipschitz continuous on the space $\vmath $, i.e.
   for every $r>0 $ there exists a constant ${L}_{r}>0$ such that 
\[
  | \hat{B}(\Phi )- \hat{B} (\tilde{\Phi } ){|}_{\vmath '} 
  \le {L}_{r}  \norm{\Phi -\tilde{\Phi } ) }{\vmath }{} , \qquad \Phi ,\tilde{\Phi } \in \vmath ,
  \quad \norm{\Phi }{\vmath }{} , \norm{\tilde{\Phi }}{\vmath }{} \le r .
\]   
\item[(3) ] If $m> \frac{d}{2}+1$, then $\hat{B}$ can be extended to the bilinear mapping from $\hmath \times \hmath $ to ${\vmath }_{m}'$ (denoted still by $\hat{B}$) such that
\[
       | \hat{B}(\Phi , \Psi ){|}_{{\vmath }_{m}'} \le {c}_{2}|\Phi {|}_{\hmath } |\Psi {|}_{\hmath } , \qquad \Phi ,\Psi  \in \hmath  ,
\]
where ${c}_{2}$ is a positive constant.
\end{description}
\end{lemma}

\bigskip 
\begin{proof}
Using the definition \eqref{E:MHD_B} of the mapping $\hat{B}$, we infer that
assertion (1) follows from \eqref{E:b_estimate_H^1},  \eqref{E:16_Sango} and \eqref{E:15_Sango}.
Assertion (2) follows from the following inequalities
\begin{align*}
\bigl| \hat{B}(\Phi ,\Phi ) -  \hat{B}(\tilde{\Phi },\tilde{\Phi }) {\bigr| }_{\vmath'}
 & \le \bigl| \hat{B}(\Phi ,\Phi -\tilde{\Phi })  {\bigr| }_{\vmath'}
   + \bigl| \hat{B}(\Phi -\tilde{\Phi },\tilde{\Phi }) {\bigr| }_{\vmath'} \\
& \le \norm{\hat{B}}{}{} \cdot \norm{\Phi }{\vmath}{} \norm{\Phi -\tilde{\Phi }}{\vmath}{}
   + \norm{\hat{B}}{}{} \cdot \norm{\Phi -\tilde{\Phi }}{\vmath}{} \norm{\tilde{\Phi }}{\vmath}{} \\
& = \norm{\hat{B}}{}{}  (\norm{\Phi }{\vmath}{} +\norm{\tilde{\Phi }}{\vmath}{} ) \norm{\Phi -\tilde{\Phi }}{\vmath}{}
  \le 2r  \norm{\hat{B}}{}{}  \cdot \norm{\Phi -\tilde{\Phi }}{\vmath}{} .
\end{align*}
Thus the Lipschitz condition holds with ${L}_{r}= 2r \norm{\hat{B}}{}{} $, where $\norm{\hat{B}}{}{}$ stands for the norm of the bilinear map $\hat{B}:\vmath \times \vmath  \to \vmath '$.
Assertion (3) follows from \eqref{E:b_estimate_H^m}. The proof is thus complete.
\end{proof}

\bigskip  \noindent
\bf Weak formulation of problem \eqref{E:MHD}\rm .
Let $\hmath $ and  $\vmath  $ be the Hilbert spaces  defined by \eqref{E:MHD_H_V}. 
\begin{itemize}
\item Let ${f}_{0} \in {L}^{2}(0,T; {V}_{0}')$, ${f}_{1} \in {L}^{2}(0,T; {V}_{1}')$,
${u}_{0} \in {H}_{0}$ and ${\mathfrak{b}}_{0} \in {H}_{1}$ be given
and let
\[
    f:=({f}_{0},{f}_{1}) , \qquad 
    {\Phi }_{0} := ({u}_{0},{\mathfrak{b}}_{0} ). 
\]
Then $f \in {L}^{2}(0,T;{\vmath }^{\prime })$ and ${\Phi }_{0} \in \hmath $.
\item Let ${Y}_{W}:={Y}^{0}_{W} \times {Y}^{1}_{W}$ and let $W(t) = ({W}_{0}(t), {W}_{1}(t))$. Then $W$ is a cylindrical Wiener process on the space ${Y}_{W}$.
Moreover, let ${G}_{0}: [0,T] \times {V}_{0} \times {V}_{1} \to \lhs ({Y}^{0}_{W},{H}_{0})$ and 
${G}_{1}:[0,T] \times {V}_{0} \times {V}_{1} \to \lhs ({Y}^{1}_{W},{H}_{1})$
be given and let us define the map $G$ by the formula
\begin{equation}
    G (\Phi )h := ({G}_{0}(u ){h}_{0}, {G}_{1}(\mathfrak{b} ){h}_{1}) , 
\end{equation}
where $\Phi = (u, \mathfrak{b}) \in {V}_{0}\times {V}_{1}$, $h :=({h}_{1},{h}_{2}) \in {Y}_{W}$ and
$t \in [0,T]$.
Then $G:\vmath \to \lhs ({Y}_{W}, \hmath )$.
\item 
Let $Y:={Y}^{0} \times {Y}^{1}$. Then $(Y,\ycal )$, where $\ycal := {\ycal }^{0} \otimes {\ycal }^{1}  $ is a measurable space and  $\eta (dt,dy):= ({\eta }_{0}(dt,d{y}_{0}),{\eta }_{1}(dt,d{y}_{2})) $ is a time homogeneous Poisson random measure on $(Y,\ycal )$ with the intensity measure $\mu := {\mu }_{0} \otimes {\mu }_{1}$. Let ${Y}_{0}:={Y}^{0}_{0}\times {Y}^{1}_{0}$.
Let ${F}_{0}:[0,T] \times {H}_{0} \times {H}_{1} \times {Y}_{0} \to {H}_{0}$ and  
${F}_{1}:[0,T] \times {H}_{0} \times {H}_{1} \times {Y}_{1} \to {H}_{1}$ be given 
 and let us define the map $F$ by the formula
\begin{equation}
 F(t,\Phi ;y) := ({F}_{0}(t,u,\mathfrak{b};{y}_{0}), {F}_{1}(t,u,\mathfrak{b};{y}_{1})), 
\end{equation}
where $\Phi = (u, \mathfrak{b}) \in {H}_{0}\times {H}_{1}$, $ y:=({y}_{0},{y}_{1}) \in Y $ and 
$ t \in [0,T] $.
Then
$F:[0,T]\times \hmath \times Y \to \hmath $.
\end{itemize}

\bigskip  \noindent
We apply the abstract framework with the spaces 
$\hmath $ and  $\vmath  $ defined by \eqref{E:MHD_H_V},  the space ${\vmath }_{\ast }:= {\vmath }_{m}$ with $m > \frac{d}{2}+1$ defined by \eqref{E:V_m}, the operator $\acal $ 
defined by  \eqref{E:MHD_Acal},
\begin{align*}
   & \bcal (\Phi ,\Psi ) := \hat{B}(\Phi ,\Psi ) , \qquad \Phi ,\Psi  \in \vmath  ,
\end{align*}
where $\hat{B}$ is defined by \eqref{E:MHD_B}, and $\rcal :=0$. By Lemma \ref{L:MHD_properties_B} the map $\bcal $ satisfies assumptions (B.1)-(B.4).
By Lemma \ref{L:B_conv_aux}, the map $\bcal $ satisfies assumption (B.5).

\begin{definition}  \rm  \label{D:MHD}
 \bf A martingale solution \rm of the problem \eqref{E:MHD}--\eqref{E:MHD_initial} is a system
$(\bar{\mathfrak{A}} , \bar{\eta }, \bar{W}, \bar{\Phi })$, where 
 $\bigl( \bar{\mathfrak{A}}, \bar{\eta }, \bar{W}\bigr) $ is as in Definition \ref{D:solution}
and 
 $\bar{\Phi }: [0,T] \times \bar{\Omega } \to \hmath $ is a predictable process with $\bar{\p } $ - a.e. paths
\[
  \bar{\Phi } (\cdot , \omega ) \in \dmath \bigl( [0,T], {\hmath  }_{w} \bigr)
   \cap {L}^{2}(0,T;\vmath )
\] 
such that for all $ t \in [0,T] $ and all $ \Psi \in \vmath  $ the following identity holds 
$\bar{\p }$- a.s.
\begin{align*}
 &\ilsk{\bar{\Phi } (t)}{\psi }{} +  \int_{0}^{t} \dual{ \acal \bar{\Phi } (\sigma )}{\Psi }{} \, d\sigma 
+ \int_{0}^{t} \dual{\hat{B}(\bar{\Phi } (\sigma ),\bar{\Phi }(\sigma ))}{\Psi }{} \, d\sigma  
  = \ilsk{{\Phi }_{0}}{\Psi }{\hmath } \\
& \quad +  \int_{0}^{t} \dual{f(\sigma )}{\Psi }{} \, d\sigma 
 + \int_{0}^{t} \int_{{Y}_{0}} \ilsk{{F}_{}(s,\bar{\Phi } ({s}^{-});y)}{\Psi }{\hmath } \,  \tilde{\bar{\eta }} (ds,dy)  \\
& \quad  + \int_{0}^{t} \int_{Y \setminus {Y}_{0}} \ilsk{{F}_{}(s,\bar{\Phi } ({s}^{-});y)}{\Psi }{\hmath } \,  \bar{\eta } (ds,dy) 
 + \Dual{\int_{0}^{t} G(\bar{\Phi } (\sigma ))\, d\bar{W}(\sigma )}{\Psi }{} .
\end{align*}
\end{definition}
 \noindent
Applying Theorem \ref{T:existence}, we obtain the following result about 
the existence of the martingale solution of the magneto-hydrodynamic equations. 

\begin{cor} \label{C:MHD_existence}
For every ${\Phi }_{0}=({u}_{0},{\mathfrak{b}}_{0}) \in \hmath $, $f\in {L}^{2}(0,T;\vmath ')$, $G:[0,T] \times \vmath \to \lhs (Y, \hmath )$ satisfying conditions   (G.1)-(G.3) and $F: [0,T] \times \hmath \times Y \to \hmath  $ 
satisfying conditions   (F.1)-(F.3)
there exists a martingale solution $(\bar{\mathfrak{A}} , \bar{\eta }, \bar{W}, \bar{\Phi })$, 
where $\bar{\Phi }=(\bar{u}, \bar{\mathfrak{b}})$,
 of  problem \eqref{E:MHD}-\eqref{E:MHD_initial} such that
\[
  \bar{\e } \Bigl[ 
\sup_{t \in [0,T]} ({|\bar{u}(t)|}_{{H}_{0}}^{2} + {|\bar{\mathfrak{b}}(t)|}_{{H}_{1}}^{2})
 + \int_{0}^{T} (\norm{\bar{u}(t)}{{V}_{0}}{2} +\norm{\bar{\mathfrak{b}}(t)}{{V}_{1}}{2}) \, dt
  \bigr]  < \infty .
\]
\end{cor}


\bigskip
\subsection{\bf Boussinesq equations \rm }

\noindent
We consider $\rd $, where $d=2,3$, with the canonical basis $\{ {e}_{1} , {e}_{2} \} $ or $\{ {e}_{1},{e}_{2},{e}_{3} \} $
and  the Boussinesq model for the B\'{e}nard problem with random influences in the  domain $\ocal $
\begin{align} 
 &  du(t) + (u \cdot \nabla ) u 
  -  \frac{1}{Re} \Delta u - \vartheta {e}_{d} + \nabla p  
  =  {f}_{0}(t) +  {G}_{0}(t,u(t),\vartheta (t)) \, d {W}_{0}(t) \nonumber  \\
 & + \int_{{Y}^{0}_{0}} {F}_{0}(t,u({t}^{-}),\vartheta ({t}^{-});y) {\tilde{\eta }}_{0} (dt,dy)
  + \int_{{Y}^{0}\setminus {Y}^{0}_{0}} {F}_{0}(t,u({t}^{-}),\vartheta ({t}^{-});y) {\eta }_{0} (dt,dy), 
   \nonumber  \\
  & d\vartheta (t)  + (u \cdot \nabla ) \vartheta  - \kappa \Delta \vartheta -{u}_{d} 
   = {f}_{2}(t) + {G}_{2}(t,u(t),\vartheta (t)) \, d {W}_{2}(t) \nonumber \\
 &+ \int_{{Y}^{2}_{0}} {F}_{2}(t,u({t}^{-}),\vartheta ({t}^{-});y) {\tilde{\eta }}_{1} (dt,dy) 
  + \int_{{Y}^{2}\setminus {Y}^{2}_{0}} {F}_{2}(t,u({t}^{-}),\vartheta ({t}^{-});y) {\eta }_{1} (dt,dy) , 
\label{E:Boussinesq}
\end{align}
where $t \in [0,T]$, with the incompressibility condition 
\begin{equation} \label{E:Bouss_incomp}
  \diver u =0  
\end{equation}
and with the homogeneous  boundary conditions
\begin{align}  \label{E:Boussinesq_boundary_cond}
  & {u}_{|\partial \ocal } = 0 \quad \text{and} \quad   
  {\vartheta }_{|\partial \ocal }= 0 .
\end{align}

\bigskip  \noindent
The functions $u=u(t,x) = \bigl( {u}_{1}(t,x),...,{u}_{d}(t,x)\bigr) $  and $p=p(t,x)$ are interpreted
as the velocity and the pressure of the fluid. Function $\vartheta =\vartheta (t,x)$ represents the temperature of the fluid  (see \cite{Foias_Manley_Temam_87}) and a given parameter $\kappa >0$ is the coefficient of the thermometric conductivity. 
Here ${f}_{0}, {f}_{2}$ stand for the deterministic external forces, 
${W}_{0}$ and ${W}_{2}$ are cylindrical Wiener processes in Hilbert spaces ${Y}_{W}^{0}$ and ${Y}_{W}^{2}$, respectively, ${\tilde{\eta }}_{0}$ and ${\tilde{\eta }}_{2}$ are compensated time homogeneous Poisson random measures with intensities ${\mu }_{0}$ and ${\mu }_{2}$ on measurable spaces 
$({Y}^{0},{\ycal }^{0})$ and 
$({Y}^{2},{\ycal }^{2})$, respectively. Moreover, ${Y}^{0}_{0}\in {\ycal }^{0}$ and 
${Y}^{2}_{0}\in {\ycal }^{2}$ are such that ${\mu }_{0}({Y}^{0}\setminus {Y}^{0}_{0})<\infty $ and
${\mu }_{2}({Y}^{2}\setminus {Y}^{2}_{0})<\infty $.
The processes ${W}_{0},{W}_{2},{\eta }_{0}, {\eta }_{2}$ are assumed to be independent.

\bigskip  \noindent
The functional setting of the problem \eqref{E:Boussinesq}-\eqref{E:Boussinesq_boundary_cond}
 is analogous to that considered in \cite{Foias_Manley_Temam_87} and \cite{Brzezniak_Motyl_2010}.

\bigskip  \noindent
\bf Function spaces. \rm
The spaces used in the theory of the Boussinesq problem are  products of
spaces used for the Navier-Stokes equations, i.e. ${\vcal }_{0}$, ${H}_{0}$ and ${V}_{0}$ defined
by \eqref{vcal_0}, \eqref{H_0}, \eqref{V_0}  
and spaces used in the theory of the heat equation (spaces denoted with the subscript 2).
They are
\begin{align*}
{\vcal }_{2} &= {\ccal }^{\infty }_{c} (\ocal , \rzecz ) , \\
   {V}_{2} &= {H}^{1}_{0}(\ocal ,\rzecz )
    := \text{the closure of ${\vcal }_{2}$ in ${H}^{1}(\ocal ,\rzecz )$} ,  \\
   {H}_{2} &= {L}^{2}(\ocal ,\rzecz ) .
\end{align*}
In the space ${V}_{2}$, we consider the inner product 
\begin{equation} \label{E:il_sk_Dirichlet_2}
   \dirilsk{\theta }{\vartheta }{2} :=\ilsk{\theta }{\vartheta }{{L}^{2}}
  + \kappa \ilsk{\nabla \theta }{\nabla \vartheta }{{L}^{2}} 
\end{equation}
and the norm $\norm{\vartheta }{{V}_{2}}{2}
 := \norm{\vartheta }{{L}^{2}}{2} + \norm{\nabla \vartheta }{{L}^{2}}{2}$, where
 $\vartheta \in {V}_{2}$.    
Finally, we define
\begin{equation}  \label{E:Boussinesq_H_V}
   \vmath  := {V}_{0} \times {V}_{2} , \qquad \hmath  := {H}_{0} \times {H}_{2} , \qquad 
   {\vmath }^{\prime }:= \text{the dual space of $\vmath $}
\end{equation}
with the following inner products
\begin{align*}
  \ilsk{\phi }{\psi }{\hmath } & = 
  \ilsk{\phi }{\psi }{}:= \ilsk{u}{v}{0} + \ilsk{\theta }{\vartheta }{2}
\qquad \text{for all} \qquad \phi =(u,\theta ), \, \, \,  \psi =(v,\vartheta ) \in \hmath    \\
 \dirilsk{\phi }{\psi }{\vmath } & = 
\dirilsk{\phi }{\psi }{} :=  \dirilsk{u}{v}{0} + \dirilsk{\theta }{\vartheta }{2}
 \qquad \text{for all} \qquad \phi , \psi \in \vmath  .
\end{align*} 
We have
$
      \vmath \subset \hmath   \subset {\vmath }^{\prime } ,
$
where the embeddings are compact and each space is dense in the following one.

\bigskip  \noindent
\bf The operators ${\acal }_{2}$ and $\acal $.  \rm 
We define the operators ${\acal }_{2}$ and   $\acal $ by 
\begin{align}
 \dual{{\acal }_{2}\theta }{\vartheta }{} &:=  \dirilsk{\theta }{\vartheta }{2}, 
\qquad \theta , \vartheta  \in {V}_{2}  \nonumber  \\
 \dual{\acal \phi }{\psi }{} & := \dual{{\acal }_{0}u}{v}{} 
   + \dual{{\acal }_{2}\theta }{\vartheta }{},  \qquad \phi ,\psi \in \vmath ,
  \label{E:Boussinesq_Acal}
\end{align}   
where ${\acal }_{0}$ is defined by \eqref{E:Acal_0}.
It is clear that ${\acal }_{2} \in \lcal ({V}_{2},{V}_{2}')$ and thus $\acal \in \lcal (\vmath ,{\vmath }^{\prime })$.
Let us notice that 
\begin{equation}
\dual{\acal \phi }{\psi }{} = \dirilsk{\phi }{\psi }{V}, \qquad \phi ,\psi \in \vmath  . 
\end{equation}

\bigskip  \noindent
\bf The  the operator ${B}_{2}$\rm .
Let us consider the following tri-linear form $b$  defined by \eqref{E:form_b} with ${d}_{1}:=1$.
If we define a bilinear map ${B}_{2}$ by
${B}_{2}(u,\vartheta ):={b}_{}(u,\vartheta , \cdot )$, then by \eqref{E:B_estimate_H^1}, we infer that ${B}_{2}(u,\vartheta )  \in {V}_{2}'$
and that the following estimate holds
\begin{align}  \label{E:estimate_B2}
 |{B}_{2}(u,\vartheta ) {|}_{{V}_{2}'}  
   \le c  \norm{u }{0}{}\norm{\vartheta  }{2}{},
\qquad  u \in {V}_{0}, \quad \vartheta \in {V}_{2} .
\end{align}
In particular,
the mapping ${B}_{2}: {V}_{0} \times {V}_{2} \to {V}_{2}' $ is bilinear and continuous.
Furthermore, by \eqref{E:antisymmetry_B_H^1}
\begin{align}  \label{E:antisymmetry_B_2}
    \dual{{B}_{2}(u,\vartheta )}{\theta }{} = - \dual{{B}_{2}(u,\theta )}{\vartheta }{} , 
\qquad   u \in {V}_{0}, \quad \vartheta ,\theta \in {V}_{2} .
\end{align}
and, in particular, 
\begin{align}  \label{E:wirowosc_b2}
\dual{{B}_{2}(u,\vartheta )}{\vartheta }{} =0,
  \qquad u \in {V}_{0}, \quad \vartheta \in {V}_{2}.
\end{align}
If $m > \frac{d}{2} +1 $,  the operator ${B}_{2}$ can be extended 
to 
$
    {B}_{2} : {H}_{0} \times {H}_{2} \to {H}^{-m}_{0}(\ocal , \rzecz ) ,
$
and  by \eqref {E:B_estimate_H^m} satisfies the following inequality
\begin{align}  \label{E:estimate_B2_ext}
 |{B}_{2}(u,\vartheta  ) {|}_{{H}^{-m}_{0}(\ocal , \rzecz )} 
\le c {|u |}_{0}  {|\vartheta |}_{2} ,\qquad u \in {H}_{0}, \, \, \vartheta  \in {H}_{2} .
\end{align}
Using the above notation, the Boussinesq problem can be written as a system of the following two equations
\begin{align}
 & du(t) + \bigl[  {\acal }_{0} u + {B}_{0}(u,u) - \vartheta {e}_{d} \bigr] \, dt 
   = {f}_{0} (t) \, dt  
  + {G}_{0}(t,u(t),\vartheta (t)) \, d{W}_{0}(t) \nonumber \\
 & \quad   + \int_{{Y}^{0}_{0}} {F}_{0}(t,u({t}^{-}),\vartheta ({t}^{-});y) \, {\tilde{\eta }}_{0} (dt,dy)
 + \int_{{Y}^{0} \setminus {Y}^{0}_{0}} {F}_{0}(t,u({t}^{-}),\vartheta ({t}^{-});y) 
  \, {\eta }_{0} (dt,dy) \label{E:Benard_1}  \\
 & d\vartheta (t)  + \bigl[   {\acal }_{2} \vartheta  + {B}_{2}(u,\vartheta ) -  {u}_{d} \bigr] \, dt 
   = {f}_{2} (t) \, dt 
 + {G}_{2}(t,u(t),\vartheta (t)) \, d{W}_{2}(t) \nonumber \\
 & \quad  + \int_{{Y}^{2}_{0}} {F}_{2}(t,u({t}^{-}),\vartheta ({t}^{-});y) \, {\tilde{\eta }}_{2} (dt,dy) + \int_{{Y}^{2} \setminus {Y}^{2}_{0}} {F}_{2}(t,u({t}^{-}),\vartheta ({t}^{-});y) 
   \, {\eta }_{2} (dt,dy).  \label{E:Benard_2}
\end{align}
with the initial conditions
\begin{equation}
   u(0) = {u}_{0}, \qquad \vartheta (0)= {\vartheta }_{0}.   \label{E:Benard_init_cond}
\end{equation}

\bigskip  \noindent
\bf Weak formulation of problem \eqref{E:Boussinesq}\rm .
Let  $\hmath $ and $\vmath $ be the Hilbert spaces defined by \eqref{E:Boussinesq_H_V}.
\begin{itemize}
\item Let ${f}_{0} \in {L}^{2}(0,T; {V}_{0}')$, ${f}_{2} \in {L}^{2}(0,T; {V}_{2}^{\prime })$,
${u}_{0} \in {H}_{0}$ and ${\vartheta }_{0} \in {H}_{2}$ be given
and let
\[
    f:=({f}_{0},{f}_{2}) , \qquad 
    {\phi }_{0} := ({u}_{0},{\vartheta }_{0} ). 
\]
Then $f \in {L}^{2}(0,T; {\vmath }^{\prime })$ and ${\phi }_{0} \in \hmath $.
\item Let ${Y}_{W}:={Y}^{0}_{W} \times {Y}^{2}_{W}$ and let $W(t) = ({W}_{0}(t), {W}_{2}(t))$. Then $W$ is a cylindrical Wiener process on the space ${Y}_{W}$.
Moreover, let ${G}_{0}: [0,T] \times {V}_{0} \times {V}_{2} \to \lhs ({Y}^{0}_{W},{H}_{0})$ and 
${G}_{2}:[0,T] \times {V}_{0} \times {V}_{2} \to \lhs ({Y}^{2}_{W},{H}_{2})$
be given.  Let  us define the map $G$ by the formula
\begin{equation}
    G (\phi )h := ({G}_{0}(u ){h}_{0}, {G}_{1}(\vartheta ){h}_{1}) , 
\end{equation}
where $\phi = (u,\vartheta ) \in \vmath $, $h :=({h}_{1},{h}_{2}) \in {Y}_{W}$ and
$t \in [0,T]$.
Then $G:\vmath \to \lhs ({Y}_{W}, \hmath )$.
\item 
Let $Y:={Y}^{0} \times {Y}^{2}$. Then $(Y,\ycal )$, where $\ycal := {\ycal }^{0} \otimes {\ycal }^{2}  $ is a measurable space and  $\eta (dt,dy):= ({\eta }_{0}(dt,d{y}_{0}),{\eta }_{2}(dt,d{y}_{2})) $ is a time homogeneous Poisson random measure on $(Y,\ycal )$ with the intensity measure $\mu := {\mu }_{0} \otimes {\mu }_{2}$. Moreover, ${Y}_{0}:= {Y}^{0}_{0}\times {Y}^{2}_{0}$.
Let ${F}_{0}:[0,T] \times {H}_{0} \times {H}_{2} \times {Y}_{0} \to {H}_{0}$ and  
${F}_{2}:[0,T] \times {H}_{2} \times {H}_{2} \times {Y}_{2} \to {H}_{2}$ be given and let us define the map $F$ by the formula
\begin{equation}
 F(t,\phi ;y) := ({F}_{0}(t,u,\vartheta ;{y}_{0}), {F}_{1}(t,u,\vartheta ;{y}_{1})), 
\end{equation}
where $\phi = (u,\vartheta ) \in \hmath $, $ y:=({y}_{0},{y}_{2}) \in Y $ and 
$ t \in [0,T] $.
Then
$F:[0,T]\times \hmath \times Y \to \hmath $.
\end{itemize}

\bigskip \noindent
Using the operator $\acal $  defined by  \eqref{E:Boussinesq_Acal} and putting
\begin{align}
    \bcal  (\phi ,\psi ) &:= \bigl( {B}_{0}(u,v),{B}_{2}(u,\theta ) \bigr) , 
\qquad \phi =(u,\vartheta ) ,\psi =(v,\theta )  \in \vmath  , \label{E:Boussinesq_Bcal} \\
    \rcal (\phi)& :=(-\vartheta {e}_{d}, - {u}_{d}), \qquad \phi \in \vmath ,
   \label{E:Boussinesq_Rcal}
\end{align}
by \eqref{E:Benard_1}, \eqref{E:Benard_2} and \eqref{E:Benard_init_cond}
we obtain the following equation for $\phi = (u,\vartheta )$  
\begin{align} 
 & d\phi (t) + \bigl[  \acal \phi  + \bcal  (\phi ) + \rcal  \phi  \bigr] \, dt 
   = {f}_{} (t) \, dt 
  + G (t,\phi (t)) \, dW(t) \nonumber \\
 & \quad + \int_{{Y}_{0}} F(t,\phi({t}^{-});y) \, \tilde{\eta } (dt,dy)
    + \int_{Y\setminus {Y}_{0}} F(t,\phi({t}^{-});y) \, \eta  (dt,dy)
  \label{E:Benard}
\end{align}  
with the initial condition
\begin{equation} \label{E:Boussinesq_ini_cond}
    \phi (0) = {\phi }_{0}.
\end{equation}

\bigskip  \noindent
W will now be concerned with some properties of the maps $\bcal $ and $\rcal $ defined by
\eqref{E:Boussinesq_Bcal} and \eqref{E:Boussinesq_Rcal}, respectively.
Let  ${U}_{m}$ be the space defined by \eqref{E:U_m}
and let us define 
\begin{equation}
{\umath }_{m}:= {U}_{m} \times {H}^{m}_{0}(\ocal ,\rzecz ). \label{E:Boussinesq_U_m}
\end{equation}

\begin{lemma}  \label{L:Boussinesq_estimate_B}
\rm (Properties of the map $\bcal $) \it 
\begin{description}
\item[(1)] There exists a constant ${c}_{2}>0 $ such that
\begin{align*}
  | \bcal (\phi ,\psi ) {|}_{{\vmath }^{\prime }} 
  \le {c}_{1} \norm{\phi }{\vmath }{}\norm{\psi }{\vmath }{} , \qquad \phi ,\psi \in \vmath  .
\end{align*}
Moreover,
\[
     \dual{\bcal (\phi, \psi)}{\chi }{} = - \dual{B(\phi, \chi )}{\psi }{},
     \qquad \phi ,\psi , \chi \in \vmath  .
\]
\item[(2)] If $m > \frac{d}{2}+1$, then
$\bcal $ can be extended to the bilinear mapping from $\hmath  \times \hmath $ to 
${\umath }_{m}^{\prime }$.
Moreover, there exists a constant ${c}_{2}>0 $ such that
\begin{align*}
  | \bcal (\phi ,\psi ) {|}_{{\umath }_{m}^{\prime }} 
  \le {c}_{2} {|\phi |}_{\hmath }^{} {|\psi |}_{\hmath }^{} , \qquad \phi , \psi \in \hmath  .
\end{align*}
\item[(3)] The mapping $\bcal $ is locally Lipschitz continuous, i.e. for every $r>0$ there exists a constant ${L}_{r}$
such that 
\begin{equation*}
   \bigl| \bcal (\phi ) - \bcal (\tilde{\phi }) {\bigr| }_{{\vmath }^{\prime }} 
 \le {L}_{r} \norm{\phi - \tilde{\phi }}{\vmath }{} ,
   \qquad \phi , \tilde{\phi } \in \vmath  , \quad \norm{\phi }{\vmath }{}, 
  \norm{\tilde{\phi }}{\vmath }{} \le r  .
\end{equation*}
\end{description}
\end{lemma}

\begin{proof} 
\bf Ad. (1) \rm
 Let $\phi =(u,\vartheta ) \in \vmath  $ and $\psi =(v,\theta ) \in \vmath  $.
By inequalities \eqref{E:estimate_B_0} and (\ref{E:estimate_B2})
we obtain the following estimates
 \begin{align*}
&|\bcal (\phi ,\psi ) {|}_{{\vmath }^{\prime }}^{2}
    =  \bigl| \bigl(  {B}_{0}(u,v), {B}_{2}(u,\theta ) \bigr) {\bigr| }_{{\vmath }^{\prime }}^{2}
     = \bigl| {B}_{0}(u,v) {\bigr| }_{{V}_{0}^{\prime }}^{2}
       + \bigl|  {B}_{2}(u,\theta )  {\bigr| }_{{V}_{2}^{\prime }}^{2} \\
   &\le {c}^{2} \norm{u}{{V}_{0}}{2} \norm{v}{{V}_{0}}{2}
       + {c}^{2} \norm{u}{{V}_{0}}{2} \norm{\theta }{{V}_{2}}{2}
     = {c}^{2} \norm{u}{{V}_{0}}{2}
        \bigl( \norm{v}{{V}_{0}}{2}+ \norm{\theta }{{V}_{2}}{2}\bigr)  
   \le {c}_{1}\norm{\phi }{\vmath }{2} \norm{\psi }{\vmath }{2},
 \end{align*}
where ${c}_{1} >0$ is a certain constant. This completes the proof of inequality (\ref{E:estimate_B}).

\bigskip  \noindent
\bf Ad. (2) \rm
Let $\phi =(u,\vartheta ) \in \hmath  $ and $\psi =(v,\theta ) \in \hmath  $. Then by inequalities
\eqref{E:estimate_B_0_ext} and (\ref{E:estimate_B2_ext}) we have the following estimates
\begin{align*}
{|\bcal (\phi ,\psi ) |}_{{\umath }_{m}^{\prime }}^{2}
   &= {|{B}_{0}(u,v) |}_{{U}_{m}^{\prime }}^{2}
       + {| {B}_{2}(u,\theta )|}_{{H}^{-m}(\ocal )}^{2} \\
   &\le {c}^{2} {|u|}_{0}^{2} {|v|}_{0}^{2}
        +{c}^{2} {|u|}_{0}^{2} {|\theta |}_{2}^{2}
    \le {c}_{2} {|\phi |}_{\hmath }^{2} {|\psi |}_{\hmath }^{2}
\end{align*}
for some constant ${c}_{2} >0$. The proof of inequality (\ref{E:estimate_B_ext}) is thus complete.

\bigskip  \noindent
\bf Ad. (3) \rm Let us fix $r>0 $ and let $\phi =(u,\vartheta ),\tilde{\phi }=(\tilde{u}, \tilde{\vartheta }) \in \vmath  $ be such that 
$ \norm{\phi }{\vmath }{}, \norm{\tilde{\phi }}{\vmath }{} \le r$. We have
\begin{align*}
\bigl| \bcal (\phi ) - \bcal (\tilde{\phi }) {\bigr| }_{{\vmath }^{\prime }}^{2}
 &= {\bigl| \bigl( {B}_{0}(u,u), {B}_{2}(u ,\vartheta ) \bigr)
 - \bigl( {B}_{0}(\tilde{u},\tilde{u}), {B}_{2}(\tilde{u },\tilde{\vartheta }) \bigr) 
\bigr|}_{{\vmath }^{\prime }}^{2}  \\
& = \bigl| {B}_{0}(u,u) -{B}_{0}(\tilde{u},\tilde{u}) {\bigr| }_{{V}_{0}^{\prime }}^{2}
 + \bigl| {B}_{2}(u,\vartheta )-{B}_{2}(\tilde{u},\tilde{\vartheta }) {\bigr| }_{{V}_{2}^{\prime }}^{2} .
\end{align*}
We will estimate each term of the right-hand side of the above equality. By inequality 
\eqref{E:estimate_B_0} we have the following  estimates
\begin{align*}
&{\bigl| {B}_{0}(u,u)-{B}_{0}(\tilde{u},\tilde{u}) \bigr| }_{{V}_{0}^{\prime }}
  \le {\bigl| {B}_{0}(u,u-\tilde{u})  \bigr| }_{{V}_{0}^{\prime }}
   + {\bigl| {B}_{0}(u-\tilde{u},\tilde{u}) \bigr| }_{{V}_{0}^{\prime }} \\
 & \le  c \norm{u}{{V}_{0}}{} \norm{u-\tilde{u}}{{V}_{0}}{}
   + c \norm{u-\tilde{u}}{{V}_{0}}{} \norm{\tilde{u}}{{V}_{0}}{} 
  \le 2r  c  \cdot \norm{u-\tilde{u}}{{V}_{0}}{}
  \le 2r  c \cdot \norm{\phi-\tilde{\phi }}{\vmath }{} .
\end{align*}
By inequality \eqref{E:estimate_B2} we obtain the following  estimates
\begin{align*}
&{\bigl| {B}_{2}(u,\vartheta ) -  {B}_{2}(\tilde{u},\tilde{\vartheta }) \bigr| }_{{V}_{2}^{\prime }}
\le  {\bigl| {B}_{2}(u,\vartheta -\tilde{\vartheta })  \bigr| }_{{V}_{2}^{\prime }}
   + {\bigl| {B}_{2}(u-\tilde{u},\tilde{\vartheta }) \bigr| }_{{V}_{2}^{\prime }} \\
&\le c \norm{u}{{V}_{1}}{} \norm{\vartheta -\tilde{\vartheta }}{{V}_{2}}{}
   + c \norm{u-\tilde{u}}{{V}_{1}}{} \norm{\tilde{\vartheta }}{{V}_{2}}{}  
\le 2r c  \cdot \norm{\phi -\tilde{\phi }}{\vmath }{} .
\end{align*}
Hence
\begin{align*}
 {\bigl| \bcal (\phi ) - \bcal (\tilde{\phi }) \bigr| }_{{\vmath }^{\prime }}^{2}
 \le 8{r}^{2}{c}^{2} \norm{\phi-\tilde{\phi }}{\vmath }{2}.
\end{align*}
Thus the Lipschitz condition holds with ${L}_{r}= 2\sqrt{2}rc$.
The proof of Lemma is thus complete.
\end{proof}

\begin{lemma} \label{L:Boussinesq_estimate_R}
Operator $\rcal $ defined by \eqref{E:Boussinesq_Rcal} has the following properties:
\begin{description}
\item[(1)] For every $\phi \in \hmath $, $\rcal  \phi \in {\vmath }^{\prime }$ and there exists a constant $c>0$ such that
\begin{equation} \label{E:estimate_R_1}
       | \rcal \phi {|}_{{\vmath }^{\prime }} \le c|\phi {|}_{\hmath }  .
\end{equation}
\item[(2)] For every $\phi \in \vmath  $:
\begin{equation} \label{E:estimate_R_2}
     \dual{\rcal \phi }{\phi }{} \ge - |\phi {|}_{\hmath }^{2}.
\end{equation}
\end{description}
\end{lemma}

\begin{proof} 
To prove the first part of the statement
let $\phi =(u,\vartheta ) \in \hmath  $ and $\psi =(v,\theta ) \in \vmath $. Since
\[
{|\rcal \phi |}^{2} = {|(-\vartheta {e}_{d}, - {u}_{d})|}^{2} = {\vartheta }^{2}+{u}_{d}^{2}
\le {|\phi |}^{2} ,
\]
we have the following estimates
\begin{align*}
 & \bigl| \int_{\ocal } (\rcal  \phi ) \cdot \psi \, dx  \bigr|
  \le \int_{\ocal } |\rcal \phi | \, |\psi | \, dx
    \le \int_{\ocal } {|\phi |}^{} \, |\psi | \, dx \\
   & \le \Bigl( \int_{\ocal } {|\phi |}^{2} \, dx {\Bigr) }^{\frac{1}{2}}
    \Bigl( \int_{\ocal } {|\psi |}^{2} \, dx {\Bigr) }^{\frac{1}{2}}  
    ={|\phi |}_{\hmath }^{} {|\psi |}_{\hmath }^{} \le c {|\phi |}_{\hmath }^{}  \norm{\psi }{\vmath }{},
\end{align*}
where $c>0$ is a certain constant. Thus $\rcal  \phi \in {\vmath }^{\prime }$ and  inequality \eqref{E:estimate_R_1} holds.
Let us move to the second part of the statement. Let $\phi =(u,\vartheta ) \in \vmath  $.
Since
$
 (\rcal  \phi )\cdot \phi  = (-\vartheta {e}_{d}, - {u}_{d}) \cdot (u,\vartheta ) = -2 {u}_{d} \vartheta
$
and $2 {u}_{d}\vartheta   \le {|\phi |}^{2}$, we infer that
\begin{align*}
 \dual{\rcal  \phi }{\phi }{} = \int_{\ocal } (\rcal \phi ) \cdot \phi \, dx
    \ge -\int_{\ocal } {|\phi |}^{2} \, dx =- {|\phi |}_{\hmath }^{2}.
\end{align*}
This completes the proof of inequality \eqref{E:estimate_R_2}
and the proof of the lemma.
\end{proof}


\bigskip  \noindent
\bf Solution of the Boussinesq equations.  \rm 
We apply the abstract framework with the spaces
$\hmath $, $\vmath  $ and 
${\vmath }_{\ast }:= {\umath }_{m}$ with $m> \frac{d}{2}+1$,
defined by \eqref{E:Boussinesq_H_V} and \eqref{E:Boussinesq_U_m}, respectively, and the  maps
$\acal $, $\bcal $ and $\rcal $  defined by \eqref{E:Boussinesq_Acal},  
 \eqref{E:Boussinesq_Bcal} and \eqref{E:Boussinesq_Rcal}, respectively.
By  Lemma \ref{L:Boussinesq_estimate_B}, the map ${\bcal }_{}$ satisfies  conditions (B.1)-(B.4).
By Lemma \ref{L:B_conv_aux}, the map $\bcal $ satisfies assumption (B.5).
By  Lemma \ref{L:Boussinesq_estimate_R}, the operator $\rcal $ satisfies condition (R.1).

\bigskip
\begin{definition}  \rm  \label{D:Boussinesq}
 \bf A martingale solution \rm of the problem 
\eqref{E:Boussinesq}-\eqref{E:Boussinesq_boundary_cond} 
is a system $(\bar{\mathfrak{A}} , \bar{\eta }, \bar{W}, \bar{\phi })$, where
$\bigl( \bar{\mathfrak{A}}, \bar{\eta }, \bar{W}\bigr) $ is as in Definition \ref{D:solution}
and 
$\bar{\phi }: [0,T] \times \bar{\Omega }\to {\hmath }_{}$ is a predictable process with $\bar{\p } $ - a.e. paths
\begin{equation*}
  \bar{\phi }(\cdot , \omega ) \in \dcal \bigl( [0,T], {{\hmath }_{}}_{w} \bigr) 
   \cap {L}^{2}(0,T;{\vmath }_{} )
\end{equation*}
such that  for all $ t \in [0,T] $ and all $ \psi \in \vmath  $ the following identity holds 
$\bar{\p }$-a.s.
\begin{align*}
& \ilsk{\bar{\phi } (t)}{\psi }{} +  \int_{0}^{t} \dual{ \acal \bar{\phi }(s)}{\psi }{} \, ds
  +\int_{0}^{t} \dual{{\bcal }_{}(\bar{\phi } (s),\bar{\phi }(s))}{\psi }{} \, ds 
  +\int_{0}^{t} \dual{{\rcal }_{}(\bar{\phi } (s))}{\psi }{} \, ds \\
 & = \ilsk{{\phi }_{0}}{\psi }{} +  \int_{0}^{t} \dual{f(s)}{\psi }{} \, ds
 + \Dual{\int_{0}^{t} {G}_{}(\bar{\phi } (s))\, d\bar{W}(s)}{\psi }{} \\
& \quad + \int_{0}^{t} \int_{{Y}_{0}} \ilsk{{F}_{}(s,\bar{\phi } ({s}^{-});y)}{\psi }{\hmath } \,  \tilde{\bar{\eta }}(ds,dy)
 + \int_{0}^{t} \int_{Y \setminus {Y}_{0}} \ilsk{{F}_{}(s,\bar{\phi } ({s}^{-});y)}{\psi }{\hmath } \,  \bar{\eta }(ds,dy).
\end{align*}
\end{definition}
\noindent
Applying Theorem \ref{T:existence}, we obtain the following result about 
the existence of the martingale solution of the Boussinesq problem.

\begin{cor} \label{C:existence_Boussinesq} For every
${\phi }_{0}\in {\hmath }_{}$, ${f}_{}\in {L}^{2}(0,T; {\vmath }^{\prime })$,
 ${G}_{}: {\vmath }_{} \to \lhs (Y,{\hmath }_{})$ satisfying conditions (G.1)-(G.3)
 and $F:[0,T]\times \hmath  \times Y \to \hmath  $ satisfying conditions (F.1)-(F.3)
 there exists a martingale solution $(\bar{\mathfrak{A}} , \bar{\eta }, \bar{W}, \bar{\phi })$,
 where $\bar{\phi }=(\bar{u},\bar{\vartheta })$,
 of  problem 
\eqref{E:Boussinesq}-\eqref{E:Boussinesq_boundary_cond} such that
\[
  \bar{\e } \Bigl[ 
\sup_{t \in [0,T]} ({|\bar{u}(t)|}_{{H}_{0}}^{2} + {| \bar{\vartheta }(t)|}_{{H}_{2}}^{2})
 + \int_{0}^{T} (\norm{\bar{u}(t)}{{V}_{0}}{2} +\norm{\bar{\vartheta }(t)}{{V}_{2}}{2}) \, dt
  \bigr]  < \infty .
\]
\end{cor}


\section{Appendix: Proof of Lemma \ref{L:Dubinsky_cadlag_unbound}}

\subsection{The space of c\`{a}dl\`{a}g functions}

\noindent
Let $(\xmath ,\varrho )$ be a separable and complete metric space.
Let $\dmath ([0,T];\xmath )$ denote the space of all $\xmath $-valued \it c\`{a}dl\`{a}g \rm functions defined on $[0,T]$, i.e. the functions  which are right continuous and have left limits at every $t\in [0,T]$ . 
We consider the space $\dmath ([0,T];\xmath )$  endowed with the Skorokhod topology.
This topology is completely metrizable, see \cite{Joffe_Metivier_86}.

\bigskip \noindent
Let us recall the notion of a \it modulus \rm of the function. It plays analogous role in the space $\dmath ([0,T];\xmath )$ as \it the modulus of continuity \rm in the space of continuous functions $\cmath ([0,T];\xmath )$.

\begin{definition}  \rm (see \cite{Metivier_88})
Let $u \in \dmath ([0,T];\xmath )$ and let $ \delta >0 $ be given. A \bf modulus \rm of $u$ is defined by 
\begin{equation}  \label{E:modulus_cadlag}
  {w}_{[0,T],\xmath }(u,\delta ) : = \inf_{{\Pi }_{\delta }} \, \max_{{t}_{i} \in \bar{\omega }} \,
  \sup_{{t}_{i} \le s <t < {t}_{i+1} \le T } \varrho  \bigl( u(t), u(s) \bigr) ,
\end{equation}
where ${\Pi }_{\delta }$ is the set of all increasing sequences 
$
   \bar{\omega } = \{ 0= {t}_{0} < {t}_{1} < ... < {t}_{n} =T  \}
$
with the following property
$
      {t}_{i+1} - {t}_{i}   \ge \delta , \ i=0,1,...,n-1.
$
If no confusion seems likely, we will denote the modulus by ${w}_{[0,T]}(u,\delta )$.
\end{definition}

\bigskip  \noindent
Let us also recall the basic criterion for relative compactness of a subset of the space $\dmath ([0,T];\xmath )$,
see \cite{Joffe_Metivier_86},\cite[Chapter II]{Metivier_88}. 

\begin{theorem} \label{T:cadlag_compactness} \it 
A set $A \subset \dmath ([0,T];\xmath ) $ has compact closure iff it satisfies the following two conditions:
\begin{itemize}
\item[(a) ] there exists a dense subset $J \subset [0,T]$ such that 
for every $ t \in J  $  the set $\{ u(t), \, \, u \in A  \} $ has compact closure in $\xmath $.
\item[(b) ] $ \lim_{\delta \to 0} \, \sup_{u \in A}  \, {w}_{[0,T]}(u,\delta ) =0 $.
\end{itemize}
\end{theorem}

\subsection{Proof of Lemma \ref{L:Dubinsky_cadlag_unbound}}

\noindent
Let us consider the ball
\[
    \ball := \{ x \in \hmath  : \, \, \, {|x|}_{\hmath } \le r \} .
\]
Let ${\ball }_{w}$ denote the ball $\ball $ endowed with the weak topology.
It is well-known that the ${\ball }_{w}$ is metrizable, see \cite{Brezis}. 
Let us denote by 
 $\dmath ([0,T]; {\ball }_{w}) $    the space of weakly \it c\`{a}dl\`{a}g \rm  functions  
$ u : [0,T] \to \hmath  $ and such that 
\begin{equation}  
     \sup_{t \in [0,T]} {|u(t)|}_{\hmath } \le r  .  \label{E:D([0,T];B_w)} 
\end{equation}          
The space $\dmath ([0,T]; {\ball }_{w})$ is  completely metrizable as well.

\bigskip  \noindent
The following lemma says that any  sequence $(\un ) \subset {L}^{\infty } (0,T;\hmath )$  convergent in 
$\dmath ([0,T];{\umath }^{\prime })$ is also convergent in the space $\dmath ([0,T];{\ball }_{w})$. 

\bigskip
\begin{lemma} (see Lemma 2 in \cite{Motyl_NS_Levy_2012}) \label{L:D(0,T,{hmath }_{w})_conv}
Let ${u}_{n}:[0,T] \to \hmath  $, $n \in \nat $, be functions such that
\begin{itemize}
\item[(i)] $\sup_{n \in \nat } \sup_{s \in [0,T]} {|\un (s)|}_{\hmath } \le r  $,
\item[(ii)] $\un \to u $ in $\dmath ([0,T];{\umath }^{\prime })$.
\end{itemize}
Then  $u, \un \in \dmath ([0,T];{\ball }_{w}) $ and $\un \to u$ in $\dmath ([0,T];{\ball }_{w})$ as $n \to \infty $.
\end{lemma}

\bigskip  \noindent
\bf Proof of Lemma \ref{L:Dubinsky_cadlag_unbound}. \rm 
Let us note that since $\kcal \subset {L}^{\infty }(0,T;\hmath ) \cap \dmath ([0,T];{\umath }^{\prime })$, by Lemma \ref{L:D(0,T,{hmath }_{w})_conv} $\kcal \subset \dmath ([0,T];{\hmath }_{w})$. Now, it is easy to see that $\kcal \subset {\zcal }_{}$.
We can  assume that $\kcal $ is a closed subset of ${\zcal }_{}$. Because of the assumption (b), the weak topology in  ${L}_{w}^{2}(0,T;\vmath )$ induced on ${\zcal }_{2}$ is metrizable. 
By assumption (a), it is sufficient to consider the metric  subspace $ \dmath ([0,T]; {\ball }_{w}) \subset \dmath ([0,T],{\hmath }_{w}) $ 
with $r:=\sup_{u\in \kcal } $ $\sup_{s \in[0,T]} $ $ {|u(s)|}_{\hmath }$.
Thus compactness of a subset of ${\zcal }_{}$ is equivalent to its sequential compactness. Let $(\un )$ be a sequence in $\kcal $. By the Banach-Alaoglu Theorem condition (b) yields that the set $\kcal $ is relatively compact in $ {L}_{w}^{2}(0,T;\vmath ) $. 

\bigskip  \noindent
Using the compactness criterion in the space of c\`{a}dl\`{a}g functions contained in Theorem 
\ref{T:cadlag_compactness}, we will prove that $(\un )$ is compact in  $ \dmath ([0,T]; {\umath }^{\prime })$.
Indeed, by (a) for every $t\in [0,T]$ the set $\{ \un (t), n\in \nat  \} $ is bounded in $\hmath $. Since the embedding $\hmath  \subset {\umath }^{\prime }$ is compact, the set $\{ \un (t), n\in \nat  \} $ is compact in ${\umath }^{\prime }$.
This together with condition (c) implies compactness of the sequence $(\un )$ in the space 
$\dmath ([0,T];{\umath }^{\prime })$.

\bigskip  \noindent
Therefore there exists a subsequence of $(\un )$, still denoted by $(\un )$, such that 
\[
  \un \to u \quad \mbox{in} \quad {L}_{w}^{2}(0,T;\vmath ) \cap \dmath ([0,T]; {\umath }^{\prime }) 
\quad \mbox{as } \quad n \to \infty  .
\] 
Since $\un \to u$ in $\dmath ([0,T];{\umath }^{\prime })$, by assumption (a) and Lemma \ref{L:D(0,T,{hmath }_{w})_conv}, we infer that $\un \to u$ in $\dmath ([0,T],{\hmath }_{w})$. 
We will prove that there exists another subsequence  of $ (\un )$ such that 
\[
   \un \to u \quad \mbox{ in } \quad  {L}^{2}(0,T;{L}^{2}_{loc}(\ocal )).
\]
To this end let us fix $R \in \nat $. Let us consider the following spaces of restrictions to ${\ocal }_{R}$ of functions defined on $\ocal $
\[
 \hmath ({\ocal }_{R}) :=\{ {u}_{|{{\ocal }_{R}}} , \ u\in \hmath  \} \subset {L}^{2}({\ocal }_{R};\rtd )
\quad \mbox{and} \quad 
 \vmath ({\ocal }_{R}) :=\{ {u}_{|{{\ocal }_{R}}} , \ u\in \vmath  \} \subset {H}^{1}({\ocal }_{R};\rtd ).
\] 
Using again  Theorem \ref{T:cadlag_compactness}, we infer that  the sequence  $
 ({\un }_{|{\ocal }_{R}} )$ is compact in  $ \dmath ([0,T]; {\umath }^{\prime })$.
Thus there exists a subsequence $(\unk ) \subset (\un )$ such that 
$ {\unk }_{|{\ocal }_{R}} \to {u}_{|{\ocal }_{R}} $ in $\dcal ([0,T];{\umath }^{\prime })  $
as $k \to \infty $. 
Since ${\ocal }_{R}$ is bounded and the norms in $\hmath $ and  $\vmath $ are equivalent to the norms in ${L}^{2}(\ocal ;\rtd )$ and ${H}^{1}(\ocal ;\rtd )$, respectively, we infer that the embedding 
$\vmath ({\ocal }_{R}) \subset \hmath ({\ocal }_{R})$ is compact.
Moreover, the embeddings $\hmath ({\ocal }_{R}) \hookrightarrow {\hmath }^{\prime }\rightarrow {\umath }^{\prime }$ are continuous.
Hence, by the Lions Lemma \cite{Lions_69}, for every $\eps >0 $ there exists a constant $ {C}_{\eps , R}>0 $ such that
\[
    {|u|}_{{L}^{2}({\ocal }_{R})}^{2} 
 \le \eps \norm{u}{\vmath }{2} + {C}_{\eps ,R} {|{u}_{|{\ocal }_{R}}|}_{{\umath }^{\prime }}^{2} ,
 \qquad u \in \vmath .
\]
In particular, for almost all $s \in [0,T]$
\[
  {|\unk (s) - u(s )|}_{{L}^{2}({\ocal }_{R})}^{2} 
 \le \eps \norm{\unk (s) - u(s)}{\vmath }{2}
 + {C}_{\eps ,R} {|{\unk }_{|{\ocal }_{R}} (s) - {u}_{|{\ocal }_{R}}(s)|}_{{\umath }^{\prime }}^{2} ,
\quad k \in \nat ,
\]
and hence
\[
  {p}_{T,R}^{2} (\unk  - u ) =
  {\| \unk  - u \| }_{{L}^{2}(0,T;{L}^{2}({\ocal }_{R}))}^{2}
  \le \eps \norm{\unk  - u }{{L}^{2}(0,T;\vmath )}{2}
 + {C}_{\eps ,R} {\| {\unk }_{|{\ocal }_{R}} - {u}_{|{\ocal }_{R}} \| }_{{L}^{2}(0,T;{\umath }^{\prime })}^{2}  .
\]
Passing to the upper limit as $k \to \infty $ in the above inequality and using the estimate
\[
 \norm{\unk  - u}{{L}^{2}(0,T;\vmath )}{2}
   \le 2 \bigl( \norm{\unk   }{{L}^{2}(0,T;\vmath )}{2}
    + \norm{ u }{{L}^{2}(0,T;\vmath )}{2} \bigr)
 \le 4 {c}_{2} ,
\]
where ${c}_{2}= \sup_{u \in \kcal } \norm{u}{{L}^{2}(0,T;\vmath )}{2}$, we infer that
$
  \limsup_{\kinf }  {p}_{T,R}^{2} (\unk  - u ) \le 4 {c}_{2} \eps  .
$
By the arbitrariness of $\eps $,
\[
  \lim_{\kinf }  {p}_{T,R} (\unk  - u ) =0 .
\]
Using the diagonal method we can choose a subsequence of $(\un )$ convergent in ${L}^{2}(0,T;{L}^{2}_{loc}(\ocal ))$.
The proof of  Lemma \ref{L:Dubinsky_cadlag_unbound} is thus complete. 
\qed


\bigskip
\section{Appendix B: Time homogeneous Poisson random measure}

\noindent
We follow the approach due to Brze\'{z}niak and Hausenblas \cite{Brzezniak_Hausenblas_2009}, 
\cite{Brzezniak_Hausenblas_2010}, see also \cite{Ikeda_Watanabe_81}, \cite{Applebaum_2009}  and \cite{Peszat_Zabczyk_2007}. 
Let us denote
 $\nat :=\{ 0,1,2,... \} , \, \,  \overline{\nat }:= \nat \cup \{ \infty  \} , \, \, 
 {\rzecz }_{+}:=[0,\infty ) $.
Let  $(S, \scal )$ be a measurable space and let 
 ${M}_{\overline{\nat }}(S) $ be the set of all $\overline{\nat }$ valued measures on $(S, \scal )$.
On the set  ${M}_{\overline{\nat }}(S)$ we consider the $\sigma $-field
 $  {\mcal }_{\overline{\nat }}(S) $ defined as
the smallest $\sigma $-field  such that for all $ B \in \scal $: the map
$
      {i}_{B} : {M}_{\overline{\nat }}(S) \ni \mu \mapsto \mu (B) \in \overline{\nat }
$ 
is measurable.

\bigskip  \noindent
Let $(\Omega , \fcal ,\p  )$ be a complete probability space with filtration $\mathbb{F}:=({\fcal }_{t}{)}_{t\ge 0}$ satisfying the usual hypotheses, see \cite{Metivier_82}.
\bigskip 
\begin{definition} \rm (see \cite{Applebaum_2009} and Appendix C in \cite{Brzezniak_Hausenblas_2009}). Let $(Y, \ycal )$ be a measurable space. A
\it  time homogeneous Poisson random measure $\eta $  \rm on $(Y, \ycal )$ over $(\! \Omega ,\fcal , \fmath ,\p )$ is a measurable function 
$$
\eta : (\Omega , \fcal ) \to \bigl( {M}_{\overline{\nat }} ({\rzecz }_{+}\times Y),{\mcal }_{\overline{\nat }} ({\rzecz }_{+}\times Y) \bigr) 
$$
such that 
\begin{itemize}
\item[(i) ] for all $ B \in \bcal ({\rzecz }_{+}) \otimes \ycal $, $\eta (B):= {i}_{B} \circ \eta : \Omega \to \overline{\nat }$ is a Poisson random variable with parameter $\e [\eta (B)]$;
\item[(ii) ] $\eta $ is independently scattered, i.e. if the sets ${B}_{j}\in \bcal ({\rzecz }_{+}) \otimes \ycal $, $j=1,...,n$, are disjoint then the random variables $\eta ({B}_{j})$, $j=1,...,n$, are independent;
\item[(iii) ]  for all $ U \in \ycal $  the $\overline{\nat }$-valued process 
$\bigl( N(t,U) {\bigr) }_{t \ge 0} $ defined by 
$$
   N(t,U):= \eta ((0,t]\times U) , \qquad  t \ge 0 
$$
is $\fmath $-adapted and its increments are independent of the past, i.e. if $t>s\ge 0$, then 
$N(t,U)-N(s,U)= \eta ((s,t]\times U)$ is independent of ${\fcal }_{s}$.
\end{itemize} 
\end{definition}
\noindent
If $\eta $ is a time homogeneous Poisson random measure then the formula
\[
  \mu (A) := \e [\eta ((0,1]\times A )] , \qquad A \in \ycal 
\]
defines a measure on $(Y, \ycal )$ called an \it intensity measure \rm of $\eta $.
Moreover, for all $T<\infty $ and all $A\in \ycal $ such that $\e \bigl[ \eta ((0,T]\times A) \bigr] <\infty $, the 
$\rzecz $-valued process $\{ \tilde{N} (t,A) {\} \! }_{t \in (0,T]\! }$ defined by 
\[
   \tilde{N} (t,A) := \eta ((0,t]\times A)  - t \mu (A) , \qquad t \in (0,T],
\] 
is an integrable martingale on $(\Omega ,\fcal , \fmath ,\p )$.
The random measure $l \otimes \mu $ on $\bcal ({\rzecz }_{+}) $ $\otimes $ $\ycal $, where $l$ stands for the Lebesgue measure, is called an \it compensator \rm of $\eta $ and 
the difference between a time homogeneous Poisson random measure $\eta $ and its compensator, i.e. 
\[
    \tilde{\eta } := \eta - l \otimes \mu   ,
\]
is called a \it compensated time homogeneous Poisson random measure.\rm 

\bigskip  \noindent
Let us also recall basic properties of the stochastic integral with respect to  $\tilde{\eta }$,
see \cite{Brzezniak_Hausenblas_2009}, \cite{Ikeda_Watanabe_81} and \cite{Peszat_Zabczyk_2007} for details.
Let $E $ be a separable Hilbert space and let $\pcal $ be a predictable $\sigma $-field on $[0,T] \times \Omega $.
Let ${\mathfrak{L}}^{2}_{\mu ,T} (\pcal \otimes \ycal ,l \otimes \p \otimes \mu ;E)$ be a space of all $E$-valued, $\pcal \otimes \ycal $-measurable processes such that
\[
  \e \Bigl[ \int_{0}^{T}\int_{Y} \norm{\xi (s, \cdot ,y)}{E}{2} \,  ds d\mu (y) \Bigr] < \infty .
\] 
If $\xi \in {\mathfrak{L}}^{2}_{\mu ,T} (\pcal \otimes \ycal ,l \otimes \p \otimes \mu ;E )$
then the integral process $\int_{0}^{t} \int_{Y} \xi (s, \cdot ,y) \, \tilde{\eta }(ds,dy)$, 
$t\in [0,T]$, is a \it c\`{a}dl\`{a}g \rm ${L}^{2}$-integrable martingale. Moreover, the following isometry formula holds
\begin{equation} \label{E:isometry}
  \e \biggl[ \Norm{\int_{0}^{t} \int_{Y} \xi (s, \cdot ,y)  \tilde{\eta }(ds,dy) }{E}{2} \biggr]
  =\e \Bigl[ \int_{0}^{t}\int_{Y} \norm{\xi (s, \cdot ,y)}{E}{2}   ds d\mu (y) \Bigr] ,
   \, \,  t \in [0,T].
\end{equation} 

\bigskip 

\section{Appendix C: A version of the Skorokhod Embedding Theorem}

\noindent
In the proof of Theorem \ref{T:existence}  we use the following  version of the Skorokhod Embedding Theorem following from the version due to Jakubowski \cite{{Jakubowski_1998}} and the version  due to Brze\'{z}niak and Hausenblas  \cite[Theorem E.1]{Brzezniak_Hausenblas_2010}. 

\bigskip 
\begin{cor} \label{C:Skorokhod_J,B,H} \rm (Corollary 2 in \cite{Motyl_NS_Levy_2012}) \it 
 Let ${\xcal }_{1}$ be a separable complete metric space and let ${\xcal }_{2}$ be a topological space such that
there exists a sequence  $\{ {f}_{\iota }{ \} }_{\iota \in \nat } $ of continuous functions ${f}_{\iota }:{\xcal }_{2} \to \rzecz $  separating points of ${\xcal }_{2}$.
Let $\xcal := {\xcal }_{1}\times {\xcal }_{2}$ with the Tykhonoff topology induced by the projections 
$$
  {\pi }_{i}: {\xcal }_{1}\times {\xcal }_{2} \to {\xcal }_{i} , \qquad  i=1,2.
$$
Let $(\Omega ,\fcal ,\p )$ be a probability space and let 
${\chi }_{n}:\Omega \to {\xcal }_{1}\times {\xcal }_{2}$, $n\in \nat $, be a family of random variables such that the sequence $\{ \lcal aw({\chi }_{n}), n \in \nat \} $ is tight on ${\xcal }_{1}\times {\xcal }_{2}$.
Finally let us assume that there exists a random variable $\rho :\Omega \to {\xcal }_{1}$ such that 
$\lcal aw({\pi }_{1}\circ {\chi }_{n}) = \lcal aw(\rho )$ for  all $ n \in \nat $.

\bigskip \noindent
Then there exists a subsequence $\bigl( {\chi }_{{n}_{k}} {\bigr) }_{k \in \nat } $,
a probability space $(\bar{\Omega }, \bar{\fcal }, \bar{\p })$, a family of ${\xcal }_{1}\times {\xcal }_{2}$-valued random variables $\{ {\bar{\chi }}_{k}, \, k \in \nat  \} $ on 
$(\bar{\Omega }, \bar{\fcal }, \bar{\p })$ and a random variable ${\chi }_{*}: \bar{\Omega } \to 
{\xcal }_{1}\times {\xcal }_{2}$ such that 
\begin{itemize}
\item[(i) ] $\lcal aw ({\bar{\chi }}_{k}) = \lcal aw ({\chi }_{{n}_{k}})$ for all $ k \in \nat $;
\item[(ii) ] ${\bar{\chi }}_{k} \to {\chi }_{*}$ in ${\xcal }_{1}\times {\xcal }_{2}$ a.s. as $k \to \infty $;
\item[(iii) ] ${\pi }_{1} \circ {\bar{\chi }}_{k} (\bar{\omega }) = {\pi }_{1} \circ {\chi }_{*}(\bar{\omega })$ for all $\bar{\omega } \in \bar{\Omega }$.
\end{itemize}
\end{cor}

\bigskip 
\noindent
To prove Theorem \ref{T:existence} we use Corollary \ref{C:Skorokhod_J,B,H} for the space 
\[
 {\xcal }_{2}:={\zcal }_{}:=  {L}_{w}^{2}(0,T;\vmath ) \cap {L}^{2}(0,T;{L}^{2}_{loc}({\ocal }_{R}))
 \cap \dmath ([0,T];{\umath }^{\prime })
  \cap \dmath ([0,T];{\hmath }_{w}).
\]
We recall now  the result about the existence of the  countable family of real valued continuous mappings defined on $\zcal $ and separating points of this space.

\begin{remark} \ \label{R:separating_maps} \rm  (see Remark 2 in \cite{Motyl_NS_Levy_2012})
\begin{itemize}
\item[(1)] Since  ${L}^{2}(0,T;{L}^{2}_{loc}({\ocal }_{R}))$ and $\dmath ([0,T];{\umath }^{\prime })$ are separable and completely metrizable spaces, we infer that on each of these spaces there exists a  countable family of continuous real valued mappings separating points, see \cite{Badrikian_70}, expos\'{e} 8.
\item[(2)] For the space ${L}^{2}_{w}(0,T;\vmath )$ it is sufficient to put
\[
    {f}_{m}(u):= \int_{0}^{T} \dirilsk{u(t)}{{v}_{n}(t)}{} \, dt \in \rzecz ,
 \qquad u \in {L}^{2}_{}(0,T;\vmath ),\quad m \in \nat ,
\]
where $\{ {v}_{m}, m \in \nat  \} $ is a dense subset of ${L}^{2}(0,T;\vmath )$.
Then $({f}_{m}{)}_{m \in \nat }$ is a sequence of continuous real valued mappings separating points of the space ${L}^{2}_{w}(0,T;\vmath )$.
\item[(3)] Let ${\hmath }_{0} \subset \hmath $ be a countable and dense subset of $\hmath $. Then by 
(\ref{E:D([0,T];H_w)_cadlag}) for each $h \in {\hmath }_{0}$ the mapping
\[
   \dmath ([0,T];{\hmath }_{w}) \ni u \mapsto \ilsk{u(\cdot )}{h}{\hmath } \in \dmath ([0,T];\rzecz ) 
\]
is continuous.  Since $\dmath ([0,T];\rzecz )$ is a separable complete metric space, there exists a sequence $({g}_{l}{)}_{l \in \nat }$ of real valued continuous functions defined on $\dmath ([0,T];\rzecz )$ separating points of this space. Then the mappings ${f}_{h,l}$, $h \in {\hmath }_{0}$, $l \in \nat $ defined by 
\[
   {f}_{h,l}(u):= {g}_{l} \bigl( \ilsk{u(\cdot )}{h}{\hmath } \bigr) , \qquad 
    u \in \dmath ([0,T];{\hmath }_{w}),
\] 
form a countable family of continuous mappings on $\dmath ([0,T];{\hmath }_{w})$ separating points of this space.
\end{itemize}
\end{remark}

\section{Appendix D: Proofs of Lemmas \ref{L:Galerkin_estimates } and \ref{L:comp_Galerkin}}

\noindent
\bf Proof of Lemma \ref{L:Galerkin_estimates }. \rm 
For every $n \in \nat $ and  $R>0$ let us define
\begin{equation} \label{E:stopping_time}
  {\taun }_{}(R):= \inf \{ t \ge 0: \, \, |\un (t){|}_{H} \ge R  \} \wedge T. 
\end{equation}
Since the process $\bigl( {u}_{n}(t) {\bigr) }_{t \in [0,T]} $ is $\fmath $-adapted and right-continuous, 
$\taun (R)$ is a stopping time. Moreover, since  the process $({u}_{n})$ is c\`{a}dl\`{ag} on $[0,T]$,  the trajectories $t \mapsto {u}_{n}(t)$ are bounded on $[0,T]$, $\p $-a.s. Thus
$\taun (R)\uparrow T$, $\p $-a.s., as $R\uparrow \infty $.

\bigskip  \noindent
Assume first that $p =2$ or $p=2+\gamma $.
Using the It\^{o} formula to the function  
$\phi(x):={|x|}_{}^{p}:={|x|}_{\hmath }^{p}$, $x \in \hmath $,  we obtain for all $t \in [0,T]$
\begin{align}
 & {|  {u}_{n}  ({t\wedge \taun (R)}) | }_{\hmath }^{p} =  {|\Pn {u}_{0}|}_{\hmath }^{p} 
+ \int_{0}^{t\wedge \taun (R)} \bigl\{ p {| \un (s)| }_{\hmath }^{p-2} 
  \dual{\Phin (\un (s))}{\un (s)}{} \bigr\}  \, ds \nonumber \\
 & + {M}_{n} ({t\wedge \taun (R)}) + {I}_{n}({t\wedge \taun (R)}) +  {K}_{n}({t\wedge \taun (R)})  
+ {N}_{n} ({t\wedge \taun (R)}) + {J}_{n} ({t\wedge \taun (R)}),  \label{E:Ito_formula}
\end{align}
where
\begin{equation}  \label{E:Phi_n}
 \Phin (v):=  -\Pn {\acal }v - {\bcal }_{n} (v) -\Pn {\rcal }v
 + \Pn f   , \quad v \in {\hmath }_{n} ,
\end{equation}
and for $t \in [0,T]$
\begin{align}
{M}_{n}(t):=&  \int_{0}^{t} 
   \int_{{Y}_{0}} \big\{  {|\un ({s}^{-}) + \Pn F(s,\un ({s}^{-});y) | }_{\hmath }^{p} -
    {|\un ({s}^{-})|}_{\hmath }^{p}  \bigr\} \, \tilde{\eta}(ds,dy) \nonumber \\
   &  + \int_{0}^{t} 
   \int_{Y\setminus {Y}_{0}} \big\{ {| \un ({s}^{-}) + \Pn F(s,\un ({s}^{-});y)|}_{\hmath }^{p} -
    {| \un ({s}^{-}) |}_{\hmath }^{p}  \bigr\} \, \tilde{\eta}(ds,dy) ,  
    \label{E:M_n} \\
{I}_{n}(t):=& \int_{0}^{t} \int_{{Y}_{0}} 
   \big\{ {| \un ({s}^{})  + \Pn F(s,\un ({s}^{});y)| }_{\hmath }^{p}
   -{| \un ({s}^{}) | }_{\hmath }^{p}  \nonumber  \\
 & -p {| \un ({s}^{}) | }_{\hmath }^{p-2} 
   \ilsk{\un ({s}^{})}{\Pn F(s,\un ({s}^{});y)}{\hmath } \bigr\} \, d\mu (y) ds  ,  \label{E:I_n} \\
{K}_{n}(t):=& \int_{0}^{t} \int_{Y \setminus {Y}_{0}} 
   \big\{ {| \un ({s}^{})  + \Pn F(s,\un ({s}^{});y)| }_{\hmath }^{p}
   -{| \un ({s}^{}) | }_{\hmath }^{p}     \bigr\} \, d\mu (y) ds  ,  \label{E:K_n} \\
 {N}_{n}(t):=& \int_{0}^{t} {|\un (s)|}_{\hmath }^{p-2} \dual{\un (s)}{G(s,\un (s)) \, d W(s) }{} ,   \label{E:N_n} \\   
 {J}_{n}(t) :=& \frac{1}{2}  \int_{0}^{t} 
\tr \bigl[ \Pn G(s,\un (s)) \frac{{\partial }^{2} \phi }{\partial  {x}^{2}}
     { \bigl( \Pn G(s,\un (s)) \bigr) }^{\ast } \bigr]  \, ds .  \label{E:J_n}
\end{align} 
Since by \eqref{E:A_acal_rel} we have  $\dual{\acal \un }{\un }{} = \dirilsk{\un }{\un }{}$ and by 
\eqref{E:antisymmetry_B}, $\dual{{\bcal }_{n} (\un )}{\un }{}=0 $, we infer that for all $s \in [0,T] $
\begin{equation*}
   \dual{\Phin (\un (s))}{\un (s)}{} = -\norm{\un (s)}{}{2} -\dual{\rcal \un (s)}{\un (s)}{}
    + \dual{f(s)}{\un (s)}{} .
\end{equation*}
By assumptions (R.1)
 and (C.1), \eqref{E:norm_V}  and the Schwarz inequality, we obtain for every 
$\eps >0$ and  for all $s \in [0,T] $
\begin{equation*}
   \dual{\Phin (\un (s))}{\un (s)}{} \le (-1+\eps )\norm{\un (s)}{}{2}  
 + \bigl( \frac{1}{2}{|f(s)|}_{{\vmath }^{\prime }} + {c}_{3} \bigr) \, {|\un (s)|}_{\hmath }^{2}  
 + \frac{1}{ 8\eps  } {|f(s)|}_{{\vmath }^{\prime }}^{2} .
\end{equation*}
Hence
\begin{align}
 & \int_{0}^{t\wedge \taun (R)} \bigl\{ p {|\un (s)| }_{\hmath }^{p-2} 
  \dual{\Phin (\un (s))}{\un (s)}{} \bigr\}  \, ds \nonumber \\
 & \le p\, \int_{0}^{t\wedge \taun (R)} \bigl\{  {|\un (s)| }_{\hmath }^{p-2} 
  \bigl[ (-1+\eps )\norm{\un (s)}{}{2}  
 + \bigl( \frac{1}{2}{|f(s)|}_{{\vmath }^{\prime }} +{c}_{3} \bigr) \, {|\un (s)|}_{\hmath }^{2}  
 + \frac{1}{ 8\eps  } {|f(s)|}_{{\vmath }^{\prime }}^{2}  \bigr] \, ds .
 \label{E:Phin_Galerkin_apriori}
\end{align}
Again, by \eqref{E:A_acal_rel}  inequality \eqref{E:G} in assumption (G.2) can be written equivalently in the following form
$$ \label{A:G'}
 \norm{G(s,u )}{\lhs ({Y}_{W},\hmath )}{2}
  \le (2-a ) \norm{u}{}{2}+{\lambda }_{}{|u|}_{\hmath }^{2}+ \kappa  , \qquad u \in \vmath . 
$$
Hence
\begin{equation}
  {J}_{n}(t\wedge \taun (R))  
  \le  \frac{p(p-1)}{2} \int_{0}^{t\wedge \taun (R)}  {|\un (s)|}_{\hmath }^{p-2} 
   \bigl[  (2-a ) \norm{\un (s)}{}{2}+{\lambda }_{}{|\un (s)|}_{\hmath }^{2}+ \kappa \bigr] \, ds  . 
 \label{E:E_J_n(t)}  
\end{equation}  
 
\bigskip  \noindent
From the Taylor formula, it follows that for every $p\ge 2 $ there exists a positive constant ${c}_{p}>0$ such that for all $x,h \in \hmath $ the following inequality holds
\begin{align}
   &\bigl| {|x+h|}_{\hmath }^{p} -{|x|}_{\hmath }^{p}- p {|x|}_{\hmath }^{p-2} \ilsk{x}{h}{\hmath } \bigr| 
  \le {c}_{p} ( {|x|}_{\hmath }^{p-2} +  {|h|}_{\hmath }^{p-2} ) \, {|h|}_{\hmath }^{2} .  
 \label{E:Taylor_I}
\end{align}
By \eqref{E:Taylor_I}
and \eqref{E:F_linear_growth} we obtain the following inequalities
\begin{align*}
| {I}_{n}(t)| 
 \le & {c}_{p}\int_{0}^{t} \int_{{Y}} {| \Pn F(s,\un ({s}^{});y) | }_{\hmath }^{2}
 \bigl\{ {| \un ({s}^{})| }_{\hmath }^{p-2} + {| \Pn F(s,\un ({s}^{});y) | }_{\hmath }^{p-2} \bigr\}
  \mu (dy) ds \nonumber \\
 \le &  {c}_{p}\int_{0}^{t} \bigl\{ 
{C}_{2} {| \un ({s}^{}) |}_{\hmath }^{p-2} 
  \bigl( 1+ {| \un ({s}^{}) | }_{\hmath }^{2}\bigr) 
 + {C}_{p}\bigl( 1+ {| \un ({s}^{})|}_{\hmath }^{p}\bigr)   \bigr\} \, ds \nonumber \\
 \le &  {\tilde{c}}_{p} \int_{0}^{t} \bigl\{ 1+ {| \un ({s}^{})| }_{\hmath }^{p} \bigr\} \, ds
= {\tilde{c}}_{p}t + {\tilde{c}}_{p}\int_{0}^{t} {| \un ({s}^{}) | }_{\hmath }^{p} \, ds , \qquad t \in [0,T] ,
\end{align*}
where  ${\tilde{c}}_{p}>0$ is a certain constant.
Thus by the Fubini Theorem, we obtain the following inequality
\begin{equation} \label{E:E_I_n(t)}
 \e \bigl[ | {I}_{n}(t) | \bigr] \le  
 {\tilde{c}}_{p}t + {\tilde{c}}_{p} \int_{0}^{t} \e \bigl[ {|\un (s)|}_{\hmath }^{p} \bigr] \, ds , \qquad t \in [0,T]. 
\end{equation}
Let us move now to the term ${K}_{n}$ defined by \eqref{E:K_n}.
From \eqref{E:Taylor_I} we obtain the following inequalities for all 
$x,h \in \hmath $
\begin{eqnarray}
& & \bigl| {|x+h|}_{\hmath }^{p} - {|x|}_{\hmath }^{p} \bigr|  
 \le \frac{p}{2} {|x|}_{\hmath }^{p} 
 + \bigl( {c}_{p} + \frac{p}{2} \bigr) {|x|}_{\hmath }^{p-2} {|h|}_{\hmath }^{2} 
 + {c}_{p} {|h|}_{\hmath }^{p}.  \label{E:Taylor_III}
\end{eqnarray}
By \eqref{E:Taylor_III}, \eqref{E:F_linear_growth} and the fact that $\mu (Y \setminus {Y}_{0})<\infty $ we get
\begin{align*}
 |{K}_{n}(t)| \le & \int_{0}^{t}\int_{Y \setminus {Y}_{0}} 
 \Bigl\{ \frac{p}{2} {|\un (s)|}_{\hmath }^{p} 
  + \bigl( {c}_{p} +\frac{p}{2}\bigr) {|\un (s)|}_{\hmath }^{p-2} {|F(s,\un (s))|}_{\hmath }^{2}
  + {c}_{p} {|F(s,\un (s))|}_{\hmath }^{p}   \Bigr\} \, d\mu (y)ds  \nonumber \\
  \le & \int_{0}^{t} \Bigl\{  \frac{p}{2}\, \mu (Y \setminus {Y}_{0}) {|\un (s)|}_{\hmath }^{p} \, ds  
 +
{C}_{2}\bigl( {c}_{p}+\frac{p}{2}\bigr) {|\un (s)|}_{\hmath }^{p-2} \bigl( 1+{|\un (s)|}_{\hmath }^{2} \bigr)
  + {c}_{p} {C}_{p} \bigl( 1+ {|\un (s)|}_{\hmath }^{p}   \Bigr\} \, ds  \nonumber \\
\le &  {\tilde{\tilde{c}}}_{p} \int_{0}^{t} \bigl\{ 1+ {| \un ({s}^{})| }_{\hmath }^{p} \bigr\} \, ds
= {\tilde{\tilde{c}}}_{p}t + {\tilde{\tilde{c}}}_{p}\int_{0}^{t} {| \un ({s}^{}) | }_{\hmath }^{p} \, ds , \qquad t \in [0,T] ,  
\end{align*}
where ${\tilde{\tilde{c}}}_{p}$ is a positive constant.
Thus by the Fubini Theorem, we obtain the following inequality
\begin{equation} \label{E:E_K_n(t)}
 \e \bigl[ | {K}_{n}(t) | \bigr] \le  
 {\tilde{\tilde{c}}}_{p}t + {\tilde{\tilde{c}}}_{p} \int_{0}^{t} \e \bigl[ {|\un (s)|}_{\hmath }^{p} \bigr] \, ds , \qquad t \in [0,T]. 
\end{equation}
By \eqref{E:Taylor_I}, \eqref{E:F_linear_growth} and \eqref{E:stopping_time}, the process 
$\bigl( {M}_{n}(t\wedge \taun (R)) {\bigr) }_{t \in [0,T]}$
is an integrable martingale. Hence $\e [{M}_{n}(t\wedge {\tau }_{n}(R))] = 0 $ for all $t\in [0,T]$.
Similarly, by \eqref{E:G} and \eqref{E:stopping_time}, the process 
$\bigl( {N}_{n}(t\wedge \taun (R)) {\bigr) }_{t \in [0,T]}$ is an integrable martingale and thus $\e [{N}_{n} (t\wedge {\tau }_{n}(R)) ] = 0 $ for all $t\in [0,T]$.

\bigskip  \noindent
By \eqref{E:Ito_formula}, \eqref{E:Phin_Galerkin_apriori}, \eqref{E:E_J_n(t)}, 
\eqref{E:E_I_n(t)} and \eqref{E:E_K_n(t)}, we have for all $t\in [0,T]$
\begin{align} 
& \e \bigl[  {| {u}_{n}  ({t\wedge \taun (R)}) | }_{\hmath }^{p} \bigr] 
 +p \bigl[ 1 -  \eps - \frac{1}{2} (p-1)(2-a)    \bigr]
 \e \Bigl[  \int_{0}^{T\wedge \taun (R)}  {| \un (s)| }_{\hmath }^{p-2} 
  \norm{\un (s)}{}{2} \, ds \Bigr] 
 \nonumber \\
\qquad  &\le c(p) 
 + \tilde{c}(p) \int_{0}^{t\wedge \taun (R)} 
 \e \bigl[ {|\un (s)|}_{\hmath }^{p} \bigr] \, ds ,
\label{E:Ito_formula_est}
\end{align} 
where $c(p)$ and $\tilde{c}(p)$ are some positive constants.
Let us choose $\eps > 0 $ such that  $ 1 -  \eps - \frac{(p-1) (2- a )}{2}  >0 $.
Note that since by assumption (G.2) $a\in \bigl( 2-\frac{2}{3+\gamma } , 2] $,  such an $\eps $ exists.
By \eqref{E:Ito_formula_est} we have, in particular, the following inequality
\[
\e \bigl[  {| {u}_{n}  ({t\wedge \taun (R)})| }_{\hmath }^{p} \bigr]
 \le  c(p) 
 + \tilde{c}(p) 
 \int_{0}^{t\wedge \taun (R)} \e \bigl[ {| \un (s)| }_{\hmath }^{p} \bigr] \, ds .
\]
By the Gronwall Lemma we infer that for all $t \in [0,T]$: $
\e \bigl[  {|{u}_{n}  ({t\wedge \taun (R)})|}^{p} \bigr]
 \le   {\tilde{\tilde{C}}}_{p} $
for some constant ${\tilde{\tilde{C}}}_{p}$ independent of $t\in [0,T]$, $R>0$ and $n \in \nat $, i.e.
$$ \label{E:E(un(t))_R_est}
\sup_{n \ge 1} \sup_{t\in [0,T]} \e \bigl[ {|{u}_{n} (t\wedge \taun (R))| }_{\hmath }^{p} \bigr]
 \le   {\tilde{\tilde{C}}}_{p} .
$$
Hence, in particular,
$ \label{E:E(int_un(t))_R_est}
\sup_{n \ge 1} \e \bigl[ \int_{0}^{{T\wedge \taun (R)}} {|{u}_{n}(s)|}_{\hmath }^{p} \, ds \bigr]  
 \le {\tilde{C}}_{p} 
$
for some constant ${\tilde{C}}_{p}>0$.
Passing to the limit as $R \uparrow \infty $, by the Fatou Lemma we infer that
\begin{equation} \label{E:E(int_un(t))_est}
\sup_{n \ge 1} \e \Bigl[ \int_{0}^{T}  {|{u}_{n}(s)|}_{\hmath }^{p} \, ds \Bigr]  
 \le  {\tilde{C}}_{p}  .
\end{equation} 
By \eqref{E:Ito_formula_est} and \eqref{E:E(int_un(t))_est}, we infer that
$  
\sup_{n \ge 1} \e \bigl[  \int_{0}^{T\wedge \taun (R)}  {|\un (s)|}_{\hmath }^{p-2} 
  \norm{\un (s)}{}{2} \, ds \bigr]  \le {C}_{p}
$
for some positive constant ${C}_{p}$. 
Passing to the limit as $R \uparrow \infty $ and using again the Fatou Lemma we infer that
\begin{equation}  \label{E:HV_estimate}
\sup_{n \ge 1} \e \Bigl[  \int_{0}^{T}  {|\un (s)|}_{\hmath }^{p-2} \norm{\un (s)}{}{2} \, ds \Bigr] 
\le {C}_{p} .
\end{equation} 
In particular, putting $p:=2$ by \eqref{E:norm_V}, \eqref{E:HV_estimate} and \eqref{E:E(int_un(t))_est}   we obtain inequality \eqref{E:V_estimate}. 

\bigskip  \noindent
Let us move to the proof of inequality \eqref{E:H_estimate}. By the Burkholder-Davis-Gundy inequality we obtain
\begin{align} 
&\e \Bigl[ \sup_{r\in [0,t ]} |{M}_{n}(r\wedge \taun (R))| \Bigr] 
 \le {\tilde{K}}_{p} \e \Bigl[ 
 \Bigl(  \int_{0}^{t\wedge \taun (R)}
   \int_{Y} \bigl(  {| \un ({s}^{})+\Pn F(s,\un ({s}^{});y)| }_{\hmath }^{p} -
    {|\un ({s}^{})|}_{\hmath }^{p}  {\bigr) }^{2}  \mu (dy) ds {\Bigr)  }^{\frac{1}{2}} 
  \Bigr]   \label{E:BDG_R}
\end{align}
for some constant ${\tilde{K}}_{p}>0$.
Let us recall that ${M}_{n}$ is defined by \eqref{E:M_n}.
By \eqref{E:Taylor_I} and the Schwarz inequality we obtain the following inequalities for all 
$x,h \in \hmath $
\begin{align*}
&  \bigl( {|x+h|}_{\hmath }^{p} - {|x|}_{\hmath }^{p} {\bigr) }^{2} 
 \le 2\bigl\{ {p}^{2}{|x|}_{\hmath }^{2p-2} {|h|}_{\hmath }^{2} + {c}_{p}^{2} 
   \bigl( {|x|}_{\hmath }^{p-2} + {|h|}_{\hmath }^{p-2} {\bigr) }^{2}   {|h|}_{\hmath }^{4} \bigr\}  \nonumber \\
&  \le 2  {p}^{2}{|x|}_{\hmath }^{2p-2} {|h|}_{\hmath }^{2} 
 + 4{c}_{p}^{2} {|x|}_{\hmath }^{2p-4} {|h|}_{\hmath }^{4} 
+ 4{c}_{p}^{2}  {|h|}_{\hmath }^{2p}.  
\end{align*}
Hence by inequality \eqref{E:F_linear_growth} in assumption (F.2) we obtain for all $s \in [0,T] $
\begin{align} 
&\int_{Y}\bigl( {| \un ({s}^{})  + \Pn F(s,\un ({s}^{});y) | }_{\hmath }^{p} -
    {| \un ({s}^{})|}_{\hmath }^{p}  {\bigr) }^{2} \mu (dy) 
  \le   2{p}^{2} {|\un ({s}^{})|}_{\hmath }^{2p-2} 
   \int_{Y} {| F(s,\un ({s}^{});y)| }_{\hmath }^{2} \, \mu (dy) \nonumber \\ 
& \quad   + 4 {c}_{p}^{2} {| \un ({s}^{})|}_{\hmath }^{2p-4} 
   \int_{Y} {| F(s,\un ({s}^{});y) | }_{\hmath }^{4} \, \mu (dy)
 + 4 {c}_{p}^{2} \int_{Y}{| F(s,\un ({s}^{});y) | }_{\hmath }^{2p} \mu (dy) \nonumber \\
&   \le {C}_{1} + {C}_{2} {| \un ({s}^{}) |}_{\hmath }^{2p-4} 
  + {C}_{3} {| \un ({s}^{})|}_{\hmath }^{2p-2} 
  + {C}_{4} {| \un ({s}^{}) |}_{\hmath }^{2p}
\label{E:BDG_R_Taylor_II'}  
\end{align}
for some positive constants ${C}_{i}$, $i=1,...,4$.
By \eqref{E:BDG_R_Taylor_II'} and the Young inequality we infer that
\begin{align}
&  
 \Bigl( \int_{0}^{t\wedge \taun (R)} 
   \int_{Y} \bigl( {|\un ({s}^{})+\Pn F(s,\un ({s}^{});y)| }_{\hmath }^{p} -
    {|\un ({s}^{})|}_{\hmath }^{p} {\bigr) }^{2}  \, \mu (dy) ds {\Bigr) }^{\frac{1}{2}} 
   \nonumber  \\
 &  \le {\bar{c}}_{1} + {\bar{c}}_{2}  \Bigl( 
  \int_{0}^{t\wedge \taun (R)}  {|\un ({s}^{})|}_{\hmath }^{2p} \, ds
{\Bigr) }^{\frac{1}{2}} , \label{E:BDG_R_Taylor_II} 
\end{align}
where ${\bar{c}}_{1}$ and  ${\bar{c}}_{2}$ are some positive constants.
By \eqref{E:BDG_R}, \eqref{E:BDG_R_Taylor_II} and \eqref{E:E(int_un(t))_est}
we obtain the following inequalities
\begin{align} 
&\e \Bigl[ \sup_{r\in [0,t]} |{M}_{n}(r \wedge \taun (R))| \Bigr] 
 \le {\tilde{K}}_{p}{\bar{c}}_{1} + {\tilde{K}}_{p}{\bar{c}}_{2} \e \Bigl[ \Bigl( 
  \int_{0}^{t\wedge \taun (R)} {|\un ({s}^{})|}_{\hmath }^{2p} \, ds
{\Bigr) }^{\frac{1}{2}} \Bigr]  \nonumber \\
 &\le {\tilde{K}}_{p}{\bar{c}}_{1} + {\tilde{K}}_{p}{\bar{c}}_{1} \e \Bigl[ 
 \Bigl( \sup_{s\in [0,t]}  {|\un ({s\wedge \taun (R)}^{})|}_{\hmath }^{p} {\Bigr) }^{\frac{1}{2}}
 \Bigl( \int_{0}^{t\wedge \taun (R)}  {| \un ({s}^{})|}_{\hmath }^{p} \, ds 
{\Bigr) }^{\frac{1}{2}} \Bigr]   \nonumber \\
 &\le {\tilde{K}}_{p}{\bar{c}}_{1} + \frac{1}{4} \e \Bigl[ 
  \sup_{s\in [0,t]}  {|\un ({s\wedge \taun (R)}^{})|}_{\hmath }^{p} \Bigr]
 + {\tilde{K}}_{p}^{2} {\bar{c}}_{2}^{2} \e \Bigl[ \int_{0}^{t\wedge \taun (R)} {|\un ({s}^{})|}_{\hmath }^{p} \, ds   \Bigr] \nonumber \\
 & \le \frac{1}{4} \e \Bigl[ 
  \sup_{s\in [0,t]}  {\bigl| \un ({s\wedge \taun (R)}^{}) \bigr|}_{\hmath }^{p} \Bigr] + \bar{c} ,
 \label{E:BDG_R_cont.}
\end{align}
where $\bar{c} = {\tilde{K}}_{p}{\bar{c}}_{1} + {\tilde{K}}_{p}^{2} {\bar{c}}_{2}^{2} {\tilde{C}}_{p}$. (The constant ${\tilde{C}}_{p}$ is the same as in \eqref{E:E(int_un(t))_est}).

\bigskip  \noindent
Similarly, by the Burkholder-Davis-Gundy inequality  we obtain
\begin{align*}
& \e \Bigl[ \sup_{r\in [0,t ]} |{N}_{n}(r\wedge \taun (R))| \Bigr] 
  \le C \, p \cdot \e \Bigl[
 {\Bigl( \int_{0}^{t \wedge \taun (R)} \, {|\un (s )|}_{\hmath }^{2p-2} \cdot
 \norm{  G(s , \un (s ) ) }{\lhs (Y,\hmath )}{2} \, ds
  \Bigr) }^{\frac{1}{2}} \Bigr]   \\
&  \le Cp\e \Bigl[ 
 \Bigl( \sup_{s\in  [0,t]} {|\un (s\wedge {\tau }_{n}(R))|}_{\hmath }^{p} {\Bigr) }^{\frac{1}{2}}
 {\Bigl( \int_{0}^{t \wedge \taun (R)} {|\un (s)|}_{\hmath }^{p-2} 
 \norm{G(s,\un (s))}{\lhs (Y,\hmath )}{2}ds  \Bigr) \!  }^{\frac{1}{2}} \Bigr]  ,
\end{align*} 
where ${N}_{n}$ is defined by \eqref{E:N_n}.
By inequality \eqref{E:G} in assumption (G.2) and
estimates \eqref{E:HV_estimate}, \eqref{E:E(int_un(t))_est}  we have the following inequalities
\begin{align}
 &\e \Bigl[ \sup_{r\in [0,t ]} |{N}_{n}(r\wedge \taun (R))| \Bigr] \nonumber  \\
 & \le C  p \, \e \Bigl[ 
\Bigl( \sup_{ s\in  [0,t] } {|\un (s\wedge {\tau }_{n}(R))|}_{\hmath }^{p} {\Bigr) }^{\frac{1}{2}}
 \cdot \Bigl( \int_{0}^{t\wedge \taun (R)} {|\un (s)|}_{\hmath }^{p-2} 
 \bigl[ {\lambda }_{}  {|\un (s )|}_{\hmath }^{2} +\kappa +(2-a)\norm{\un (s)}{}{2} \bigr] \, ds
  {\Bigr) }^{\frac{1}{2}} \Bigr]  \nonumber  \\
 &\le \frac{1}{4} \e \bigl[ \sup_{ r\in[0,t]} {|\un (r\wedge {\tau }_{n}(R))|}_{\hmath }^{p} \bigr]  \\
  & \quad  + {C}^{2}{p}^{2}
 \e \Bigl[ \int_{0}^{t\wedge \taun (R)} \bigl[ {\lambda }_{}  {|\un (s )|}_{\hmath }^{p} 
+\kappa {|\un (s )|}_{\hmath }^{p-2} +(2-a) {|\un (s )|}_{\hmath }^{p-2} \norm{\un (s )}{}{2}\bigr] \, ds
 \Bigr]  \nonumber \\
 &\le \frac{1}{4} \e \bigl[ \sup_{r\in [0,t]} {|\un (r\wedge {\tau }_{n}(R))|}_{\hmath }^{p} \bigr]
 + \bar{\bar{c}} ,   \label{E:BDG_Nn_R}
\end{align}
where $\bar{\bar{c}} = {C}^{2}{p}^{2} [\lambda {\tilde{C}}_{p} + \kappa {\tilde{C}}_{p-2} + (2-a){C}_{2}]$.
(The constants ${\tilde{C}}_{p}, {\tilde{C}}_{p-2}$ are the same as in \eqref{E:E(int_un(t))_est} and ${C}_{2}$ is the same as in \eqref{E:HV_estimate}.)
Therefore by \eqref{E:Ito_formula} for all $t\in [0,T]$
\begin{eqnarray} 
& & {|{u}_{n}({t\wedge \taun (R)})|}_{\hmath }^{p}
 \le    c(p)
+ \tilde{c}(p) \int_{0}^{T} {|\un (s)|}_{\hmath }^{p}  \, ds \nonumber  \\
& & 
 + \sup_{r \in [0,T] } |{M}_{n}(r\wedge \taun (R))|      
 + \sup_{r \in [0,T] } |{N}_{n}(r\wedge \taun (R))|.
 \label{E:Ito_BDG_R}
\end{eqnarray} 
Since inequality \eqref{E:Ito_BDG_R} holds for all $t\in [0,T]$ and the right-hand side of \eqref{E:Ito_BDG_R} in independent of $t$, we infer that
\begin{eqnarray} 
& &\e \Bigl[ \sup_{t\in [0,T]}{|{u}_{n}(t\wedge \taun (R))|}_{\hmath }^{p} \Bigr]
 \le  c(p)
+ \tilde{c}(p) \e \Bigl[ \int_{0}^{T} {|\un (s)|}_{\hmath }^{p}  \, ds  \Bigr] \nonumber \\
& & +  \e \Bigl[ \sup_{r \in [0,T] } |{M}_{n}(r\wedge \taun (R))| \Bigr] 
 +  \e \Bigl[ \sup_{r \in [0,T] } |{N}_{n}(r\wedge \taun (R))|  \Bigr] . \nonumber \\
    \label{E:Ito_BDG_R_E}
\end{eqnarray}
Using inequalities \eqref{E:E(int_un(t))_est}, \eqref{E:BDG_R_cont.} and 
\eqref{E:BDG_Nn_R} in \eqref{E:Ito_BDG_R_E}  we infer that
$$ \label{E:H_est_R}
 \e \Bigl[ \sup_{t \in T} {|{u}_{n}(t\wedge \taun (R))|}_{\hmath }^{p} \Bigr]
 \le {C}_{1}(p) ,
$$
where ${C}_{1}(p)>0$ is a constant independent of $n \in \nat $ and $R>0$. Passing to the limit as $R \to \infty $, we obtain inequality  \eqref{E:H_estimate}. 
Thus the lemma holds for $p\in \{ 2,2+\gamma  \} $.

\bigskip  \noindent
Let now $p \in [1,2+\gamma ) \setminus \{ 2\}  $. Let us fix $n \in \nat $. Then
\[
   {|{u}_{n}(t)|}_{\hmath }^{p} 
=  \bigl( {|{u}_{n}(t)|}_{\hmath }^{2+\gamma } {\bigr) }^{\frac{p}{2+\gamma }}
   \le \bigl( \sup_{t\in [0,T]} {|{u}_{n}(t)|}_{\hmath }^{2+\gamma } {\bigr) }^{\frac{p}{2+\gamma }} ,
   \qquad t \in [0,T].
\]
Thus 
\[
  \sup_{t\in [0,T]} {|{u}_{n}(t)|}_{\hmath }^{p}
   \le \bigl( \sup_{t\in [0,T]}{|{u}_{n}(t)|}_{\hmath }^{2+\gamma } {\bigr) }^{\frac{p}{2+\gamma }}
\]  
and by the H\"{o}lder inequality and inequality \eqref{E:H_estimate} with $p:=2+\gamma $
\begin{align*}
&  \e \Bigl[ \sup_{t\in [0,T]} {|{u}_{n}(t)|}_{\hmath }^{p} \Bigr] 
 \le \e \Bigl[ \bigl( \sup_{t\in [0,T]} {|{u}_{n}(t)|}_{\hmath }^{2+\gamma } {\bigr) }^{\frac{p}{2+\gamma }} \Bigr]  
    \le \Bigl( \e \Bigl[ \sup_{t\in [0,T]}{|{u}_{n}(t)|}_{\hmath }^{2+\gamma }\Bigr] 
   {\Bigr) }^{\frac{p}{2+\gamma }} 
\le {\bigl[ {C}_{1} (2+\gamma ) \bigr] }^{\frac{p}{2+\gamma }} .
\end{align*}
Since $n \in \nat $ was chosen in an arbitray way, we infer that
\[
\sup_{n \in \nat } \e \Bigl[ \sup_{t\in [0,T]} {|{u}_{n}(t)|}_{\hmath }^{p} \Bigr] 
   \le {C}_{1}(p) ,
\]
where ${C}_{1}(p)={[ {C}_{1} (2+\gamma ) ] }^{\frac{p}{2+\gamma }}$.
The proof of Lemma \ref{L:Galerkin_estimates } is thus complete.
\qed

\bigskip \noindent
To prove Lemma \ref{L:comp_Galerkin}
we will use Corollary \ref{C:tigthness_criterion_cadlag_unbound}.
To check condition (c) in Corollary \ref{C:tigthness_criterion_cadlag_unbound} we will use the following lemma.

\bigskip  
\begin{lemma} \label{L:Aldous_criterion} \rm (Lemma 9 in \cite{Motyl_NS_Levy_2012}) \it
Let ${({X}_{n})}_{n \in \nat }$ be a sequence of ${\umath }^{\prime }$-valued random variables.
Assume that there exist constants $\alpha ,\beta >0$ and  $C >0 $ such that 
for every sequence $({{\tau}_{n} } {)}_{n \in \nat }$ of $\mathbb{F}$-stopping times with
${\tau }_{n}\le T$ and for every $n \in \nat $ and $\theta \ge 0 $ the following condition holds
\begin{equation} \label{E:Aldous_est}
 \e \bigl[ \bigl( {| {X}_{n}({\tau }_{n} +\theta )-{X}_{n} ( {\tau }_{n}  ) |}_{{\umath }^{\prime }}^{\alpha } \bigr] \le C {\theta }^{\beta } .
\end{equation}
Then the sequence $({X}_{n}{)}_{n\in \nat }$  satisfies the Aldous condition  in the space ${\umath }^{\prime }$.  \rm
\end{lemma}

\bigskip  \noindent
\bf Proof of Lemma \ref{L:comp_Galerkin}.  \rm 
The proof is essentially the same as the proof of Lemma 5 in \cite{Motyl_NS_Levy_2012}.
We provide the details because of its importance in the proof of Theorem \ref{T:existence}.  
We will apply Corollary \ref{C:tigthness_criterion_cadlag_unbound}.
By estimates  \eqref{E:H_estimate} and \eqref{E:V_estimate}, conditions (a), (b) are satisfied.
Using Lemma \ref{L:Aldous_criterion} we will  prove that the sequence ${(\un )}_{n \in \nat }$ satisfies the Aldous condition  in the space ${\umath }^{\prime }$.  
Let ${(\taun )}_{n \in \nat} $ be a sequence of stopping times such that $0 \le \taun \le T$.
By \eqref{E:Galerkin}, we have
\begin{align*}
 {u}_{n} (t) \,
 =  & \Pn {u}_{0}  - \int_{0}^{t} \Pn \acal  {u}_{n} (s) \, ds
   - \int_{0}^{t} \Pn \rcal  {u}_{n} (s) \, ds
  - \int_{0}^{t} \Bn \bigl( {u}_{n} (s) \bigr) \, ds 
  + \int_{0}^{t} \Pn f(s) \, ds  \\
 & + \int_{0}^{t} \int_{Y}\Pn F (s,{u}_{n}({s}^{-}),y) \tilde{\eta}(ds,dy) 
    + \int_{0}^{t} \int_{Y\setminus {Y}_{0}}\Pn F (s,{u}_{n}({s}^{}),y) d\mu (y)ds \nonumber \\
 & + \int_{0}^{t} \Pn G(s,\un (s)) \, dW(s) 
  =:   \sum_{i=1}^{8} \Jn{i} (t) , \qquad t \in [0,T].
\end{align*}
Let $\theta  >0 $. 
We will check that each term $\Jn{i}$, i=1,...,8, satisfies condition \eqref{E:Aldous_est}
in Lemma \ref{L:Aldous_criterion}.

\bigskip \noindent
Since by assumption (A.1) $\acal :\vmath  \to {\vmath }^{\prime }$ and ${|\acal (u)|}_{{\vmath }^{\prime }} \le \norm{u}{}{}$
and the embedding  ${\vmath }^{\prime } \hookrightarrow {\umath }^{\prime }$ is continuous, 
 using inequality \eqref{E:V_estimate} we obtain the following estimates
\begin{align*}
 &\e \bigl[ {| \Jn{2} (\taun + \theta ) - \Jn{2}(\taun ) | }_{{\umath }^{\prime }}  \bigr]
 = \e \Bigl[ {\Bigl| \int_{\taun }^{\taun + \theta } \Pn \acal {u}_{n} (s) \, ds \Bigr| }_{{\umath }^{\prime }} \Bigr]
 \le c \, \e \Bigl[  \int_{\taun }^{\taun + \theta }
\bigl|  \acal  {u}_{n} (s) {\bigr| }_{{\vmath }^{\prime }}  \, ds \Bigr]  \nonumber \\
 &
 \le c \, \e \Bigl[  \int_{\taun }^{\taun + \theta } \nonumber
 \norm{ {u}_{n} (s) }{}{}  \, ds \Bigr] 
  \le c \, \e \Bigl[ {\theta }^{\frac{1}{2}} 
 \Bigl( \int_{0 }^{T } \norm{{u}_{n} (s) }{}{2}\, ds {\Bigr) }^{\frac{1}{2}} \Bigr]
\le c \sqrt{{C}_{2}} \cdot {\theta }^{\frac{1}{2}}=: {c}_{2} \cdot {\theta }^{\frac{1}{2}}.  \label{E:Jn2}
\end{align*}
By Assumption (R.1) and estimate \eqref{E:H_estimate} we have
\begin{align*}
& \e \bigl[ {| \Jn{3} (\taun + \theta ) - \Jn{3}(\taun )| }_{{\umath }^{\prime  }}  \bigr]
 = \e \Bigl[ {\Bigl| \int_{\taun }^{\taun + \theta } \Pn \rcal \un  (s) \, ds \Bigr| }_{{\umath }^{\prime }} \Bigr]  
 \le  \, \e \Bigl[  \int_{\taun }^{\taun + \theta }
{| \rcal \un (s) |}_{{\vmath }^{\prime }}  \, ds \Bigr]  \\
& \le {c}_{} \, \e \Bigl[  \int_{\taun }^{\taun + \theta } 
 {|\un (s)|}_{\hmath }  \, ds \Bigr] 
  \le c\theta \, \e \bigl[ \sup_{s \in [0,T]} {|\un (s)|}_{\hmath } \bigr] \le c{C}_{1}(1) \, \theta 
 =: {c}_{3}\theta .  
\end{align*}
Since  $\umath \hookrightarrow {\vmath }_{\ast }$,  by  \eqref{E:estimate_B_ext}  
and \eqref{E:H_estimate} we have the following inequalities
\begin{align*}
& \e \bigl[ {|\Jn{4} (\taun + \theta )  - \Jn{4}(\taun ) |}_{{\umath }^{\prime }}  \bigr]
 = \e \Bigl[ { \Bigl| \int_{\taun }^{\taun + \theta }
  \Bn ({u}_{n} (s) ) \, ds \Bigr| }_{{\umath }^{\prime }} \Bigr] 
 \le c\e \Bigl[  \int_{\taun }^{\taun + \theta }
 |  B( {u}_{n} (s))  {| }_{{\vmath }_{\ast }^{\prime }} \, ds \Bigr]  \\
& \le c\e \Bigl[  \int_{\taun }^{\taun + \theta }
  \| B \| \cdot \bigl|  {u}_{n} (s) { \bigr| }_{\hmath }^{2}   \, ds \Bigr] 
  \le c\| B \|  \cdot  \e \bigl[ \sup_{s \in [0,T]} {|{u}_{n}(s)|}_{\hmath }^{2}\bigr] \cdot \theta
 \le c\| B \| \,  {C}_{1}(2)  \cdot \theta =: {c}_{4} \cdot \theta , 
\end{align*}
where $\| B \| $ stands for the norm of $B: \hmath  \times \hmath  \to {\vmath }_{\ast }^{\prime }$. 

\bigskip \noindent
Let us move to the term ${\Jn{5}}$. By the H\"{o}lder inequality, we have
\begin{align*}
& \e \bigl[ {|\Jn{5} (\taun + \theta ) - \Jn{5}(\taun ) |}_{{\umath }^{\prime }}  \bigr]
\le c  \e \Bigl[ {\Bigl| \int_{\taun }^{\taun + \theta } 
  \Pn f (s)  ds \Bigr| }_{{\vmath }^{\prime }} \Bigr]  
\le c  \cdot {\theta }^{\frac{1}{2}} \cdot \norm{f}{{L}^{2}(0,T;{\vmath }^{\prime })}{} =: {c}_{5} \cdot {\theta }^{\frac{1}{2}}. 
\end{align*}
Let us consider the noise term ${\Jn{6}}$.  Since  $ \hmath  \hookrightarrow {\umath }^{\prime } $, 
by \eqref{E:isometry},
condition \eqref{E:F_linear_growth} with $p=2$ in Assumption (F.2) and by \eqref{E:H_estimate}, we obtain the following inequalities
\begin{align*}
&  \e \bigl[ {|\Jn{6} (\taun + \theta ) - \Jn{6}(\taun )| }_{{\umath }^{\prime }}^{2}  \bigr] \nonumber 
  =  \e \Bigl[ \Bigl| \int_{\taun }^{\taun + \theta } \int_{Y} \Pn F(s, {u}_{n} ({s}^{-});y ) \, \tilde{\eta}(ds,dy)  {\Bigr| }_{{\umath }^{\prime }}^{2} 
\Bigr] \nonumber \\
&   \le c\e \Bigl[ \Bigl| \int_{\taun }^{\taun + \theta } \int_{Y} \Pn F(s, {u}_{n} ({s}^{-});y ) \, \tilde{\eta}(ds,dy)  {\Bigr| }_{\hmath }^{2} 
\Bigr] 
  = c\e \Bigl[  \int_{\taun }^{\taun + \theta } 
  \int_{Y} {\bigl| \Pn F(s, {u}_{n} (s);y ) \bigr| }_{\hmath }^{2} \, \mu (dy)ds   \Bigr]  \\
&\le C  \e \Bigl[  \int_{\taun }^{\taun + \theta } (1+{|{u}_{n}(s)|}_{\hmath }^{2} ) ds  \Bigr]  \nonumber 
   \le C \cdot  \theta \cdot \bigl( 1+
  \e \bigl[ \sup_{s \in [0,T]} {|{u}_{n}(s)| }_{\hmath }^{2} \bigr] \bigr)
\le C \cdot (1+ {C}_{1}(2)) \cdot \theta =: {c}_{6} \cdot \theta . 
\end{align*}
By condition \eqref{E:F_linear_growth} with $p=1$ in Assumption (F.2) and estimate \eqref{E:H_estimate} we have
\begin{align*}
& \e \bigl[ {| \Jn{7} (\taun + \theta ) - \Jn{7}(\taun )| }_{{\umath }^{\prime  }}  \bigr]
 = \e \Bigl[ {\Bigl| \int_{\taun }^{\taun + \theta } \int_{Y \setminus {Y}_{0}} \Pn F(s,\un  (s)) 
\, d \mu (y) ds \Bigr| }_{{\umath }^{\prime }} \Bigr]
 \nonumber \\
&  \le  \, \e \Bigl[  \int_{\taun }^{\taun + \theta } \int_{Y}
{| F(s, \un (s)) |}_{\hmath }  \, ds \Bigr]
 \le {C}_{1} \, \e \Bigl[  \int_{\taun }^{\taun + \theta } 
 ( 1+ {|\un (s)|}_{\hmath } ) \, ds \Bigr] \nonumber \\
&  \le {C}_{1}\theta \, \e \bigl[ 1+ \sup_{s \in [0,T]} {|\un (s)|}_{\hmath } \bigr] \le {C}_{1}(1+{C}_{1}(1)) \, \theta 
 =: {c}_{7}\theta .  
\end{align*}
Let us consider the term ${\Jn{8}}$. 
By the It\^{o} isometry,
condition \eqref{E:G*} in assumption (G.3), continuity of the embedding ${\vmath }^{\prime }\hookrightarrow {\umath }^{\prime }$ and inequality \eqref{E:H_estimate}, we have
\begin{align*}
& \e \bigl[ {| \Jn{8} (\taun +\theta ) - \Jn{8}(\taun )|}_{{\umath }^{\prime }}^{2}  \bigr] 
  = \e \Bigl[ \Bigl| \int_{\taun }^{\taun +\theta } 
    \Pn G(s,\un (s))\, dW(s) {\Bigr| }_{{\umath }^{\prime }}^{2} \Bigr]  \nonumber \\
&  \le c  \e \Bigl[ \int_{\taun }^{\taun +\theta }(1+{|\un (s)|}_{\hmath }^{2} ) ds  \Bigr] 
 \le c   \theta \bigl( 1+ \e \bigl[ \sup_{s \in [0,T]} {|\un (s)|}_{\hmath }^{2}\bigr] \bigr) 
\le c (1+ {C}_{1}(2) ) \theta =:{c}_{8}\cdot \theta  .  
\end{align*}
By Lemma \ref{L:Aldous_criterion} the sequence ${({u}_{n})}_{n\in \nat }$ satisfies the Aldous condition in the space ${\umath }^{\prime }$.
This completes the proof of Lemma \ref{L:comp_Galerkin}.  \qed


\end{document}